\documentclass{tac}

\usepackage{mathrsfs,amssymb,stmaryrd,bbm,mathtools, upgreek}
\usepackage{chngcntr}
\usepackage[shortlabels]{enumitem}
\makeatletter
\@addtoreset{equation}{section}
\makeatother
\numberwithin{equation}{subsection}

\newtheorem{theo}{Theorem}[section]
\newtheorem{lem}[theo]{Lemma}
\newtheorem{prop}[theo]{Proposition}
\newtheorem{coro}[theo]{Corollary}
\newtheorem{defi}[theo]{Definition}
\newtheorem{rem}[theo]{Remark}

\let\olddefi\defi
\renewcommand{\defi}{\olddefi\normalfont}
\let\oldrem\rem
\renewcommand{\rem}{\oldrem\upshape}

\def\mathrmdef#1{\expandafter\def\csname#1\endcsname{{\rm#1}}}
\def\mathsfdef#1{\expandafter\def\csname#1\endcsname{{\rm\sf#1}}}
\def\mathcaldef#1{\expandafter\def\csname#1\endcsname{{\mathcal#1}}}

\mathrmdef{id}\mathrmdef{op}\mathrmdef{co}\mathrmdef{coop}
\mathsfdef{Alg}  \mathsfdef{Set} \mathsfdef{Cat}
\mathsfdef{colim}\mathsfdef{Mod}\mathsfdef{Ring}
\mathrmdef{ran}\mathsfdef{lim}\mathsfdef{mnd}
\mathrmdef{lan}\mathrmdef{rlift}\mathrmdef{llift}
\def\t{\mathtt{t}}

\def\AAA{\mathbb{A}}

\def\dd{\mathbbmss{d}}
\def\DD{\mathbbm{d}}
\def\ss{\mathbbmss{s}}
\def\uu{\mathbbmss{u}}
\def\ueta{\overline{\upeta} }
\def\etta{\uprho }

\def\terminall{\iota }

\def\d{\delta }
\def\D{\partial }
\def\s{\mathfrak{s} }

\def\umm{{\underline{\mathsf{1} } }}
\def\doiss{{\underline{\mathsf{2} } }}
\def\tress{{\underline{\mathsf{3} } }}

\def\um{\mathfrak{1} }
\def\dois{\mathfrak{2} }
\def\zero{\mathfrak{0}}

\def\AAAA{\mathcal{A}}
\def\BBBB{\mathcal{B}}
\def\FFFF{\mathcal{F}}

\def\WWWW{\mathcal{W}}

\def\ella{{\ell _ \ast}}

\def\H{ \mathcal{H}}

\def\SSS{\mathbb{S}}

\def\SSSS{\mathtt{S}}

\def\RRRR{\mathcal{R}}

\usepackage{ifluatex}
\input diagxy        
\xyoption{curve}     
\xyoption{all}
\xyoption{2cell}

\begin{document}  
	
\ifluatex
\catcode`\^^J=10
\directlua{dofile "dednat6load.lua"}
\else
%
\def\diagxyto{\ifnextchar/{\toop}{\toop/>/}}
\def\to     {\rightarrow}
\def\defded#1#2{\expandafter\def\csname ded-#1\endcsname{#2}}
\def\ifdedundefined#1{\expandafter\ifx\csname ded-#1\endcsname\relax}
\def\ded#1{\ifdedundefined{#1}
    \errmessage{UNDEFINED DEDUCTION: #1}
  \else
    \csname ded-#1\endcsname
  \fi
}
\def\defdiag#1#2{\expandafter\def\csname diag-#1\endcsname{\bfig#2\efig}}
\def\defdiagprep#1#2#3{\expandafter\def\csname diag-#1\endcsname{{#2\bfig#3\efig}}}
\def\ifdiagundefined#1{\expandafter\ifx\csname diag-#1\endcsname\relax}
\def\diag#1{\ifdiagundefined{#1}
    \errmessage{UNDEFINED DIAGRAM: #1}
  \else
    \csname diag-#1\endcsname
  \fi
}
\newlength{\celllower}
\newlength{\lcelllower}
\def\cellfont{}
\def\lcellfont{}
\def\cell #1{\lower\celllower\hbox to 0pt{\hss\cellfont${#1}$\hss}}
\def\lcell#1{\lower\celllower\hbox to 0pt   {\lcellfont${#1}$\hss}}
\def\expr#1{\directlua{output(tostring(#1))}}
\def\eval#1{\directlua{#1}}
\def\pu{\directlua{pu()}}
%

\defdiag{truncatedpseudocosimplicialcategory}{   
  \morphism(1200,0)|m|/->/<-1200,0>[{\AAAA\left(\mathsf{2}\right)}`{\AAAA\left(\mathsf{1}\right)};{\AAAA\left(s^0\right)}]
  \morphism(0,0)|m|/{@{->}@/^25pt/}/<1200,0>[{\AAAA\left(\mathsf{1}\right)}`{\AAAA\left(\mathsf{2}\right)};{\AAAA\left(d^0\right)}]
  \morphism(0,0)|m|/{@{->}@/_25pt/}/<1200,0>[{\AAAA\left(\mathsf{1}\right)}`{\AAAA\left(\mathsf{2}\right)};{\AAAA\left(d^1\right)}]
  \morphism(1200,0)|m|/->/<1200,0>[{\AAAA\left(\mathsf{2}\right)}`{\AAAA\left(\mathsf{3}\right)};{\AAAA\left(d^1\right)}]
  \morphism(1200,0)|m|/{@{->}@/^25pt/}/<1200,0>[{\AAAA\left(\mathsf{2}\right)}`{\AAAA\left(\mathsf{3}\right)};{\AAAA\left(d^0\right)}]
  \morphism(1200,0)|m|/{@{->}@/_25pt/}/<1200,0>[{\AAAA\left(\mathsf{2}\right)}`{\AAAA\left(\mathsf{3}\right)};{\AAAA\left(d^2\right)}]
}
\defdiag{truncatedpseudocosimplicialcategoryinducedbyamorphisminabifibredcategory}{   
  \morphism(1200,0)/->/<-1200,0>[{\FFFF\left(w\times_qw\right)}`{\FFFF\left(w\right)};]
  \morphism(0,0)|a|/{@{->}@<7pt>}/<1200,0>[{\FFFF\left(w\right)}`{\FFFF\left(w\times_qw\right)};{\FFFF(\pi^w)}]
  \morphism(0,0)|b|/{@{->}@<-7pt>}/<1200,0>[{\FFFF\left(w\right)}`{\FFFF\left(w\times_qw\right)};{\FFFF(\pi_w)}]
  \morphism(1200,0)/->/<1200,0>[{\FFFF\left(w\times_qw\right)}`{\FFFF\left(w\times_qw\times_qw\right)};]
  \morphism(1200,0)/{@{->}@<7pt>}/<1200,0>[{\FFFF\left(w\times_qw\right)}`{\FFFF\left(w\times_qw\times_qw\right)};]
  \morphism(1200,0)/{@{->}@<-7pt>}/<1200,0>[{\FFFF\left(w\times_qw\right)}`{\FFFF\left(w\times_qw\times_qw\right)};]
}
\defdiag{twodimensionalcokerneldiagram}{   
  \morphism(1200,0)/->/<-1200,0>[{b\uparrow_pb}`{b};]
  \morphism(0,0)|a|/{@{->}@<7pt>}/<1200,0>[{b}`{b\uparrow_pb};{\delta_{p\uparrow{p}}^0}]
  \morphism(0,0)|b|/{@{->}@<-7pt>}/<1200,0>[{b}`{b\uparrow_pb};{\delta_{p\uparrow{p}}^1}]
  \morphism(1200,0)/->/<1200,0>[{b\uparrow_pb}`{b\uparrow_pb\uparrow_pb};]
  \morphism(1200,0)/{@{->}@<7pt>}/<1200,0>[{b\uparrow_pb}`{b\uparrow_pb\uparrow_pb};]
  \morphism(1200,0)/{@{->}@<-7pt>}/<1200,0>[{b\uparrow_pb}`{b\uparrow_pb\uparrow_pb};]
}
\defdiag{Semanticdescentfactorizationofp}{   
  \morphism(0,0)|a|/->/<1200,0>[{e}`{b};{p}]
  \morphism(0,0)/->/<600,-300>[{e}`{\mathrm{lax}\textrm{-}\mathcal{D}\mathrm{esc}\left(\mathcal{H}_p\right)};]
  \morphism(600,-300)/->/<600,300>[{\mathrm{lax}\textrm{-}\mathcal{D}\mathrm{esc}\left(\mathcal{H}_p\right)}`{b};]
}
\defdiag{SemanticFactorizationp}{   
  \morphism(0,0)|a|/->/<1200,0>[{e}`{b};{p}]
  \morphism(0,0)/->/<600,-300>[{e}`{b^\t};]
  \morphism(600,-300)/->/<600,300>[{b^\t}`{b};]
}
\defdiag{pushoutdiagramdefinition}{   
  \morphism(0,0)|a|/->/<0,-450>[{e}`{b_0};{p_0}]
  \morphism(0,0)|a|/->/<900,0>[{e}`{b_1};{p_1}]
  \morphism(0,-450)|l|/->/<900,0>[{b_0}`{b_0\sqcup_{(p_0,p_1)}b_1};{\mathfrak{d}^1_{p_0\sqcup_{e}p_1}}]
  \morphism(900,0)|r|/->/<0,-450>[{b_1}`{b_0\sqcup_{(p_0,p_1)}b_1};{\mathfrak{d}^0_{p_0\sqcup_{e}p_1}}]
}
\defdiag{twocellpushoutdefinitionrightside}{   
  \morphism(675,0)|l|/{@{->}@/_25pt/}/<0,-525>[{b_1}`{y};{h_0}]
  \morphism(675,0)|r|/{@{->}@/^25pt/}/<0,-525>[{b_1}`{y};{h_0'}]
  \morphism(0,0)|l|/->/<0,-525>[{e}`{b_0};{p_0}]
  \morphism(0,0)|a|/->/<675,0>[{e}`{b_1};{p_1}]
  \morphism(0,-525)|r|/->/<675,0>[{b_0}`{y};{h_1}]
  \morphism(98,-262)/=/<180,0>[{\phantom{O}}`{\phantom{O}};]
  \morphism(525,-262)|a|/=>/<300,0>[{\phantom{O}}`{\phantom{O}};{\xi_0}]
}
\defdiag{twocellofpushoutdefinitionleftside}{   
  \morphism(0,-525)|a|/{@{->}@/^25pt/}/<675,0>[{b_0}`{y};{h_1'}]
  \morphism(0,0)|a|/->/<675,0>[{e}`{b_1};{p_1}]
  \morphism(0,-525)|b|/{@{->}@/_25pt/}/<675,0>[{b_0}`{y};{h_1}]
  \morphism(0,-525)|l|/<-/<0,525>[{b_0}`{e};{p_0}]
  \morphism(675,-525)|r|/<-/<0,525>[{y}`{b_1};{h_0'}]
  \morphism(338,-22)/=/<0,-180>[{\phantom{O}}`{\phantom{O}};]
  \morphism(338,-375)|r|/<=/<0,-300>[{\phantom{O}}`{\phantom{O}};{\xi_1}]
}
\defdiag{xizeropushout}{   
  \morphism(0,-450)|r|/{@{->}@/^28pt/}/<0,-600>[{b_0\sqcup_{(p_0,p_1)}b_1}`{y};{h'}]
  \morphism(0,-450)|l|/{@{->}@/_28pt/}/<0,-600>[{b_0\sqcup_{(p_0,p_1)}b_1}`{y};{h}]
  \morphism(0,0)|l|/->/<0,-450>[{b_1}`{b_0\sqcup_{(p_0,p_1)}b_1};{\mathfrak{d}^0_{p_0\sqcup_{e}p_1}}]
  \morphism(-188,-750)|a|/=>/<375,0>[{\phantom{O}}`{\phantom{O}};{\xi}]
}
\defdiag{xiumpushout}{   
  \morphism(0,-450)|r|/{@{->}@/^28pt/}/<0,-600>[{b_0\sqcup_{(p_0,p_1)}b_1}`{y};{h'}]
  \morphism(0,-450)|l|/{@{->}@/_28pt/}/<0,-600>[{b_0\sqcup_{(p_0,p_1)}b_1}`{y};{h}]
  \morphism(0,0)|l|/->/<0,-450>[{b_0}`{b_0\sqcup_{(p_0,p_1)}b_1};{\mathfrak{d}^1_{p_0\sqcup_{e}p_1}}]
  \morphism(-188,-750)|a|/=>/<375,0>[{\phantom{O}}`{\phantom{O}};{\xi}]
}
\defdiag{opcommatwocelldefinition}{   
  \morphism(375,0)|a|/->/<-375,-375>[{e}`{b_0};{p_0}]
  \morphism(375,0)|a|/->/<375,-375>[{e}`{b_1};{p_1}]
  \morphism(0,-375)|l|/->/<375,-375>[{b_0}`{p_0\uparrow{p_1}};{\delta_{p_0\uparrow{p_1}}^1}]
  \morphism(750,-375)|r|/->/<-375,-375>[{b_1}`{p_0\uparrow{p_1}};{\delta_{p_0\uparrow{p_1}}^0}]
  \morphism(150,-375)|a|/=>/<450,0>[{\phantom{O}}`{\phantom{O}};{\alpha^{p_0\uparrow{p_1}}}]
}
\defdiag{opcommatwocelldefinitionuniversalproperty}{   
  \morphism(375,0)|a|/->/<-375,-375>[{e}`{b_0};{p_0}]
  \morphism(375,0)|a|/->/<375,-375>[{e}`{b_1};{p_1}]
  \morphism(0,-375)|l|/->/<375,-375>[{b_0}`{p_0\uparrow{p_1}};{\delta_{p_0\uparrow{p_1}}^1}]
  \morphism(750,-375)|r|/->/<-375,-375>[{b_1}`{p_0\uparrow{p_1}};{\delta_{p_0\uparrow{p_1}}^0}]
  \morphism(375,-750)|a|/->/<0,-300>[{p_0\uparrow{p_1}}`{y};{h}]
  \morphism(150,-375)|a|/=>/<450,0>[{\phantom{O}}`{\phantom{O}};{\alpha^{p_0\uparrow{p_1}}}]
}
\defdiag{opcommatwocelldefinitionuniversalpropertyrightside}{   
  \morphism(375,0)|a|/->/<-375,-375>[{e}`{b_0};{p_0}]
  \morphism(375,0)|a|/->/<375,-375>[{e}`{b_1};{p_1}]
  \morphism(0,-375)|l|/->/<375,-375>[{b_0}`{y};{h_1}]
  \morphism(750,-375)|r|/->/<-375,-375>[{b_1}`{y};{h_0}]
  \morphism(188,-375)|a|/=>/<375,0>[{\phantom{O}}`{\phantom{O}};{\beta}]
}
\defdiag{twocellofopcommatwodefinitionrightsidenovo}{   
  \morphism(450,0)|l|/->/<-450,-450>[{e}`{b_0};{p_0}]
  \morphism(450,0)|r|/->/<450,-450>[{e}`{b_1};{p_1}]
  \morphism(0,-450)|m|/->/<450,-450>[{b_0}`{p_0\uparrow{p_1}};{\delta_{p_0\uparrow{p_1}}^1}]
  \morphism(900,-450)|m|/->/<-450,-450>[{b_1}`{p_0\uparrow{p_1}};{\delta_{p_0\uparrow{p_1}}^0}]
  \morphism(450,-900)|m|/->/<0,-450>[{p_0\uparrow{p_1}}`{y};{h}]
  \morphism(900,-450)|r|/{@{->}@/^30pt/}/<-450,-900>[{b_1}`{y};{h'\,\cdot\,\delta_{p_0\uparrow{p_1}}^0}]
  \morphism(225,-450)|a|/=>/<450,0>[{\phantom{O}}`{\phantom{O}};{\alpha^{p_0\uparrow{p_1}}}]
  \morphism(578,-900)|a|/=>/<345,0>[{\phantom{O}}`{\phantom{O}};{\xi_0}]
}
\defdiag{twocellofopcommatwodefinitionleftsidenovo}{   
  \morphism(450,0)|l|/->/<-450,-450>[{e}`{b_0};{p_0}]
  \morphism(450,0)|r|/->/<450,-450>[{e}`{b_1};{p_1}]
  \morphism(0,-450)|m|/->/<450,-450>[{b_0}`{p_0\uparrow{p_1}};{\delta_{p_0\uparrow{p_1}}^1}]
  \morphism(900,-450)|m|/->/<-450,-450>[{b_1}`{p_0\uparrow{p_1}};{\delta_{p_0\uparrow{p_1}}^0}]
  \morphism(450,-900)|m|/->/<0,-450>[{p_0\uparrow{p_1}}`{y};{h'}]
  \morphism(0,-450)|l|/{@{->}@/_30pt/}/<450,-900>[{b_0}`{y};{h\,\cdot\,\delta_{p_0\uparrow{p_1}}^1}]
  \morphism(225,-450)|a|/=>/<450,0>[{\phantom{O}}`{\phantom{O}};{\alpha^{p_0\uparrow{p_1}}}]
  \morphism(-22,-900)|a|/=>/<345,0>[{\phantom{O}}`{\phantom{O}};{\xi_1}]
}
\defdiag{xizeroopcomma}{   
  \morphism(0,-450)|r|/{@{->}@/^28pt/}/<0,-600>[{p_0\uparrow{p_1}}`{y};{h'}]
  \morphism(0,-450)|l|/{@{->}@/_28pt/}/<0,-600>[{p_0\uparrow{p_1}}`{y};{h}]
  \morphism(0,0)|l|/->/<0,-450>[{b_1}`{p_0\uparrow{p_1}};{\delta_{p_0\uparrow{p_1}}^0}]
  \morphism(-188,-750)|a|/=>/<375,0>[{\phantom{O}}`{\phantom{O}};{\xi}]
}
\defdiag{xiumopcomma}{   
  \morphism(0,-450)|r|/{@{->}@/^28pt/}/<0,-600>[{p_0\uparrow{p_1}}`{y};{h'}]
  \morphism(0,-450)|l|/{@{->}@/_28pt/}/<0,-600>[{p_0\uparrow{p_1}}`{y};{h}]
  \morphism(0,0)|l|/->/<0,-450>[{b_0}`{p_0\uparrow{p_1}};{\delta_{p_0\uparrow{p_1}}^1}]
  \morphism(-188,-750)|a|/=>/<375,0>[{\phantom{O}}`{\phantom{O}};{\xi}]
}
\defdiag{universaltwocelllaxdescentpsi}{   
  \morphism(450,0)|a|/->/<-450,-375>[{\lim(\mathfrak{D},\BBBB)}`{\BBBB{(\umm)}};{\dd^{\left(\mathfrak{D},\BBBB\right)}}]
  \morphism(450,0)|a|/->/<450,-375>[{\lim(\mathfrak{D},\BBBB)}`{\BBBB{(\umm)}};{\dd^{\left(\mathfrak{D},\BBBB\right)}}]
  \morphism(0,-375)|b|/->/<450,-375>[{\BBBB{(\umm)}}`{\BBBB{(\doiss)}};{\BBBB(\dd^1)}]
  \morphism(900,-375)|b|/->/<-450,-375>[{\BBBB{(\umm)}}`{\BBBB{(\doiss)}};{\BBBB(\dd^0)}]
  \morphism(225,-375)|a|/=>/<450,0>[{\phantom{O}}`{\phantom{O}};{\Psi^{\left(\mathfrak{D},\BBBB\right)}}]
}
\defdiag{universalonecelllaxdescentda}{   
  \morphism(0,0)|a|/->/<675,0>[{\lim(\mathfrak{D},\BBBB)}`{\BBBB{(\umm)}};{\dd^{\left(\mathfrak{D},\BBBB\right)}}]
}
\defdiag{laxdescentassociativityleftside}{   
  \morphism(300,0)|a|/->/<-300,-450>[{\BBBB{(\umm)}}`{\BBBB{(\doiss)}};{\BBBB(\dd^1)}]
  \morphism(300,0)|m|/->/<300,-450>[{\BBBB{(\umm)}}`{\BBBB{(\doiss)}};{\BBBB(\dd^1)}]
  \morphism(900,0)|a|/->/<-600,0>[{y}`{\BBBB{(\umm)}};{h}]
  \morphism(900,0)|r|/->/<300,-450>[{y}`{\BBBB{(\umm)}};{h}]
  \morphism(0,-450)|b|/->/<300,-450>[{\BBBB{(\doiss)}}`{\BBBB{(\tress)}};{\BBBB(\DD^2)}]
  \morphism(600,-450)|m|/->/<-300,-450>[{\BBBB{(\doiss)}}`{\BBBB{(\tress)}};{\BBBB(\DD^1)}]
  \morphism(1200,-450)|m|/->/<-600,0>[{\BBBB{(\umm)}}`{\BBBB{(\doiss)}};{\BBBB(\dd^0)}]
  \morphism(1200,-450)|r|/->/<-300,-450>[{\BBBB{(\umm)}}`{\BBBB{(\doiss)}};{\BBBB(\dd^0)}]
  \morphism(900,-900)|b|/->/<-600,0>[{\BBBB{(\doiss)}}`{\BBBB{(\tress)}};{\BBBB(\DD^0)}]
  \morphism(112,-450)|a|/=>/<375,0>[{\phantom{O}}`{\phantom{O}};{\BBBB(\sigma_{12})}]
  \morphism(548,-225)|a|/=>/<450,0>[{\phantom{O}}`{\phantom{O}};{\beta}]
  \morphism(562,-675)|a|/=>/<450,0>[{\phantom{O}}`{\phantom{O}};{\BBBB(\sigma_{01})}]
}
\defdiag{laxdescentassociativityrightside}{   
  \morphism(900,0)|a|/->/<300,-450>[{y}`{\BBBB{(\umm)}};{h}]
  \morphism(900,0)|m|/->/<-300,-450>[{y}`{\BBBB{(\umm)}};{h}]
  \morphism(300,0)|a|/<-/<600,0>[{\BBBB{(\umm)}}`{y};{h}]
  \morphism(300,0)|l|/->/<-300,-450>[{\BBBB{(\umm)}}`{\BBBB{(\doiss)}};{\BBBB(\dd^1)}]
  \morphism(1200,-450)|b|/->/<-300,-450>[{\BBBB{(\umm)}}`{\BBBB{(\doiss)}};{\BBBB(\dd^0)}]
  \morphism(600,-450)|m|/->/<300,-450>[{\BBBB{(\umm)}}`{\BBBB{(\doiss)}};{\BBBB(\dd^1)}]
  \morphism(0,-450)|m|/<-/<600,0>[{\BBBB{(\doiss)}}`{\BBBB{(\umm)}};{\BBBB(\dd^0)}]
  \morphism(0,-450)|l|/->/<300,-450>[{\BBBB{(\doiss)}}`{\BBBB{(\tress)}};{\BBBB(\DD^2)}]
  \morphism(300,-900)|b|/<-/<600,0>[{\BBBB{(\tress)}}`{\BBBB{(\doiss)}};{\BBBB(\DD^0)}]
  \morphism(712,-450)|a|/=>/<375,0>[{\phantom{O}}`{\phantom{O}};{\beta}]
  \morphism(202,-225)|a|/=>/<450,0>[{\phantom{O}}`{\phantom{O}};{\beta}]
  \morphism(188,-675)|a|/=>/<450,0>[{\phantom{O}}`{\phantom{O}};{\BBBB(\sigma_{02})}]
}
\defdiag{laxdescentidentityleftside}{   
  \morphism(0,0)|a|/->/<525,0>[{y}`{\BBBB{(\umm)}};{h}]
  \morphism(0,0)|l|/->/<0,-525>[{y}`{\BBBB{(\umm)}};{h}]
  \morphism(525,0)|m|/->/<0,-525>[{\BBBB{(\umm)}}`{\BBBB{(\doiss)}};{\BBBB(\dd^0)}]
  \morphism(525,0)/{@{=}@/^22pt/}/<300,-900>[{\BBBB{(\umm)}}`{\BBBB{(\umm)}};]
  \morphism(0,-525)|m|/->/<525,0>[{\BBBB{(\umm)}}`{\BBBB{(\doiss)}};{\BBBB(\dd^1)}]
  \morphism(0,-525)/{@{=}@/_20pt/}/<825,-375>[{\BBBB{(\umm)}}`{\BBBB{(\umm)}};]
  \morphism(525,-525)|m|/->/<300,-375>[{\BBBB{(\doiss)}}`{\BBBB{(\umm)}};{\BBBB(\ss^0)}]
  \morphism(150,-712)|a|/{@{=>}@<-3pt>}/<450,0>[{\phantom{O}}`{\phantom{O}};{\AAAA(\mathfrak{n}_1)^{-1}}]
  \morphism(540,-450)|a|/=>/<345,0>[{\phantom{O}}`{\phantom{O}};{\AAAA(\mathfrak{n}_0)}]
  \morphism(90,-262)|a|/=>/<345,0>[{\phantom{O}}`{\phantom{O}};{\beta}]
}
\defdiag{laxdescentidentityrightside}{   
  \morphism(0,0)|r|/{@{->}@/^15pt/}/<0,-900>[{y}`{\BBBB{(\umm)}};{h}]
  \morphism(0,0)|l|/{@{->}@/_15pt/}/<0,-900>[{y}`{\BBBB{(\umm)}};{h}]
  \morphism(-98,-450)/=/<195,0>[{\phantom{O}}`{\phantom{O}};]
}
\defdiag{laxdescentgeneralfactorization}{   
  \morphism(0,0)|a|/->/<1800,0>[{y}`{\BBBB{(\umm)}};{h}]
  \morphism(0,0)|l|/->/<900,-375>[{y}`{\lim\left(\mathfrak{D},\BBBB\right)};{h^{(\BBBB{,}\beta)}}]
  \morphism(900,-375)|r|/->/<900,375>[{\lim\left(\mathfrak{D},\BBBB\right)}`{\BBBB{(\umm)}};{\dd^{(\mathfrak{D},\BBBB)}}]
}
\defdiag{equation_on_the_descent_datum_leftside}{   
  \morphism(450,0)|r|/->/<0,-300>[{y}`{\lim(\mathfrak{D},\BBBB)};{h^{(\BBBB{,}\beta)}}]
  \morphism(450,-300)|l|/->/<-450,-300>[{\lim(\mathfrak{D},\BBBB)}`{\BBBB{(\umm)}};{\dd^{(\mathfrak{D},\BBBB)}}]
  \morphism(450,-300)|r|/->/<450,-300>[{\lim(\mathfrak{D},\BBBB)}`{\BBBB{(\umm)}};{\dd^{(\mathfrak{D},\BBBB)}}]
  \morphism(0,-600)|l|/->/<450,-300>[{\BBBB{(\umm)}}`{\BBBB{(\doiss)}};{\BBBB(\dd^1)}]
  \morphism(900,-600)|r|/->/<-450,-300>[{\BBBB{(\umm)}}`{\BBBB{(\doiss)}};{\BBBB(\dd^0)}]
  \morphism(188,-600)|a|/=>/<525,0>[{\phantom{O}}`{\phantom{O}};{\Psi^{(\mathfrak{D},\BBBB)}}]
}
\defdiag{equation_on_the_descent_datum_rightside}{   
  \morphism(450,0)|l|/->/<-450,-600>[{y}`{\BBBB{(\umm)}};{h}]
  \morphism(450,0)|r|/->/<450,-600>[{y}`{\BBBB{(\umm)}};{h}]
  \morphism(0,-600)|l|/->/<450,-300>[{\BBBB{(\umm)}}`{\BBBB{(\doiss)}};{\BBBB(\dd^1)}]
  \morphism(900,-600)|r|/->/<-450,-300>[{\BBBB{(\umm)}}`{\BBBB{(\doiss)}};{\BBBB(\dd^0)}]
  \morphism(188,-450)|a|/=>/<525,0>[{\phantom{O}}`{\phantom{O}};{\beta}]
}
\defdiag{equation_two_cell_for_descent_left_side}{   
  \morphism(0,0)|a|/->/<600,0>[{y}`{\BBBB{(\umm)}};{h_0}]
  \morphism(0,0)|m|/{@{->}@/^20pt/}/<0,-600>[{y}`{\BBBB{(\umm)}};{h_0}]
  \morphism(0,0)|l|/{@{->}@/_20pt/}/<0,-600>[{y}`{\BBBB{(\umm)}};{h_1}]
  \morphism(600,0)|r|/->/<0,-600>[{\BBBB{(\umm)}}`{\BBBB{(\doiss)}};{\BBBB(\dd^0)}]
  \morphism(0,-600)|r|/->/<600,0>[{\BBBB{(\umm)}}`{\BBBB{(\doiss)}};{\BBBB(\dd^1)}]
  \morphism(202,-300)|a|/=>/<375,0>[{\phantom{O}}`{\phantom{O}};{\beta_0}]
  \morphism(-150,-300)|a|/=>/<300,0>[{\phantom{O}}`{\phantom{O}};{\xi}]
}
\defdiag{equation_two_cell_for_descent_right_side}{   
  \morphism(0,0)|a|/->/<0,-600>[{y}`{\AAAA{(\umm)}};{h_1}]
  \morphism(0,0)|a|/{@{->}@/^20pt/}/<600,0>[{y}`{\BBBB{(\umm)}};{h_0}]
  \morphism(0,0)|m|/{@{->}@/_20pt/}/<600,0>[{y}`{\BBBB{(\umm)}};{h_1}]
  \morphism(600,0)|r|/->/<0,-600>[{\BBBB{(\umm)}}`{\AAAA{(\doiss)}};{\BBBB(\dd^0)}]
  \morphism(0,-600)|r|/->/<600,0>[{\AAAA{(\umm)}}`{\AAAA{(\doiss)}};{\BBBB(\dd^1)}]
  \morphism(112,-300)|a|/{@{=>}@<-8pt>}/<375,0>[{\phantom{O}}`{\phantom{O}};{\beta_1}]
  \morphism(300,150)|r|/<=/<0,-300>[{\phantom{O}}`{\phantom{O}};{\xi}]
}
\defdiag{xilaxdescenttwocellproperty}{   
  \morphism(0,-675)|l|/{@{<-}@/^30pt/}/<0,675>[{\lim(\mathfrak{D},\BBBB)}`{y};{h_1^{(\BBBB,\beta_1)}}]
  \morphism(0,-675)|r|/{@{<-}@/_30pt/}/<0,675>[{\lim(\mathfrak{D},\BBBB)}`{y};{h_0^{(\BBBB,\beta_0)}}]
  \morphism(0,-1125)|l|/<-/<0,450>[{\BBBB{(\umm)}}`{\lim(\mathfrak{D},\BBBB)};{\dd^{\left(\mathfrak{D},\BBBB\right)}}]
  \morphism(-225,-338)|a|/=>/<450,0>[{\phantom{O}}`{\phantom{O}};{\xi^{(\BBBB,\beta_1,\beta_0)}}]
}
\defdiag{opcomma_p_along_p}{   
  \morphism(300,0)|a|/->/<-300,-300>[{e}`{b};{p}]
  \morphism(300,0)|a|/->/<300,-300>[{e}`{b};{p}]
  \morphism(0,-300)|l|/->/<300,-300>[{b}`{b\uparrow_pb};{\d^1}]
  \morphism(600,-300)|r|/->/<-300,-300>[{b}`{b\uparrow_pb};{\d^0}]
  \morphism(98,-300)|a|/=>/<405,0>[{\phantom{O}}`{\phantom{O}};{\upalpha}]
}
\defdiag{pushoutdiagramfordefiningtwodimensionalcokerneldiagram}{   
  \morphism(0,0)|a|/->/<0,-525>[{b}`{b\uparrow_pb};{\d^1}]
  \morphism(0,0)|a|/->/<525,0>[{b}`{b\uparrow_pb};{\d^0}]
  \morphism(0,-525)|l|/->/<525,0>[{b\uparrow_pb}`{b\uparrow_pb\uparrow_pb};{\D^0}]
  \morphism(525,0)|r|/->/<0,-525>[{b\uparrow_pb}`{b\uparrow_pb\uparrow_pb};{\D^2}]
}
\defdiag{twodimensionalcokerneldiagramassociativityleftsidedefinition}{   
  \morphism(375,0)|a|/->/<-375,-375>[{e}`{b};{p}]
  \morphism(375,0)|a|/->/<375,-375>[{e}`{b};{p}]
  \morphism(0,-375)|l|/->/<375,-375>[{b}`{b\uparrow_pb};{\d^1}]
  \morphism(750,-375)|r|/->/<-375,-375>[{b}`{b\uparrow_pb};{\d^0}]
  \morphism(375,-750)|l|/->/<0,-375>[{b\uparrow_pb}`{b\uparrow_pb\uparrow_pb};{\D^1}]
  \morphism(150,-375)|a|/=>/<450,0>[{\phantom{O}}`{\phantom{O}};{\upalpha}]
}
\defdiag{twodimensionalcokerneldiagramassociativityrightsidedefinition}{   
  \morphism(675,0)|a|/->/<-675,-750>[{e}`{b};{p}]
  \morphism(675,0)|a|/->/<675,-750>[{e}`{b};{p}]
  \morphism(0,-750)|l|/->/<675,-375>[{b}`{b\uparrow_pb\uparrow_pb};{\D^2\d^1}]
  \morphism(1350,-750)|r|/->/<-675,-375>[{b}`{b\uparrow_pb\uparrow_pb};{\D^0\d^0}]
  \morphism(675,-375)|m|/{@{->}@/_21pt/}/<0,-750>[{b}`{b\uparrow_pb\uparrow_pb};{\D^0\d^0}]
  \morphism(675,-375)|m|/{@{->}@/^21pt/}/<0,-750>[{b}`{b\uparrow_pb\uparrow_pb};{\D^0\d^1}]
  \morphism(675,0)|m|/->/<0,-375>[{e}`{b};{p}]
  \morphism(112,-750)|a|/{@{=>}@<10pt>}/<375,0>[{\phantom{O}}`{\phantom{O}};{\id_{\D^2}\ast\upalpha}]
  \morphism(862,-750)|a|/{@{=>}@<10pt>}/<375,0>[{\phantom{O}}`{\phantom{O}};{\id_{\D^0}\ast\upalpha}]
  \morphism(578,-750)/=/<195,0>[{\phantom{O}}`{\phantom{O}};]
}
\defdiag{twodimensionalcokerneldiagramidentityleftsidedefinition}{   
  \morphism(375,0)|a|/->/<-375,-375>[{e}`{b};{p}]
  \morphism(375,0)|a|/->/<375,-375>[{e}`{b};{p}]
  \morphism(0,-375)|l|/->/<375,-375>[{b}`{b\uparrow_pb};{\d^1}]
  \morphism(750,-375)|r|/->/<-375,-375>[{b}`{b\uparrow_pb};{\d^0}]
  \morphism(375,-750)|l|/->/<0,-375>[{b\uparrow_pb}`{b};{\s^0}]
  \morphism(150,-375)|a|/=>/<450,0>[{\phantom{O}}`{\phantom{O}};{\upalpha}]
}
\defdiag{twodimensionalcokerneldiagramidentityrightsidedefinition}{   
  \morphism(0,0)|m|/{@{->}@/_21pt/}/<0,-1125>[{e}`{b};{p}]
  \morphism(0,0)|m|/{@{->}@/^21pt/}/<0,-1125>[{e}`{b};{p}]
  \morphism(-98,-562)/=/<195,0>[{\phantom{O}}`{\phantom{O}};]
}
\defdiag{semanticdescent_factorization_of_p}{   
  \morphism(0,0)|a|/->/<1800,0>[{e}`{b};{p}]
  \morphism(0,0)|l|/->/<900,-375>[{e}`{\lim\left(\mathfrak{D},\mathcal{H}_p\right)};{p^\H}]
  \morphism(900,-375)|r|/->/<900,375>[{\lim\left(\mathfrak{D},\mathcal{H}_p\right)}`{b};{\dd^{p}}]
}
\defdiag{equation_on_the_descent_datum_facorization_twodimensionalcokerneldiagram_leftside}{   
  \morphism(450,0)|r|/->/<0,-300>[{e}`{\lim(\mathfrak{D},\mathcal{H}_p)};{p^{\H}}]
  \morphism(450,-300)|l|/->/<-450,-300>[{\lim(\mathfrak{D},\mathcal{H}_p)}`{b};{\dd^{p}}]
  \morphism(450,-300)|r|/->/<450,-300>[{\lim(\mathfrak{D},\mathcal{H}_p)}`{b};{\dd^{p}}]
  \morphism(0,-600)|l|/->/<450,-300>[{b}`{b\uparrow_pb};{\d^1}]
  \morphism(900,-600)|r|/->/<-450,-300>[{b}`{b\uparrow_pb};{\d^0}]
  \morphism(188,-600)|a|/=>/<525,0>[{\phantom{O}}`{\phantom{O}};{\Psi^{p}}]
}
\defdiag{equation_on_the_descent_datum_facorization_twodimensionalcokerneldiagram_rightside}{   
  \morphism(450,0)|l|/->/<-450,-600>[{e}`{b};{p}]
  \morphism(450,0)|r|/->/<450,-600>[{e}`{b};{p}]
  \morphism(0,-600)|l|/->/<450,-300>[{b}`{b\uparrow_pb};{\d^1}]
  \morphism(900,-600)|r|/->/<-450,-300>[{b}`{b\uparrow_pb};{\d^0}]
  \morphism(188,-450)|a|/=>/<525,0>[{\phantom{O}}`{\phantom{O}};{\upalpha}]
}
\defdiag{image_by_of_opcomma_p_along_p_Ax}{   
  \morphism(450,0)|a|/->/<-450,-450>[{\AAA\left(x,e\right)}`{\AAA\left(x,b\right)};{\AAA\left(x,p\right)}]
  \morphism(450,0)|a|/->/<450,-450>[{\AAA\left(x,e\right)}`{\AAA\left(x,b\right)};{\AAA\left(x,p\right)}]
  \morphism(0,-450)|l|/->/<450,-450>[{\AAA\left(x,b\right)}`{\AAA\left(x,b\uparrow_pb\right)};{\AAA\left(x,\d^1\right)}]
  \morphism(900,-450)|r|/->/<-450,-450>[{\AAA\left(x,b\right)}`{\AAA\left(x,b\uparrow_pb\right)};{\AAA\left(x,\d^0\right)}]
  \morphism(232,-450)|a|/=>/<435,0>[{\phantom{O}}`{\phantom{O}};{\AAA\left(x,\upalpha\right)}]
}
\defdiag{image_semanticdescent_factorization_of_p}{   
  \morphism(0,0)|a|/->/<1800,0>[{\AAA\left(x,e\right)}`{\AAA\left(x,b\right)};{\AAA(x,p)}]
  \morphism(0,0)|m|/->/<900,-480>[{\AAA\left(x,e\right)}`{\lim\left(\mathfrak{D},\AAA(x,\mathcal{H}_p-)\right)};{\AAA(x,p)^{\left(\AAA\left(x,\mathcal{H}_p-\right),\AAA\left(x,\upalpha\right)\right)}}]
  \morphism(900,-480)|m|/->/<900,480>[{\lim\left(\mathfrak{D},\AAA(x,\mathcal{H}_p-)\right)}`{\AAA\left(x,b\right)};{\dd^{\left(\mathfrak{D},\AAA\left(x,\mathcal{H}_p-\right)\right)}}]
}
\defdiag{dual_semanticcodescent_factorization_of_l}{   
  \morphism(0,0)|a|/->/<1800,0>[{e}`{b};{l}]
  \morphism(0,0)|l|/->/<900,-480>[{e}`{\colim\left(\mathfrak{D},\mathcal{H}^l\right)};{\dd_{l}}]
  \morphism(900,-480)|r|/->/<900,480>[{\colim\left(\mathfrak{D},\mathcal{H}^l\right)}`{b};{l_\H}]
}
\defdiag{multiplication_of_the_monad_definition}{   
  \morphism(450,0)|l|/->/<-450,-300>[{b}`{b};{t}]
  \morphism(0,-300)|l|/->/<450,-300>[{b}`{b};{t}]
  \morphism(450,0)|r|/->/<0,-600>[{b}`{b};{t}]
  \morphism(75,-300)|a|/=>/<375,0>[{\phantom{O}}`{\phantom{O}};{m}]
}
\defdiag{unit_of_the_monad_definition}{   
  \morphism(0,0)/=/<0,-600>[{b}`{b};]
  \morphism(0,0)|r|/{@{->}@/^40pt/}/<0,-600>[{b}`{b};{t}]
  \morphism(-38,-300)|a|/=>/<375,0>[{\phantom{O}}`{\phantom{O}};{\eta}]
}
\defdiag{leftsideoftheequationassociativityofmonad}{   
  \morphism(600,0)|r|/->/<0,-600>[{b}`{b};{t}]
  \morphism(0,-600)|b|/->/<600,0>[{b}`{b};{t}]
  \morphism(0,0)|m|/->/<600,-600>[{b}`{b};{t}]
  \morphism(0,-600)|l|/->/<0,600>[{b}`{b};{t}]
  \morphism(0,0)|a|/->/<600,0>[{b}`{b};{t}]
  \morphism(450,-0)|r|/{@{=>}@<-5pt>}/<0,-375>[{\phantom{O}}`{\phantom{O}};{m}]
  \morphism(150,-225)|l|/{@{=>}@<5pt>}/<0,-375>[{\phantom{O}}`{\phantom{O}};{m}]
}
\defdiag{rightsideoftheequationassociativityofmonad}{   
  \morphism(600,0)|r|/->/<0,-600>[{b}`{b};{t}]
  \morphism(0,0)|a|/->/<600,0>[{b}`{b};{t}]
  \morphism(0,-600)|l|/->/<0,600>[{b}`{b};{t}]
  \morphism(0,-600)|b|/->/<600,0>[{b}`{b};{t}]
  \morphism(0,-600)|m|/->/<600,600>[{b}`{b};{t}]
  \morphism(150,-0)|l|/{@{=>}@<5pt>}/<0,-375>[{\phantom{O}}`{\phantom{O}};{m}]
  \morphism(450,-225)|r|/{@{=>}@<-5pt>}/<0,-375>[{\phantom{O}}`{\phantom{O}};{m}]
}
\defdiag{firstsideoftheequationidenityofamonad}{   
  \morphism(0,0)|r|/->/<600,-600>[{b}`{b};{t}]
  \morphism(0,-600)|m|/->/<0,600>[{b}`{b};{t}]
  \morphism(0,-600)|b|/->/<600,0>[{b}`{b};{t}]
  \morphism(0,0)/{@{=}@/_35pt/}/<0,-600>[{b}`{b};]
  \morphism(-292,-300)|a|/=>/<300,0>[{\phantom{O}}`{\phantom{O}};{\eta}]
  \morphism(300,-255)|r|/{@{=>}@<-13pt>}/<0,-360>[{\phantom{O}}`{\phantom{O}};{m}]
}
\defdiag{secondsideoftheequationidenityofamonad}{   
  \morphism(0,0)|m|/->/<600,-600>[{b}`{b};{t}]
  \morphism(0,-600)|l|/->/<0,600>[{b}`{b};{t}]
  \morphism(0,-600)|b|/->/<600,0>[{b}`{b};{t}]
  \morphism(600,-600)/{@{=}@/_35pt/}/<-600,600>[{b}`{b};]
  \morphism(300,-52)|r|/{@{=>}@<12pt>}/<0,-300>[{\phantom{O}}`{\phantom{O}};{\eta}]
  \morphism(300,-255)|l|/{@{=>}@<-10pt>}/<0,-360>[{\phantom{O}}`{\phantom{O}};{m}]
}
\defdiag{algebra_multiplication_of_the_monad_definition}{   
  \morphism(0,0)|a|/->/<600,-600>[{b}`{b};{t}]
  \morphism(0,-600)|l|/->/<0,600>[{b^\t}`{b};{\uu^\t}]
  \morphism(0,-600)|b|/->/<600,0>[{b^\t}`{b};{\uu^\t}]
  \morphism(300,-255)|l|/{@{=>}@<-10pt>}/<0,-360>[{\phantom{O}}`{\phantom{O}};{\mu^\t}]
}
\defdiag{algebra_leftsideoftheequationassociativityofmonad}{   
  \morphism(600,0)|r|/->/<0,-600>[{b}`{b};{t}]
  \morphism(0,-600)|b|/->/<600,0>[{y}`{b};{h}]
  \morphism(0,0)|m|/->/<600,-600>[{b}`{b};{t}]
  \morphism(0,-600)|l|/->/<0,600>[{y}`{b};{h}]
  \morphism(0,0)|a|/->/<600,0>[{b}`{b};{t}]
  \morphism(450,-0)|r|/{@{=>}@<-5pt>}/<0,-375>[{\phantom{O}}`{\phantom{O}};{m}]
  \morphism(150,-225)|l|/{@{=>}@<5pt>}/<0,-375>[{\phantom{O}}`{\phantom{O}};{\beta}]
}
\defdiag{algebra_rightsideoftheequationassociativityofmonad}{   
  \morphism(600,0)|r|/->/<0,-600>[{b}`{b};{t}]
  \morphism(0,0)|a|/->/<600,0>[{b}`{b};{t}]
  \morphism(0,-600)|l|/->/<0,600>[{y}`{b};{h}]
  \morphism(0,-600)|b|/->/<600,0>[{y}`{b};{h}]
  \morphism(0,-600)|m|/->/<600,600>[{y}`{b};{h}]
  \morphism(150,-0)|l|/{@{=>}@<5pt>}/<0,-375>[{\phantom{O}}`{\phantom{O}};{\beta}]
  \morphism(450,-225)|r|/{@{=>}@<-5pt>}/<0,-375>[{\phantom{O}}`{\phantom{O}};{\beta}]
}
\defdiag{algebra_secondsideoftheequationidenityofamonad}{   
  \morphism(0,0)|m|/->/<600,-600>[{b}`{b};{t}]
  \morphism(0,-600)|l|/->/<0,600>[{y}`{b};{h}]
  \morphism(0,-600)|b|/->/<600,0>[{y}`{b};{h}]
  \morphism(600,-600)/{@{=}@/_35pt/}/<-600,600>[{b}`{b};]
  \morphism(300,-52)|r|/{@{=>}@<12pt>}/<0,-300>[{\phantom{O}}`{\phantom{O}};{\eta}]
  \morphism(300,-255)|l|/{@{=>}@<-10pt>}/<0,-360>[{\phantom{O}}`{\phantom{O}};{\beta}]
}
\defdiag{algebra_firstsideoftheequationidenityofamonad}{   
  \morphism(0,-600)|r|/{@{->}@/_25pt/}/<0,600>[{y}`{b};{h}]
  \morphism(0,-600)|l|/{@{->}@/^25pt/}/<0,600>[{y}`{b};{h}]
  \morphism(-90,-300)/=/<180,0>[{\phantom{O}}`{\phantom{O}};]
}
\defdiag{generalalgebrafactorization_universalproperty}{   
  \morphism(0,0)|a|/->/<1800,0>[{y}`{b};{h}]
  \morphism(0,0)|l|/->/<900,-375>[{y}`{b^\t};{h^{(\t{,}\beta)}}]
  \morphism(900,-375)|r|/->/<900,375>[{b^\t}`{b};{\uu^\t}]
}
\defdiag{lefttwocell_algebrastructure_monad_resulting}{   
  \morphism(600,0)|a|/->/<600,-600>[{b}`{b};{t}]
  \morphism(600,-600)|l|/->/<0,600>[{b^\t}`{b};{\uu^\t}]
  \morphism(600,-600)|m|/->/<600,0>[{b^\t}`{b};{\uu^\t}]
  \morphism(0,-900)|a|/->/<600,300>[{y}`{b^\t};{h_1^{(\t{,}\beta_1)}}]
  \morphism(0,-900)|b|/{@{->}@/_30pt/}/<1200,300>[{y}`{b};{\uu^\t{\,\cdot\,}{h_0^{(\t{,}\beta_0)}}}]
  \morphism(900,-255)|l|/{@{=>}@<-10pt>}/<0,-360>[{\phantom{O}}`{\phantom{O}};{\mu^\t}]
  \morphism(600,-645)|r|/=>/<0,-360>[{\phantom{O}}`{\phantom{O}};{\xi}]
}
\defdiag{righttwocell_algebrastructure_monad_resulting}{   
  \morphism(600,0)|a|/->/<600,-600>[{b}`{b};{t}]
  \morphism(600,-600)|m|/->/<0,600>[{b^\t}`{b};{\uu^\t}]
  \morphism(600,-600)|b|/->/<600,0>[{b^\t}`{b};{\uu^\t}]
  \morphism(0,-900)|b|/->/<600,300>[{y}`{b^\t};{h_0^{(\t{,}\beta_0)}}]
  \morphism(0,-900)|l|/{@{->}@/^25pt/}/<600,900>[{y}`{b};{{\uu^\t}{\,\cdot\,}h_1^{(\t{,}\beta_1)}}]
  \morphism(900,-255)|l|/{@{=>}@<-10pt>}/<0,-360>[{\phantom{O}}`{\phantom{O}};{\mu^\t}]
  \morphism(158,-450)|a|/=>/<360,0>[{\phantom{O}}`{\phantom{O}};{\xi}]
}
\defdiag{xilaxdescenttwocellproperty_algebras}{   
  \morphism(0,-675)|l|/{@{<-}@/^30pt/}/<0,675>[{b^\t}`{y};{h_1^{(\t{,}\beta_1)}}]
  \morphism(0,-675)|r|/{@{<-}@/_30pt/}/<0,675>[{b^\t}`{y};{h_0^{(\t{,}\beta_0)}}]
  \morphism(0,-1050)|l|/<-/<0,375>[{b}`{b^\t};{\uu^\t}]
  \morphism(-225,-338)|a|/=>/<450,0>[{\phantom{O}}`{\phantom{O}};{\xi_{(\t,\beta_1,\beta_0)}}]
}
\defdiag{secondsideoftheequation_definition_of_Kan_Extensions}{   
  \morphism(0,0)|m|/->/<750,-750>[{x}`{y};{\ran_gf}]
  \morphism(0,-750)|l|/->/<0,750>[{z}`{x};{g}]
  \morphism(0,-750)|b|/->/<750,0>[{z}`{y};{f}]
  \morphism(0,0)|a|/{@{->}@/^35pt/}/<750,-750>[{x}`{y};{h}]
  \morphism(375,-112)|r|/{@{=>}@<12pt>}/<0,-270>[{\phantom{O}}`{\phantom{O}};{\beta}]
  \morphism(375,-382)|l|/{@{=>}@<-15pt>}/<0,-330>[{\phantom{O}}`{\phantom{O}};{\gamma^{\ran_gf}}]
}
\defdiag{firstsideoftheequation_definition_of_Kan_Extensions}{   
  \morphism(0,0)|m|/->/<750,-750>[{x}`{y};{\ran_gf}]
  \morphism(0,0)|a|/{@{->}@/^35pt/}/<750,-750>[{x}`{y};{h}]
  \morphism(375,-112)|r|/{@{=>}@<12pt>}/<0,-270>[{\phantom{O}}`{\phantom{O}};{\beta}]
}
\defdiag{firstdiagram_for_the_proof_of_opcomma_result_about_Kan_Extensions}{   
  \morphism(900,-525)|b|/->/<375,0>[{p_0\uparrow{p_1}}`{y};{h}]
  \morphism(300,-525)|b|/->/<600,0>[{b_1}`{p_0\uparrow{p_1}};{\delta^0_{p_0\uparrow{p_1}}}]
  \morphism(0,-525)|l|/->/<0,525>[{e}`{b_0};{p_0}]
  \morphism(0,-525)|b|/->/<300,0>[{e}`{b_1};{p_1}]
  \morphism(0,0)|a|/{@{->}@/^40pt/}/<1275,-525>[{b_0}`{y};{h_1'}]
  \morphism(450,38)|r|/{@{=>}@<12pt>}/<0,-525>[{\phantom{O}}`{\phantom{O}};{\check{\beta}}]
}
\defdiag{seconddiagram_for_the_proof_of_opcomma_result_about_Kan_Extensions}{   
  \morphism(900,-525)|b|/->/<375,0>[{p_0\uparrow{p_1}}`{y};{h'}]
  \morphism(300,-525)|b|/->/<600,0>[{b_1}`{p_0\uparrow{p_1}};{\delta^0_{p_0\uparrow{p_1}}}]
  \morphism(0,-525)|l|/->/<0,525>[{e}`{b_0};{p_0}]
  \morphism(0,-525)|b|/->/<300,0>[{e}`{b_1};{p_1}]
  \morphism(0,0)|a|/->/<900,-525>[{b_0}`{p_0\uparrow{p_1}};{\delta^1_{p_0\uparrow{p_1}}}]
  \morphism(0,0)|a|/{@{->}@/^40pt/}/<1275,-525>[{b_0}`{y};{h_1'}]
  \morphism(450,-195)|l|/{@{=>}@<-19pt>}/<0,-300>[{\phantom{O}}`{\phantom{O}};{\alpha^{p_0\uparrow{p_1}}}]
  \morphism(638,-98)/{@{=}@<13pt>}/<0,-180>[{\phantom{O}}`{\phantom{O}};]
}
\defdiag{KanExtension_Identity_Support_Lemma_Opcomma_Objects}{   
  \morphism(0,0)|m|/->/<900,-450>[{p_0\uparrow{p_1}}`{y};{h}]
  \morphism(0,-450)|l|/->/<0,450>[{b_1}`{p_0\uparrow{p_1}};{\delta^0_{p_0\uparrow{p_1}}}]
  \morphism(0,-450)|b|/->/<450,0>[{b_1}`{p_0\uparrow{p_1}};{\delta^0_{p_0\uparrow{p_1}}}]
  \morphism(450,-450)|b|/->/<450,0>[{p_0\uparrow{p_1}}`{y};{h}]
  \morphism(0,0)|a|/{@{->}@/^35pt/}/<900,-450>[{p_0\uparrow{p_1}}`{y};{h'}]
  \morphism(450,15)|r|/{@{=>}@<12pt>}/<0,-300>[{\phantom{O}}`{\phantom{O}};{\underline{\beta}}]
  \morphism(450,-225)/{@{=}@<-15pt>}/<0,-180>[{\phantom{O}}`{\phantom{O}};]
}
\defdiag{Leftside_KanExtension_Identity_Support_Lemma_Opcomma_Objects}{   
  \morphism(0,-450)|l|/->/<0,450>[{b_1}`{p_0\uparrow{p_1}};{\delta^0_{p_0\uparrow{p_1}}}]
  \morphism(0,-450)|b|/->/<450,0>[{b_1}`{p_0\uparrow{p_1}};{\delta^0_{p_0\uparrow{p_1}}}]
  \morphism(450,-450)|b|/->/<450,0>[{p_0\uparrow{p_1}}`{y};{h}]
  \morphism(0,0)|a|/{@{->}@/^35pt/}/<900,-450>[{p_0\uparrow{p_1}}`{y};{h'}]
  \morphism(450,-128)/=/<0,-195>[{\phantom{O}}`{\phantom{O}};]
}
\defdiag{uniquetwocellbetaunderlineoftheproofsupportopcommadois}{   
  \morphism(450,0)|l|/->/<-450,-450>[{e}`{b_0};{p_0}]
  \morphism(450,0)|r|/->/<450,-450>[{e}`{b_1};{p_1}]
  \morphism(0,-450)|m|/->/<450,-450>[{b_0}`{p_0\uparrow{p_1}};{\delta_{p_0\uparrow{p_1}}^1}]
  \morphism(900,-450)|m|/->/<-450,-450>[{b_1}`{p_0\uparrow{p_1}};{\delta_{p_0\uparrow{p_1}}^0}]
  \morphism(450,-900)|m|/->/<0,-450>[{p_0\uparrow{p_1}}`{y};{h'}]
  \morphism(900,-450)|r|/{@{->}@/^40pt/}/<-450,-900>[{b_1}`{y};{h\,\cdot\,\delta_{p_0\uparrow{p_1}}^0}]
  \morphism(225,-450)|a|/=>/<450,0>[{\phantom{O}}`{\phantom{O}};{\alpha^{p_0\uparrow{p_1}}}]
  \morphism(698,-900)/=/<180,0>[{\phantom{O}}`{\phantom{O}};]
}
\defdiag{uniquetwocellbetaunderlineoftheproofsupportopcomma}{   
  \morphism(450,0)|l|/->/<-450,-450>[{e}`{b_0};{p_0}]
  \morphism(450,0)|r|/->/<450,-450>[{e}`{b_1};{p_1}]
  \morphism(0,-450)|m|/->/<450,-450>[{b_0}`{p_0\uparrow{p_1}};{\delta_{p_0\uparrow{p_1}}^1}]
  \morphism(900,-450)|m|/->/<-450,-450>[{b_1}`{p_0\uparrow{p_1}};{\delta_{p_0\uparrow{p_1}}^0}]
  \morphism(450,-900)|m|/->/<0,-450>[{p_0\uparrow{p_1}}`{y};{h}]
  \morphism(0,-450)|l|/{@{->}@/_40pt/}/<450,-900>[{b_0}`{y};{h'_1}]
  \morphism(225,-450)|a|/=>/<450,0>[{\phantom{O}}`{\phantom{O}};{\alpha^{p_0\uparrow{p_1}}}]
  \morphism(-90,-900)|a|/=>/<405,0>[{\phantom{O}}`{\phantom{O}};{\underline{\beta}\ast{\id_{\delta_{p_0\uparrow{p_1}}^1}}}]
}
\defdiag{left_side_codensity_multiplication_definition}{   
  \morphism(0,0)|m|/->/<750,-750>[{b}`{b};{t}]
  \morphism(0,-750)|l|/->/<0,750>[{e}`{b};{p}]
  \morphism(0,-750)|b|/->/<750,0>[{e}`{b};{p}]
  \morphism(0,0)|a|/{@{->}@/^30pt/}/<750,-750>[{b}`{b};{t^2}]
  \morphism(375,-142)|r|/{@{=>}@<12pt>}/<0,-330>[{\phantom{O}}`{\phantom{O}};{m}]
  \morphism(375,-345)|l|/{@{=>}@<-15pt>}/<0,-405>[{\phantom{O}}`{\phantom{O}};{\gamma}]
}
\defdiag{right_side_codensity_multiplication_definition}{   
  \morphism(0,0)|a|/->/<375,-375>[{b}`{b};{t}]
  \morphism(0,-750)|l|/->/<0,750>[{e}`{b};{p}]
  \morphism(0,-750)|b|/->/<750,0>[{e}`{b};{p}]
  \morphism(375,-375)|a|/->/<375,-375>[{b}`{b};{t}]
  \morphism(0,-750)|m|/->/<375,375>[{e}`{b};{p}]
  \morphism(375,-428)|r|/=>/<0,-330>[{\phantom{O}}`{\phantom{O}};{\gamma}]
  \morphism(188,-240)|l|/{@{=>}@<-4pt>}/<0,-330>[{\phantom{O}}`{\phantom{O}};{\gamma}]
}
\defdiag{left_side_codensity_identity_definition}{   
  \morphism(0,0)|m|/->/<600,-600>[{b}`{b};{t}]
  \morphism(0,-600)|l|/->/<0,600>[{e}`{b};{p}]
  \morphism(0,-600)|b|/->/<600,0>[{e}`{b};{p}]
  \morphism(0,0)/{@{=}@/^30pt/}/<600,-600>[{b}`{b};]
  \morphism(300,-68)|r|/{@{=>}@<12pt>}/<0,-330>[{\phantom{O}}`{\phantom{O}};{\eta}]
  \morphism(300,-210)|l|/{@{=>}@<-15pt>}/<0,-405>[{\phantom{O}}`{\phantom{O}};{\gamma}]
}
\defdiag{right_side_codensity_identity_definition}{   
  \morphism(0,0)/=/<600,-600>[{b}`{b};]
  \morphism(0,-600)|l|/->/<0,600>[{e}`{b};{p}]
  \morphism(0,-600)|b|/->/<600,0>[{e}`{b};{p}]
  \morphism(98,-300)/{@{=}@<-15pt>}/<180,0>[{\phantom{O}}`{\phantom{O}};]
}
\defdiag{semantic_factorization_of_p_definition_dednat}{   
  \morphism(0,0)|a|/->/<1800,0>[{e}`{b};{p}]
  \morphism(0,0)|l|/->/<900,-375>[{e}`{b^\t};{p^\t}]
  \morphism(900,-375)|r|/->/<900,375>[{b^\t}`{b};{\uu^\t}]
}
\defdiag{imageofAx_semantic_factorization_of_p_definition_dednat}{   
  \morphism(0,0)|a|/->/<1800,0>[{\AAA\left({x},e\right)}`{\AAA\left({x},b\right)};{\AAA\left(x,p\right)}]
  \morphism(0,0)|m|/->/<900,-375>[{\AAA\left({x},e\right)}`{\AAA\left({x},b\right)^{\AAA\left({x},\t\right)}};{\AAA\left({x},p\right)^{\left(\AAA\left(x,\t\right){,}\AAA\left({x},\gamma\right)\right)}}]
  \morphism(900,-375)|m|/->/<900,375>[{\AAA\left({x},b\right)^{\AAA\left({x},\t\right)}}`{\AAA\left({x},b\right)};{\uu^{\AAA\left(x,\t\right)}}]
}
\defdiag{first_diagram_triangle_identity}{   
  \morphism(0,0)/=/<0,-600>[{b}`{b};]
  \morphism(600,0)/=/<0,-600>[{e}`{e};]
  \morphism(0,-600)|b|/->/<600,0>[{b}`{e};{l}]
  \morphism(0,0)|a|/->/<600,0>[{b}`{e};{l}]
  \morphism(600,0)|m|/->/<-600,-600>[{e}`{b};{p}]
  \morphism(8,-0)|a|/{@{=>}@<-20pt>}/<405,0>[{\phantom{O}}`{\phantom{O}};{\eta}]
  \morphism(188,-600)|a|/{@{=>}@<20pt>}/<405,0>[{\phantom{O}}`{\phantom{O}};{\varepsilon}]
}
\defdiag{second_diagram_triangle_identity}{   
  \morphism(0,0)/=/<0,-600>[{b}`{b};]
  \morphism(600,0)/=/<0,-600>[{e}`{e};]
  \morphism(0,-600)|b|/<-/<600,0>[{b}`{e};{p}]
  \morphism(0,0)|a|/<-/<600,0>[{b}`{e};{p}]
  \morphism(600,0)|m|/<-/<-600,-600>[{e}`{b};{l}]
  \morphism(8,-0)|a|/{@{=>}@<-20pt>}/<405,0>[{\phantom{O}}`{\phantom{O}};{\eta}]
  \morphism(188,-600)|a|/{@{=>}@<20pt>}/<405,0>[{\phantom{O}}`{\phantom{O}};{\varepsilon}]
}
\defdiag{usualfactorizationKleisli}{   
  \morphism(0,0)|a|/->/<1800,0>[{e}`{b};{l}]
  \morphism(0,0)|l|/->/<900,-375>[{e}`{b_{\left(pl,\id_p\ast\varepsilon\ast\id_l,\eta\right)}};{\mathbbmss{l}_{\left(pl,\id_p\ast\varepsilon\ast\id_l,\eta\right)}}]
  \morphism(900,-375)|r|/->/<900,375>[{b_{\left(pl,\id_p\ast\varepsilon\ast\id_l,\eta\right)}}`{b};{l_{\t}}]
}
\defdiag{firstsideoftheequation_definition_preservation_of_Kan_Extensions}{   
  \morphism(0,0)|a|/->/<750,-750>[{x}`{y};{\ran_gf}]
  \morphism(0,-750)|l|/->/<0,750>[{z}`{x};{g}]
  \morphism(0,-750)|b|/->/<750,0>[{z}`{y};{f}]
  \morphism(750,-750)|b|/->/<600,0>[{y}`{y'};{\delta}]
  \morphism(375,-360)|l|/{@{=>}@<-15pt>}/<0,-375>[{\phantom{O}}`{\phantom{O}};{\gamma^{\ran_gf}}]
}
\defdiag{First_Mate_Definition}{   
  \morphism(0,0)/=/<0,-450>[{b_0}`{b_0};]
  \morphism(0,-450)|l|/->/<0,-450>[{b_0}`{b_1};{h_b}]
  \morphism(450,-450)/=/<0,-450>[{e_1}`{e_1};]
  \morphism(450,0)|r|/->/<0,-450>[{e_0}`{e_1};{h_e}]
  \morphism(0,0)|a|/->/<450,0>[{b_0}`{e_0};{l_0}]
  \morphism(0,-900)|b|/->/<450,0>[{b_1}`{e_1};{l_1}]
  \morphism(450,0)|m|/->/<-450,-450>[{e_0}`{b_0};{p_0}]
  \morphism(450,-450)|m|/->/<-450,-450>[{e_1}`{b_1};{p_1}]
  \morphism(0,-225)|a|/{@{=>}@<12pt>}/<300,0>[{\phantom{O}}`{\phantom{O}};{\eta_0}]
  \morphism(150,-675)|a|/{@{=>}@<-13pt>}/<300,0>[{\phantom{O}}`{\phantom{O}};{\varepsilon_1}]
  \morphism(75,-450)|a|/=>/<300,0>[{\phantom{O}}`{\phantom{O}};{\beta}]
}
\defdiag{Twocell_Mate_Definition}{   
  \morphism(450,0)|r|/->/<0,-450>[{e_0}`{e_1};{h_e}]
  \morphism(0,-450)|l|/->/<0,-450>[{b_0}`{b_1};{h_b}]
  \morphism(450,0)|m|/->/<-450,-450>[{e_0}`{b_0};{p_0}]
  \morphism(450,-450)|m|/->/<-450,-450>[{e_1}`{b_1};{p_1}]
  \morphism(75,-450)|a|/=>/<300,0>[{\phantom{O}}`{\phantom{O}};{\beta}]
}
\defdiag{Second_Mate_Definition}{   
  \morphism(0,0)/=/<0,-450>[{e_0}`{e_0};]
  \morphism(0,-450)|l|/->/<0,-450>[{e_0}`{e_1};{h_e}]
  \morphism(450,-450)/=/<0,-450>[{b_1}`{b_1};]
  \morphism(450,0)|r|/->/<0,-450>[{b_0}`{b_1};{h_b}]
  \morphism(0,0)|a|/->/<450,0>[{e_0}`{b_0};{p_0}]
  \morphism(0,-900)|b|/->/<450,0>[{e_1}`{b_1};{p_1}]
  \morphism(450,0)|m|/->/<-450,-450>[{b_0}`{e_0};{l_0}]
  \morphism(450,-450)|m|/->/<-450,-450>[{b_1}`{e_1};{l_1}]
  \morphism(0,-225)|a|/{@{<=}@<12pt>}/<300,0>[{\phantom{O}}`{\phantom{O}};{\varepsilon_0}]
  \morphism(150,-675)|a|/{@{<=}@<-13pt>}/<300,0>[{\phantom{O}}`{\phantom{O}};{\eta_1}]
  \morphism(75,-450)|a|/<=/<300,0>[{\phantom{O}}`{\phantom{O}};{\beta{'}}]
}
\defdiag{SecondTwocell_Mate_Definition}{   
  \morphism(450,0)|r|/->/<0,-450>[{b_0}`{b_1};{h_e}]
  \morphism(0,-450)|l|/->/<0,-450>[{e_0}`{e_1};{h_b}]
  \morphism(450,0)|m|/->/<-450,-450>[{b_0}`{e_0};{l_0}]
  \morphism(450,-450)|m|/->/<-450,-450>[{b_1}`{e_1};{l_1}]
  \morphism(75,-450)|a|/<=/<300,0>[{\phantom{O}}`{\phantom{O}};{\beta{'}}]
}
\defdiag{firstequationwrtmaintheoremscodensity_monada_Kanextensionopcommaobjects}{   
  \morphism(300,0)|a|/->/<-300,-300>[{e}`{b};{p}]
  \morphism(300,0)|a|/->/<300,-300>[{e}`{b};{p}]
  \morphism(0,-300)|l|/->/<300,-300>[{b}`{b\uparrow_pb};{\d^1}]
  \morphism(600,-300)|r|/->/<-300,-300>[{b}`{b\uparrow_pb};{\d^0}]
  \morphism(300,-600)|l|/->/<0,-300>[{b\uparrow_pb}`{b};{\ell}]
  \morphism(112,-300)|a|/=>/<375,0>[{\phantom{O}}`{\phantom{O}};{\upalpha}]
}
\defdiag{firstequationwrtmaintheoremscodensity_monada_Kanextension}{   
  \morphism(0,0)|a|/->/<750,-750>[{b}`{b};{t}]
  \morphism(0,-750)|l|/->/<0,750>[{e}`{b};{p}]
  \morphism(0,-750)|b|/->/<750,0>[{e}`{b};{p}]
  \morphism(375,-345)|l|/{@{=>}@<-15pt>}/<0,-405>[{\phantom{O}}`{\phantom{O}};{\gamma}]
}
\defdiag{basicdiagramthatcommutesforthedefinitionofella}{   
  \morphism(0,0)|a|/->/<600,0>[{b}`{b\uparrow_pb};{\d^0}]
  \morphism(0,0)|l|/->/<0,-600>[{b}`{b\uparrow_pb};{\d^1}]
  \morphism(0,-600)|b|/->/<600,0>[{b\uparrow_pb}`{b\uparrow_pb\uparrow_pb};{\D^0}]
  \morphism(600,0)|r|/->/<0,-600>[{b\uparrow_pb}`{b\uparrow_pb\uparrow_pb};{\D^2}]
  \morphism(0,-600)|b|/{@{->}@/_25pt/}/<1050,-450>[{b\uparrow_pb}`{b};{\ell}]
  \morphism(600,0)|r|/{@{->}@/^25pt/}/<450,-1050>[{b\uparrow_pb}`{b};{t\ell}]
  \morphism(600,-600)|a|/->/<450,-450>[{b\uparrow_pb\uparrow_pb}`{b};{\ella}]
}
\defdiag{betaumofthemainproof}{   
  \morphism(375,0)|a|/->/<-375,-375>[{x}`{b};{h}]
  \morphism(375,0)|a|/->/<375,-375>[{x}`{b};{h}]
  \morphism(0,-375)|m|/->/<375,-375>[{b}`{b\uparrow_pb};{\d^1}]
  \morphism(750,-375)|m|/->/<-375,-375>[{b}`{b\uparrow_pb};{\d^0}]
  \morphism(375,-750)|m|/->/<0,-450>[{b\uparrow_pb}`{b\uparrow_pb\uparrow_pb};{\D^1}]
  \morphism(0,-375)|l|/{@{->}@/_25pt/}/<375,-825>[{b}`{b\uparrow_pb\uparrow_pb};{\D^2\d^1}]
  \morphism(750,-375)|r|/{@{->}@/^25pt/}/<-375,-825>[{b}`{b\uparrow_pb\uparrow_pb};{\D^0\d^0}]
  \morphism(188,-375)|a|/=>/<375,0>[{\phantom{O}}`{\phantom{O}};{\beta}]
  \morphism(510,-788)/=/<180,0>[{\phantom{O}}`{\phantom{O}};]
  \morphism(60,-788)/=/<180,0>[{\phantom{O}}`{\phantom{O}};]
}
\defdiag{betacofthemainproof}{   
  \morphism(750,0)|a|/->/<-750,-480>[{x}`{b};{h}]
  \morphism(750,0)|a|/->/<750,-480>[{x}`{b};{h}]
  \morphism(750,0)|m|/->/<0,-480>[{x}`{b};{h}]
  \morphism(750,-480)|m|/->/<0,-480>[{b}`{b\uparrow_pb};{\d^0}]
  \morphism(750,-480)|m|/->/<750,-480>[{b}`{b\uparrow_pb};{\d^1}]
  \morphism(1500,-960)|b|/->/<-750,-480>[{b\uparrow_pb}`{b\uparrow_pb\uparrow_pb};{\D^0}]
  \morphism(1500,-480)|r|/->/<0,-480>[{b}`{b\uparrow_pb};{\d^0}]
  \morphism(0,-480)|l|/->/<750,-480>[{b}`{b\uparrow_pb};{\d^1}]
  \morphism(750,-960)|l|/->/<0,-480>[{b\uparrow_pb}`{b\uparrow_pb\uparrow_pb};{\D^2}]
  \morphism(225,-480)|a|/=>/<300,0>[{\phantom{O}}`{\phantom{O}};{\beta}]
  \morphism(975,-480)|a|/=>/<300,0>[{\phantom{O}}`{\phantom{O}};{\beta}]
  \morphism(1035,-960)/=/<180,0>[{\phantom{O}}`{\phantom{O}};]
}
\defdiag{algebra_rightsideoftheequationassociativityofmonadforthemainproof}{   
  \morphism(600,0)|r|/->/<0,-600>[{b}`{b};{t}]
  \morphism(0,0)|a|/->/<600,0>[{b}`{b};{t}]
  \morphism(0,-600)|l|/->/<0,600>[{x}`{b};{h}]
  \morphism(0,-600)|b|/->/<600,0>[{x}`{b};{h}]
  \morphism(0,-600)|m|/->/<600,600>[{x}`{b};{h}]
  \morphism(150,-0)|l|/{@{=>}@<10pt>}/<0,-375>[{\phantom{O}}`{\phantom{O}};{\id_\ell\ast\beta}]
  \morphism(450,-225)|r|/{@{=>}@<-10pt>}/<0,-375>[{\phantom{O}}`{\phantom{O}};{\id_\ell\ast\beta}]
}
\defdiag{algebra_leftsideoftheequationassociativityofmonadforthemainproof}{   
  \morphism(600,0)|r|/->/<0,-600>[{b}`{b};{t}]
  \morphism(0,-600)|b|/->/<600,0>[{y}`{b};{h}]
  \morphism(0,0)|m|/->/<600,-600>[{b}`{b};{t}]
  \morphism(0,-600)|l|/->/<0,600>[{y}`{b};{h}]
  \morphism(0,0)|a|/->/<600,0>[{b}`{b};{t}]
  \morphism(450,-0)|r|/{@{=>}@<-10pt>}/<0,-375>[{\phantom{O}}`{\phantom{O}};{m}]
  \morphism(150,-225)|l|/{@{=>}@<10pt>}/<0,-375>[{\phantom{O}}`{\phantom{O}};{\id_\ell\ast\beta}]
}

\def\pu{}
\fi


\title{Semantic Factorization and Descent} 
\author{Fernando Lucatelli Nunes}
\address{CMUC, Centre for Mathematics, University of Coimbra, Portugal \& \\
	Department of Information and Computing Sciences, Utrecht University, The Netherlands
 }
\eaddress{f.lucatellinunes@uu.nl}
\amsclass{18N10, 18Cxx, 18Dxx, 18A22, 18A30, 18A40,  18A25}


\thanks{This research was partially supported by the Institut de Recherche en Math\'{e}matique et Physique (IRMP, UCLouvain, Belgium), and by the Centre for Mathematics of the University of Coimbra - UIDB/00324/2020, funded by the Portuguese Government through FCT/MCTES. \\	
This work was also supported through the programme ``Oberwolfach Leibniz Fellows'' by the Mathematisches Forschungsinstitut Oberwolfach in 2022.}

\maketitle 

\begin{abstract}
Let $\AAA $ be a $2$-category with suitable opcomma objects and pushouts.
We give a direct proof that, provided that the codensity monad of a morphism $p$ exists and
is preserved by a suitable morphism, the factorization given by the lax descent object of the \textit{two-dimensional cokernel diagram} of $p$ is up to isomorphism the same 
as the semantic factorization of $p$, either one existing if the other does. The result can be seen as
a counterpart account to the celebrated
B\'{e}nabou-Roubaud theorem. This leads in particular to a monadicity theorem, since it characterizes monadicity via descent.
It should be
noted that all the conditions on the codensity monad of $p$ trivially hold
whenever $p$ has a left adjoint and, hence, in this case, we find monadicity to be a two-dimensional exact
condition on $p$, namely, to be an effective faithful morphism of the $2$-category $\AAA $. 
\end{abstract}

\tableofcontents  
\setcounter{secnumdepth}{-1}
\section{Introduction}
\textit{Grothendieck descent theory}~\cite{MR1603475} has been generalized
from a solution of the problem of understanding the image of the functors  $\Mod (f) $ in which $\Mod :\Ring\to \Cat $ is the usual pseudofunctor between the category of rings and the $2$-category of categories that associates each ring $ \RRRR $ with the category $\Mod ( \RRRR )$ of right $\RRRR $-modules (\textit{e.g.} \cite{MR2107397}).

It is often more descriptive to portray \textit{descent theory} as a higher dimensional counterpart of \textit{sheaf theory} (see, for instance, the introduction of \cite{MR1285884}). In this context, 
the analogy can be roughly stated  as follows:  the \textit{descent condition} and the \textit{descent data}
are respectively two-dimensional counterparts of the \textit{sheaf condition} and the \textit{gluing condition}.

The most fundamental constructions in descent theory are the lax descent category and its variations
(\textit{e.g.} \cite[pag.~177]{MR0401868}). Namely, given
a truncated pseudocosimplicial category
$$ \AAAA : \Delta _\mathrm{3}\to \Cat $$
%
\pu
\begin{equation*}
\diag{truncatedpseudocosimplicialcategory}
\end{equation*}
we construct its \textit{lax descent category} or \textit{descent category}. An object of the lax descent category (descent category) is an object $x$ of the category $\AAAA (\mathsf{1} )$
endowed with a descent data which is a morphism (respectively, invertible morphism) $\AAAA (d^1)(x)\to \AAAA (d^0)(x) $ satisfying the usual
\textit{cocycle/associativity} and \textit{identity} conditions. 
Morphisms are morphisms between the underlying objects in $\AAAA (\mathsf{1} )$ that respect the \textit{descent data}.

Another perspective, which highlights descent theory's main role in Janelidze-Galois theory,
is that, given a bifibred category, the lax descent category of the truncated pseudocosimplicial category 
induced by an internal category generalizes the notion of the category of \textit{internal (pre)category actions} (\textit{e.g.} \cite[Section~1]{MR1466540}). 

In the setting above, if the bifibration is the basic one, we actually get the notion 
of internal actions. The simplest example is the category of actions of a small category 
in $\Set $; that is to say, the category of functors from a small category into $\Set $. A small category $a $ is just an internal category in $\Set $ and the category of actions 
(functors) $a\to \Set $
coincides with the lax descent category of the composition of the (image by $\op : \Cat ^\mathrm{co} \to \Cat$ of the) internal category $a$, $\op (a): \Delta _\mathrm{3} \to \Set ^\op $, with
the pseudofunctor $\Set / - :\Set ^\op \to \Cat  $ that comes from the basic fibration.

Assume that we have a pseudofunctor  $\FFFF : \mathbb{C} ^\op\to \Cat $ 
such that $\mathbb{C}$ has pullbacks, and $\FFFF (q)!\dashv\FFFF (q) $ for every morphism $q $ of $\mathbb{C}$.
Given a morphism $q : w\to w' $ of $\mathbb{C}$, the B\'{e}nabou-Roubaud theorem
(see \cite{MR0255631} or, for instance, 
\cite[Theorem~1.4]{2016arXiv160604999L}) says that the (lax) descent category
of the truncated pseudocosimplicial category
%
\pu
\begin{equation*}
\diag{truncatedpseudocosimplicialcategoryinducedbyamorphisminabifibredcategory}
\end{equation*}
given by the composition of $\FFFF $ with the internal groupoid induced by $q $, 
is equivalent to the Eilenberg-Moore category of the monad induced by the adjunction $\FFFF (q)!\dashv\FFFF (q) $, 
provided that $\FFFF $ satisfies the so called Beck-Chevalley condition (see, for instance, the Beck-Chevalley condition for  pseudofunctors in \cite[Section~4]{2019arXiv190600517L}).

Since \textit{monad theory} already was a established 
subfield of category theory, the B\'{e}nabou-Roubaud theorem gave
an insightful connection between the theories,
motivating what is nowadays often called \textit{monadic approach to descent} by  giving a 
characterization of  \textit{descent via 
monadicity} in several cases of interest
(see, for instance, \cite[Section~1]{MR1271335}, \cite[Section~2]{MR1285884}, or  the introduction of \cite{2016arXiv160604999L}).

The main contribution of the present article can be seen as a 
counterpart account to the  B\'{e}nabou-Roubaud theorem. We give the \textit{semantic factorization} via descent, hence giving, in particular, a characterization of 
\textit{monadicity via descent}. Although the B\'{e}nabou-Roubaud theorem is originally a result 
in the setting of the $2$-category  $\Cat $,
our contribution takes place in 
the more general context of \textit{two-dimensional category theory} (\textit{e.g.} \cite{MR0357542}), or in the so called
\textit{formal category theory}, as briefly explained below.

In his pioneering work on bicategories, B\'{e}nabou observed that the notion of \textit{monad}, formerly called \textit{standard construction} or \textit{triple}, coincides with the notion of a lax
functor $\mathsf{1}\to\Cat $ and can be pursued in any bicategory, giving convincing examples to the
generalization of the notion~\cite[Section~5]{MR0220789}. 

Taking B\'{e}nabou's point in consideration, Street~\cite{MR0299653, MR0347936} gave a formal account and generalization of the former established theory of monads by developing   
the theory within the general setting of $2$-categories. \textit{The formal theory of monads} is a celebrated example of how two-dimensional category theory can give insight to
$1$-dimensional category theory, since, besides generalizing several notions, it 
conceptually enriches the formerly established theory of monads. 
Street~\cite{MR0299653} starts showing that, when it exists, the 
Eilenberg-Moore construction of a monad in a $2$-category $\mathbb{A}$ 
is given by a right $2$-reflection of the monad along a $2$-functor from the $2$-category $\mathbb{A} $ into the $2$-category
of monads in $\mathbb{A} $. From this point, making good use of 
the four dualities of \textit{two-dimensional category theory}, Street develops the formal account of aspects
of monad theory, including distributive laws, Kleisli construction,   and a generalization of 
the \textit{semantics-structure adjunction} (\textit{e.g.} \cite[Chapter~II]{MR0280560}).

The theory of \textit{two-dimensional limits} (\textit{e.g.} \cite{MR0401868, MR998024}), or weighted limits in $2$-categories, also provides a great
account of formal category theory, since it shows that several constructions previously introduced in
$1$-dimensional category theory
 are actually examples of weighted limits and, hence, are universally defined and can be pursued in the
general context of a $2$-category.

Examples of the constructions that are particular weighted limits are: the lax descent category
and variations, the Eilenberg-Moore category~\cite[Theorem~2.2]{MR0184984} and the comma category~\cite[pag.~36]{LawvereThesis}. Duality also plays important role in this context: it 
usually illuminates or expands the original concepts of $1$-dimensional category theory. For instance: 
\begin{itemize}
\renewcommand\labelitemi{--}
\item The dual of the notion of descent object gives
the notion of codescent object, which is important, for instance, in $2$-dimensional monad theory (see, for instance, \cite{MR1935980, MR3491845, 2016arXiv160703087L});
\item The
dual notion of the \textit{Eilenberg-Moore object} in $\Cat $ gives the Kleisli category~\cite{MR0177024} of a monad, while the codual gives the category of the coalgebras of a comonad.
\end{itemize}

Despite receiving less attention in the literature than the notion of \textit{comma object}, the dual notion, 
called \textit{opcomma object}, was already considered in
\cite[pag.~109]{MR0354813} and it is essential to the present work.
More precisely, given a morphism $p: e\to b $ of a $2$-category $\AAA $, if $\AAA $ has suitable opcomma objects and pushouts,
on one hand, we can consider the \textit{two-dimensional cokernel diagram}  
$$\mathcal{H}_p: \Delta _\textrm{Str}\to \AAA $$
%
\pu
\begin{equation*}
\diag{twodimensionalcokerneldiagram}
\end{equation*}
of $p$, whose precise definition can be found in \ref{definicaodehighercokernel} below. Assuming that $\AAA $ has the lax descent object of $\H_p $, the universal property of $\mathrm{lax}\textrm{-}\mathcal{D}\mathrm{esc}\left(\mathcal{H}_p\right)$ and the universal $2$-cell of the opcomma object induce a factorization
%
\pu
\begin{equation}\label{eq:lax-descentfactorization-introduction}
	\diag{Semanticdescentfactorizationofp}\tag{$\mathsf{SLDF}$}
\end{equation} 
of $p$, called herein the \textit{semantic lax descent factorization}. If the comparison morphism $e\to \mathrm{lax}\textrm{-}\mathcal{D}\mathrm{esc}\left(\mathcal{H}_p \right) $ is an equivalence, we say that $p$ is an \textit{effective faithful morphism}. This concept is actually self-codual, meaning that its codual notion
 coincides with the original one.

On the other hand, if such a morphism $p$ has a \textit{codensity monad} $\t $  which means that the right Kan extension of $p$ along itself exists in $\AAA $, then the universal $2$-cell of the $\ran _p p $ induces the  \textit{semantic factorization}
%
\pu
\begin{equation}\label{eq:semantic-introduction}
	\diag{SemanticFactorizationp}\tag{$\mathsf{SF}$}
\end{equation} 
through the Eilenberg-Moore object $b^\t $ of $\t $ provided that it exists (see, for instance,  
\cite[pag.~67]{MR0280560} for the case of the $2$-category of enriched categories). If the comparison $e\to b ^\t $ is an equivalence,
we say that $p$ is \textit{monadic}. The codual notion is that of \textit{comonadicity}. 

The main theorem of the present article concerns both the factorizations above. More precisely, Theorem \ref{maintheoremofthepaper} states the following:

\begin{inmtheorem}
Let $\AAA $ be a $2$-category which has the two-dimensional cokernel diagram of a morphism $p$. Moreover, assume that $\ran _p p $ exists and is preserved
by the universal morphism $\delta _ {p\uparrow p} ^0   $ of the opcomma object $ b\uparrow _p b$. 

There is an isomorphism between the Eilenberg-Moore object $b ^\t $ and the lax descent object $\mathrm{lax}\textrm{-}\mathcal{D}\mathrm{esc}\left(\mathcal{H}_p \right)$, either one existing if the other does. In this case, the semantic factorization \eqref{eq:semantic-introduction} is isomorphic to the semantic lax descent factorization \eqref{eq:lax-descentfactorization-introduction}.
\end{inmtheorem}

In particular, this gives a \textit{formal monadicity theorem} as a corollary, since it shows that, 
assuming that a morphism $p$ of $\AAA $ satisfies the conditions above on the codensity monad, $p$ is monadic if and only if $p$ is an effective faithful morphism. 
Moreover, since this result holds for any $2$-category, we can consider the duals of this formal monadicity theorem: namely, we also get characterizations
of comonadic, Kleisli and co-Kleisli morphisms.

By the \textit{Dubuc-Street formal adjoint-functor theorem} (\textit{viz.}, \cite[Theorem~I.4.1]{MR0280560} and \cite[Prop.~2]{MR0463261}), if $p $ has a left adjoint, 
the codensity monad is the monad induced by the adjunction and $\ran _p p $ is absolute. Thus, in this case,
assuming the existence of the two-dimensional cokernel diagram, 
the hypothesis of our theorem holds. Therefore, as a corollary of our main result, we get the following monadicity characterization:
\begin{intheorem}
Assume that the $2$-category $\AAA $ has the two-dimensional cokernel diagram of $p: e\to b$.
\begin{itemize}
\renewcommand\labelitemi{--}
\item The morphism $p$  is monadic if and only if $p$ is an effective faithful morphism and has a left adjoint;
\item The morphism $p$ is comonadic if and only if $p$ is an effective faithful morphism and has a right adjoint.
\end{itemize}
\end{intheorem}

Recall that, in the particular case of $\AAA = \Cat $ (and other $2$-categories,
such as the $2$-category of enriched categories), we have  
 Beck's monadicity theorem (\textit{e.g.} \cite[Theorem~3.14]{MR2178101} and \cite[Theorem~II.2.1]{MR0280560}). It states that:  a functor is monadic if and only if it creates absolute coequalizers and it has a left adjoint. Hence, by our main result, we can conclude that: provided that the functor $p$ has a left adjoint, $p$ creates absolute coequalizers if and only if it is an effective faithful morphism.

The fact
above suggests the following question: are effective faithful morphisms in $\Cat $ characterized by the property 
of creating absolute coequalizers? In Remark \ref{whyitisnotthesamething} we show that the answer to this question is negative by the self coduality
of the concept of effective faithful morphism and non-self duality of the concept of functor that creates  absolute
coequalizers.

This work was motivated by two main aims. Firstly, 
to 
get a formal monadicity 
theorem given by a \textit{$2$-dimensional exact condition}. 
Secondly, to better understand the relation between descent and monadicity
in a given $2$-category and, together with \cite{2016arXiv160604999L}, get alternative
guiding templates for the development of higher descent theory and monadicity.  

Although we do not make these connections in this paper, 
the results on $2$-dimensional category theory of the present work already establish framework and have applications to 
 the author's ongoing work on 
descent theory in the context of \cite{MR1285884, 2016arXiv160604999L}.

The main aim of Section \ref{sectioncategoriesdelta} is to set up basic terminology related to the category of the finite nonempty ordinals
$\Delta $ and its strict replacement $\Delta _ \mathrm{Str} $. As observed above, this work is meant to be applicable in the 
classical context of descent theory and, hence, we should consider lax descent categories of pseudofunctors $\Delta _\mathrm{3}\to \Cat $.
In order to do so, we consider suitable strict replacements $\Delta _ \mathrm{Str}\to \Cat $.

The main results (Theorem \ref{principal} and Theorem \ref{maintheoremofthepaper}) can be seen as theorems on $2$-dimensional limits and colimits. 
For this reason, we recall basics on  $2$-dimensional limits in Section \ref{2dimensionallimitscolimits}. 
We give an explicit definition of the $2$-dimensional limits 
 related to the two-dimensional cokernel diagram. This helps to establish terminology and framework for 
the rest of the paper. 

In \ref{laxdescentobjectsdetails}, we give  an explicit definition of the lax descent object for $2$-functors
$\Delta _ \mathrm{Str}\to \AAA $  in order to 
establish the lax descent factorization induced by the two-dimensional cokernel diagram \eqref{laxdescentfactorizationofthehighercokernel} of a morphism $p$, the \textit{semantic lax descent factorization of $p$}. This perspective over lax descent objects 
is also useful to future work on 
giving further applications of the results of the present paper in \textit{Grothendieck descent theory} within the context of \cite{MR1285884, 2016arXiv160604999L}.

In Section \ref{semanticandfurther}, we recall basic aspects of Eilenberg-Moore objects in a $2$-category $\AAA $. Given a tractable
morphism $p$ in $\AAA $, it induces a monad and, in the presence of the Eilenberg-Moore objects, it also induces a factorization,
the \textit{semantic factorization} of $p$ (\textit{e.g.} \cite[Section~2]{MR0299653} or \cite[Theorem~II.1.1]{MR0280560}). We are only interested
in morphisms $p$ that have codensity monads, that is to say, the right Kan extension of $p$ along itself. We recall the
basics of this setting, including the definitions of right Kan extensions and codensity monads in Section \ref{semanticandfurther}.

We do not present more than very basic toy examples of codensity monads.
We refer to \cite[Chapter~II]{MR0280560} for the classical theory on codensity monads, while \cite{MR3856919, MR3080612} are recent considerations that
can be particularly useful to understand interesting examples.

Still in Section \ref{semanticandfurther}, Lemma \ref{teoremadeultimomomento} gives a connection between opcomma objects and right Kan extensions. The statement 
is particularly useful for establishing an important adjunction (see Propositions \ref{rightkanextensionoftheidentityalongd0} and \ref{conditionofpreservation}) and proving the main results.
The Dubuc-Street formal adjoint-functor theorem, also important for our proofs,
is recalled in Theorem \ref{DUBUCSTREET}. 

The \textit{mate correspondence}~\cite[Proposition~2.1]{MR0357542} is a 
useful framework in $2$-dimensional category theory
that states an isomorphism between two special double categories that come from each $2$-category $\AAA $. It plays a central role in the proof of Theorem \ref{principal}, but we only need it in very basic terms
as recalled in Remark \ref{Mate Correspondence Theorem PhD}, with which we finish Section \ref{semanticandfurther}. 

The point of Section \ref{MainTheoremsSection} is to prove the main results of the present paper.  We start by establishing an important condition to the main theorem, which arises from Propositions \ref{rightkanextensionoftheidentityalongd0}
and \ref{conditionofpreservation}: the 
condition of preservation of $\ran _ p p $ by the universal morphism $\d^0_ {p\uparrow p}  $. We, then, go towards the proof
of the main result,  constructing an adjunction  in Proposition \ref{definitionella}, defining particularly useful $2$-cells for
our proof in Lemma \ref{important2cells} and, finally, proving Theorem \ref{principal}. 

We also give a brief discussion on the condition of Proposition \ref{conditionofpreservation}. Firstly,
we show that every right adjoint morphism satisfies the condition in Proposition \ref{okforrightadjoints}. Then
we give examples and counterexamples in \ref{toyexamples}. 

The final section is mostly intended to apply our main result 
in order to get our monadicity theorem using the concept of \textit{effective faithful morphism}. We finish the article with a remark
on the self-coduality of this concept, in opposition to the non-self duality of the property of creating absolute coequalizers. This gives a comparison between the Beck's monadicity theorem and ours, showing in particular that effective faithful morphisms in $\Cat$ are not characterized by the property of creating absolute coequalizers.

\setcounter{secnumdepth}{5}

\section{Categories of ordinals}\label{sectioncategoriesdelta}

Let $\Cat $ be the cartesian closed category of categories in some universe. We denote the internal hom  by 
$$\Cat[-,-]: \Cat ^\op\times \Cat\to \Cat .$$ 

A $2$-category $\AAA $ herein is the same as a $\Cat $-enriched category.
As usual, 
the composition of $1$-cells (morphisms) is denoted
by $\circ $, $\cdot $, or omitted whenever it is clear from the context. 
The vertical composition of $2$-cells is denoted by $\cdot $  or omitted when it is clear,  while the horizontal composition is denoted
by $\ast$. Recall that, from the vertical and horizontal compositions, we construct the basic operation of \textit{pasting} (see \cite[pag.~79]{MR0357542} and \cite{MR1040947}).

As mentioned in the introduction, duality is one of the most fundamental
aspects of theories on $2$-categories. Unlike $1$-dimensional category theory, two-dimensional category theory 
has four duals. More precisely,
any $2$-category $\AAA $ gives rise to four $2$-categories:  $\AAA $, $\AAA ^\op $,  $\AAA ^\co $, $\AAA ^{\coop } $ which are respectively related 
to \textit{inverting the directions} of nothing, morphisms, $2$-cells, morphisms and $2$-cells. Hence every concept/result gives rise to four (not necessarily different) 
duals: the concept/result itself, the dual, the codual, the codual of the dual.

Although it is important to keep in mind the importance of duality, 
we usually leave to the interested reader the straightforward exercise of stating precisely the four duals of
most of the dualizable aspects of the present work.

In this section, we fix notation related to the categories of ordinals and the strict replacement $\Delta _\mathrm{Str} $. 
We denote by $\Delta  $ the locally discrete  $2$-category of finite nonempty ordinals and order preserving functions between them. Recall that 
$\Delta $ is generated
by  the degeneracy and face maps; 
that is to say, $\Delta  $ is generated by the diagram
$$\xymatrix{  \mathsf{1}\ar@<2ex>[rr]|-{d^0}\ar@<-2ex>[rr]|-{d^1} && \mathsf{2}\ar[ll]|-{s^0}
\ar@<2 ex>[rr]|-{d ^0}\ar[rr]|-{d ^1}\ar@<-2ex>[rr]|-{d ^2} && \mathsf{3}\ar@/_4ex/@<-2 ex>[ll]|-{s^0}\ar@/^4ex/@<2ex>[ll]|-{s^1}\ar@<1.5ex>[rr]\ar@<0.5ex>[rr]\ar@<-0.5ex>[rr]\ar@<-1.5ex>[rr]&& \cdots\ar@/^2ex/@<1.5ex>[ll]\ar@/^5ex/@<1.5ex>[ll]\ar@/_3ex/@<-1.5 ex>[ll] } $$
with the following relations: 
\begin{equation*}
\begin{aligned}
d ^k d ^i&=& d^{i}d^{k-1}, &\hspace{1mm}\mbox{if  }\hspace{1mm} i<k ;& \\
s^ks^i &=& s^is^{k+1}, &\hspace{1mm}\mbox{if  }\hspace{1mm} i\leq k ;&  \\
s^k d^i &=& d^i s^{k-1}, &\hspace{1mm}\mbox{if  }\hspace{1mm} i< k ;&
\end{aligned}
\qquad\qquad\qquad
\begin{aligned}
s^k d^i &= \id ,\hspace{1mm}\mbox{if  }\hspace{1mm} i=k\mbox{ or } i=k+1 ; \\
s^k d^i &= d^{i-1}s^k, \hspace{1mm}\mbox{if  }\hspace{1mm} i>k+1 . 
\end{aligned}
\end{equation*}
We are particularly interested in the sub-$2$-category $\Delta _\mathrm{3} $ of $\Delta $
with the objects $\mathsf{1}, \mathsf{2} $ and $\mathsf{3} $ 
generated by the morphisms below.
$$\xymatrix{  \mathsf{1}\ar@<2ex>[rr]|-{d^0}\ar@<-2ex>[rr]|-{d^1} && \mathsf{2}\ar[ll]|-{s^0}
\ar@<2 ex>[rr]|-{d ^0}\ar[rr]|-{d ^1}\ar@<-2ex>[rr]|-{d ^2} && \mathsf{3} } $$
For simplicity, we use the same notation to the objects and morphisms of $\Delta $ and their images by the 
usual inclusion $\Delta\to\Cat $ which is locally bijective on objects. It should be noted that the image  
of the faces and degeneracy maps by $\Delta\to\Cat $ are given by:
\begin{equation*}
\begin{aligned}
d ^k : &\mathsf{n}-\mathsf{1} &\to &\hskip 1em\mathsf{n} & \\
       &\mathsf{t} &\mapsto &
			\begin{dcases}
    		\mathsf{t} + \mathsf{1},& \text{if } t\geq k\\
    \mathsf{t} ,              & \text{otherwise}
\end{dcases} &
\end{aligned}
\qquad\qquad\qquad
\begin{aligned}
s ^k : &\mathsf{n}+\mathsf{1} &\to &\hskip 1em\mathsf{n} & \\
    &\mathsf{t} &\mapsto &
			\begin{dcases}
    \mathsf{t},& \text{if } t\leq k\\
    \mathsf{t} - \mathsf{1},              & \text{otherwise.}
\end{dcases} & 
\end{aligned}
\end{equation*}
Furthermore, in order to give the weight of the \textit{lax descent object},
we consider
the $2$-category $\Delta _ \mathrm{Str} $.

\begin{defi}[$\Delta _\mathrm{Str} $]\label{deltaestritodescidasecao1}
We denote by $\Delta _\mathrm{Str} $ the $2$-category freely generated by the diagram
\[\xymatrix{ \umm\ar@<1.7 ex>[rrr]^-{\dd ^0}\ar@<-1.7ex>[rrr]_-{\dd ^1} &&& \doiss\ar[lll]|-{\ss ^0}
\ar@<1.7 ex>[rrr]^{\DD ^0}\ar[rrr]|-{\DD ^1}\ar@<-1.7ex>[rrr]_{\DD ^2} &&& \tress }\] 
with the invertible $2$-cells:
\begin{equation*}
\begin{aligned}
\sigma_{01} &:&  \DD ^1 \dd ^0\Rightarrow \DD ^{0}\dd ^{0},\\
\sigma_{02} &:&  \DD ^2 \dd ^0\Rightarrow \DD ^{0}\dd ^{1},\\
\sigma_{12} &:&  \DD ^2 \dd ^1\Rightarrow \DD ^{1}\dd ^{1},									
\end{aligned}
\qquad\qquad\qquad
\begin{aligned}										
 \mathfrak{n}_0        &:&  \ss ^0\dd ^0\Rightarrow \id _{\umm},  \\
 \mathfrak{n}_1        &:&  \ss ^{0}\dd ^{1}\Rightarrow \id _{\umm } .
\end{aligned}
\end{equation*}
\end{defi}

\begin{lem}[$\mathtt{e}_{\Delta _ \mathrm{Str} }$]
There is a biequivalence 
$\mathtt{e}_{\Delta _ \mathrm{Str} } : \Delta _ \mathrm{Str}\approx \Delta_{\mathrm{3}}$ 
which is bijective on objects,
defined by:
$$ \umm \mapsto \mathsf{1}, \mbox{  }\, \doiss\mapsto \mathsf{2}, \mbox{  }\, \tress\mapsto \mathsf{3}, \hskip2em  \dd ^k \mapsto d^k, \mbox{  }\,  \ss ^0 \mapsto s^0, \mbox{  }\, \DD ^k \mapsto d^k, \hskip2em \sigma _ {ki}\mapsto \id _ {d ^i d ^k}, \mbox{  }\, \mathfrak{n} _ {k}\mapsto \id _ {\id _ {{\mathsf{1} } } }.$$
\end{lem}

\begin{rem}\label{biequivalencedeltadeltastr}
It should be noted that, given a $2$-category $\AAA $ and a pseudofunctor $\BBBB  : \Delta _ {\mathrm{3}}\to \AAA $, we can replace it by a $2$-functor $\AAAA  : \Delta _ \mathrm{Str}\to \AAA $ defined by
\begin{equation*}
\begin{aligned}
\AAAA  (\dd ^k ) &:= & \BBBB \circ \mathtt{e}_{\Delta _ \mathrm{Str} } (\dd ^k )\\
\AAAA (\DD ^k ) &:= & \BBBB \circ \mathtt{e}_{\Delta _ \mathrm{Str} } (\DD ^k )\\
\AAAA (\ss ^0 ) &:= & \BBBB \circ \mathtt{e}_{\Delta _ \mathrm{Str} } (\ss ^0 )
\end{aligned}
\qquad\qquad\qquad
\begin{aligned}			
\AAAA (\sigma_{ki} ) &: = & \left(\mathfrak{b}  _{d^i d^{k-1} }\right) ^{-1} &\cdot \mathfrak{b}  _{d^k d^i} \\
\AAAA ( \mathfrak{n}_k ) &: = &  \mathfrak{b} _{s^0 d^k} &							
\end{aligned}
\end{equation*}
in which, for each pair of morphisms $ (v, v') $ of $\Delta _ {\mathrm{3} }$,  $\mathfrak{b} _{v v' }$ is the invertible $2$-cell
$$\mathfrak{b}  _{v v' } : \BBBB  (v)\BBBB  (v ')\Rightarrow  \BBBB  (vv') $$
component of the pseudofunctor $\BBBB $ (see, for instance, \cite[Def.~2.1]{MR3491845}).
Whenever we refer to a pseudofunctor (truncated pseudocosimplicial category) $\Delta _ \mathrm{3}\to \Cat$
in the introduction, we actually  consider the replacement 
$2$-functor $\Delta_{\mathrm{Str}}\to \Cat $. For this work, there is no need for further considerations on coherence theorems.
\end{rem}

\section{Weighted colimits and the two-dimensional cokernel diagram}\label{2dimensionallimitscolimits}
The main result of this paper relates the factorization given by the \textit{lax descent object} of the \textit{two-dimensional cokernel diagram} of a morphism with the \textit{semantic factorization}, in the presence of \textit{opcomma objects} and \textit{pushouts} inside a $2$-category $\AAA $. In other words, it relates the \textit{lax descent objects}, the \textit{Eilenberg-Moore objects}, the opcomma objects and pushouts. These are known to be examples of $2$-dimensional limits and colimits.
Hence, in this section,  before defining the \textit{two-dimensional cokernel diagram} and the factorization induced by its lax descent object,
we recall the basics of the special weighted (co)limits related to the definitions.

Two dimensional limits are the same as weighted limits in the $\Cat $-enriched context~\cite{MR0401868, MR998024}. 
Assuming that $\SSS $ is a small $2$-category, let $\WWWW : \SSS\to \Cat, \mathcal{D} : \SSS \to \Cat $ and $\mathcal{D} ' : \SSS^\op \to\AAA $ be $2$-functors. If it exists, we denote the \textit{weighted limit} of $\mathcal{D} $ with weight $\WWWW $ by  $\lim\left(\WWWW, \mathcal{D}\right) $. Dually, we denote by $\colim \left(\WWWW , \mathcal{D} '\right) $ the \textit{weighted colimit} of $\mathcal{D} ' $ provided that it exists. Recall  that $\colim \left(\WWWW , \mathcal{D} '\right) $ is a weighted colimit if and only if we have a $2$-natural isomorphism (in $z$) 
$$\AAA (\colim \left( \WWWW , \mathcal{D} '\right) , z )\cong \left[\SSS ^\op, \Cat \right](\WWWW , \AAA (\mathcal{D} ' -, z  ) )\cong \lim \left(\WWWW , \AAA (\mathcal{D} ' -, z  )\right)
$$
in which $\left[\SSS ^\op , \Cat \right]$ denotes the $2$-category of $2$-functors $\SSS ^\op\to\Cat $, $2$-natural transformations and modifications. By the Yoneda embedding of $2$-categories, if a two dimensional (co)limit exists, it is unique up to isomorphism. It is also important to keep in mind the fact that existing
weighted limits in $\AAA $ are created by the Yoneda embedding $\AAA \to \left[ \AAA ^\op , \Cat \right] $,
since it preserves weighted limits and is locally an isomorphism ($\Cat$-fully faithful).

Recall that $\Cat $ has all weighted colimits and all weighted limits.
Moreover, in any $2$-category $\AAA $, every weighted colimit 
can be constructed from some special $2$-colimits provided that they exist: namely, tensor coproducts (with $\mathsf{2} $), coequalizers
and (conical) coproducts. Dually, weighted limits can be constructed from cotensor products (with $\mathsf{2}$), equalizers and
products provided that they exist.

\subsection{Tensorial coproducts}
\textit{Tensorial products} and \textit{tensorial coproducts} are weighted limits and colimits with the domain/shape $\mathsf{1} $.
So, in this case, the weight of a tensorial coproduct is entirely defined by a category $a$ in $\Cat $. If
$b $ is an object of $\mathbb{A} $, assuming its existence, we usually denote by $a\otimes b$ the tensorial coproduct, while the dual,
the cotensorial product, is denoted by $a\pitchfork b $. 

Clearly, if $b$ is an object of $\Cat $, the tensorial coproduct $a\otimes b$  in $\Cat $ is isomorphic to the (conical) product $a\times b $, while $a\pitchfork b \cong \Cat[a,b] $.

\subsection{Pushouts and coproducts}
Two dimensional conical (co)limits are just weighted limits with a weight constantly equal to the terminal category $\mathsf{1} $. Hence the weight/shape of a two dimensional conical 
(co)limits is entirely defined by the domain of the diagram.

The existence of a $2$-dimensional conical (co)limit of a $2$-functor $\mathcal{D}: \SSS\to \mathbb{A} $ defined in a locally discrete $2$-category $\SSS $  (\textit{i.e.} a diagram defined in a category $\SSS $) in a $2$-category $\mathbb{A} $ is stronger than the
existence of the $1$-dimensional conical (co)limit of the underlying functor of the $2$-functor $\mathcal{D} $ in the underlying category of $\mathbb{A} $.
However, in the presence of the former, by the Yoneda lemma for $2$-categories, both are isomorphic.

As in the $1$-dimensional case, the conical $2$-colimits of diagrams shaped by discrete categories are 
called \textit{coproducts}, while the conical $2$-colimits of diagrams with the domain being the opposite of the category $\SSSS $ defined by \eqref{eq:category-S} gives the notion
of \textit{pushout}. 
\begin{equation}\label{eq:category-S} 
\xymatrix@=1.5em{ & \dois &\\ \zero \ar[ru]^{\mathtt{d}_1} && \um\ar[lu]_{\mathtt{d}_0} }
\end{equation} 

Recall that, if $p_0 :e\to b_0 $, $p_1 : e\to b_1 $ are morphisms of a $2$-category $\AAA $, assuming its existence, the \textit{pushout} of 
$p_1 $ along $p_0 $
is an object $\displaystyle b_0\sqcup _{(p_0, p_1)} b_1 $, also denoted by  $\displaystyle p_0\sqcup_{e}p_1 $,
satisfying the following:  
there are $1$-cells 
\begin{center}
$ \displaystyle \mathfrak{d} ^0_ {p_0\sqcup_{e}p_1 } :  b_1\to b_0\sqcup _{(p_0, p_1)} b_1 $ and $\displaystyle\mathfrak{d}_{p_0\sqcup_{e}p_1} ^1:  b_0\to b_0\sqcup _{(p_0, p_1)} b_1 $ 
\end{center}
making the diagram
%
\pu
\begin{equation}
\diag{pushoutdiagramdefinition}
\end{equation}
commutative and, 
for every object $y $ and every pair of $2$-cells $$(\xi _ 0 : h_0\Rightarrow h_0' : b_1\to y,\, \xi _ 1 : h_1\Rightarrow h_1' : b_0\to y ) $$
such that the equation
%
\pu
%
\pu
\begin{equation}
\diag{twocellofpushoutdefinitionleftside}\quad = \quad \diag{twocellpushoutdefinitionrightside}
\end{equation}
holds, 
there is a unique $2$-cell $\xi : h\Rightarrow h' :  b_0\sqcup_{(p_0,p_1)}b_1\to y $ satisfying the equations
%
\pu
%
\pu
\begin{equation}
\xi_1 \, = \diag{xiumpushout}\qquad\mbox{and}\qquad \diag{xizeropushout} = \,\xi_0.
\end{equation}

\subsection{Opcomma objects}\label{OPCOMMAOBJECTSSECTIONWEIGHTEDLIMITS}
We consider the $2$-category $\SSSS $ defined in \eqref{eq:category-S} and the
weight 
$P:\SSSS\to \Cat $, defined by  $P(\um ): = P(\zero ): = \mathsf{1} $, $P(\dois ): =\mathsf{2} $, and 
$P(\mathtt{d}_0 )= d^0 , P(\mathtt{d}_1 ) =  d^1  $; that is to say, the weight
$$\xymatrix@=1.5em{ & \mathsf{2} &\\ \mathsf{1} \ar[ru]^{d^1 } && \mathsf{1}\ar[lu]_{d^0} } $$
in which $d^0 $ and $d^1 $ are respectively the inclusion of the codomain and the inclusion of the domain of the non-trivial morphism of $\mathsf{2} $ (as defined in Section \ref{sectioncategoriesdelta}).

Limits
weighted by $P $ are the well known \textit{comma objects}, while the colimits weighted by $P$ are
called \textit{opcomma objects}.
By definition, if $p_0 :e\to b_0 $, $p_1 : e\to b_1 $ are morphisms of a $2$-category $\AAA $ and   $p_0\uparrow p_1 $ 
is the opcomma object of $p_1 $ along $p_0 $, then  $\AAA ( p_0\uparrow p_1 , - ) $ is the 
comma object of   $\AAA ( p_1 , - )$ along  $\AAA ( p_0 , - )$. This means that: 
there are $1$-cells 
\begin{equation}
\delta _{p_0\uparrow p_1} ^0 :  b_1\to p_0\uparrow p_1,\quad \delta _{p_0\uparrow p_1} ^1:  b_0\to p_0\uparrow p_1 
\end{equation}
and a $2$-cell 
%
\pu
\begin{equation}\label{universal_two_cell_opcomma_object}
\diag{opcommatwocelldefinition}
\end{equation}
satisfying the following: 
\begin{enumerate} 
\item For every triple $(h_0: b_1\to y,\, h_1: b_0\to y, \,\beta : h_1 p_0\Rightarrow h_0 p_1 ) $ in which $h_0, h_1 $ are morphisms
and $\beta $ is a $2$-cell of $\AAA $, there is a unique morphism $h:  p_0\uparrow p_1\to y $ such that the equations
$h_0 = h\cdot\delta _ {p_0\uparrow p_1} ^0 $, $h_1 = h\cdot \delta _{p_0\uparrow p_1} ^1 $ and
%
\pu
%
\pu
\begin{equation}
\diag{opcommatwocelldefinitionuniversalproperty}\quad =\quad \diag{opcommatwocelldefinitionuniversalpropertyrightside}
\end{equation}
hold.
\item For every pair of $2$-cells $$(\xi _ 0 : h\cdot \delta_{p_0\uparrow{p_1}}^0\Rightarrow h'\cdot \delta_{p_0\uparrow{p_1}}^0 : b_1\to y,\quad \xi _ 1 : h\cdot \delta_{p_0\uparrow{p_1}}^1\Rightarrow h'\cdot \delta_{p_0\uparrow{p_1}}^1 : b_0\to y ) $$ such that
%
\pu
%
\pu
\begin{equation}\label{definition_for_twocells_for_opcommaobjects}
\diag{twocellofopcommatwodefinitionleftsidenovo}\quad=\quad\diag{twocellofopcommatwodefinitionrightsidenovo}
\end{equation}
holds, 
there is a unique $2$-cell $\xi : h\Rightarrow h' :  p_0\uparrow  p_1\to y $ such that
%
\pu
%
\pu
\begin{equation*}
\xi_1 \, = \diag{xiumopcomma}\qquad\mbox{and}\qquad \diag{xizeroopcomma} = \,\xi_0.
\end{equation*}
\end{enumerate}

\begin{rem}
Since $\Cat $ has all weighted colimits and limits, it has opcomma objects. More generally, 
if any $2$-category $\AAA$ has tensorial coproducts and pushouts, 
then $\AAA$ has opcomma objects.

More precisely, assuming that the tensorial coproduct $\mathsf{2}\otimes e $ exists in $\AAA $, 
we have the universal $2$-cell
$ d^1\otimes e\Rightarrow d^0\otimes e :  e\to \mathsf{2}\otimes e $
given by the image of the identity $\mathsf{2}\otimes e\to \mathsf{2}\otimes e $ by the isomorphism
$$\AAA (\mathsf{2}\otimes e, \mathsf{2}\otimes e )\cong \Cat\left[ \mathsf{2} , \AAA (e, \mathsf{2}\otimes e)
\right] . $$
If it exists, the  conical colimit of the diagram below is the opcomma object $p_0\uparrow  p_1 $ of  $p_1 : e\to b_1 $ along $p_0 : e\to b_0 $. 
$$\xymatrix{
&
e
\ar[ld]|-{p_1}
\ar[rd]|-{d^1\otimes e}
&
&
e
\ar[rd]|-{p_0}
\ar[ld]|-{d^0\otimes e}
&
\\
b_1
&
&
\mathsf{2}\otimes e
&
&
b_0
}$$
\end{rem}

\subsection{Lax descent objects}\label{laxdescentobjectsdetails}
We consider the $2$-category $\Delta _ \mathrm{Str} $ of Definition \ref{deltaestritodescidasecao1}
and we define the weight $\mathfrak{D} : \Delta _ \mathrm{Str}\to \Cat $ by
$$\xymatrix@C=2.8em{ \Delta _ \mathrm{Str}(\umm , \umm )\times \mathsf{1} 
\ar@<2.2 ex>[rrr]^-{\id _ {\Delta _ \mathrm{Str}(\umm , \umm )}\times d^0}
\ar@<-2.2ex>[rrr]_-{\id_{\Delta _ \mathrm{Str}(\umm ,\umm )}\times d^1} &&& \Delta _ \mathrm{Str}(\umm ,\umm )\times \mathsf{2} \ar[lll]|-{\overline{S}   }
\ar@<2.2 ex>[rrr]^{ \id _ {\Delta _ \mathrm{Str}(\umm ,\umm )}\times \overline{D}^0 }\ar[rrr]|-{\id _ {\Delta _ \mathrm{Str}(\umm ,\umm )}\times\overline{D}^1 }\ar@<-2.2ex>[rrr]_{\id _ {\Delta _ \mathrm{Str}(\umm ,\umm )}\times\overline{D}^2 } 
&&& \Delta _ \mathrm{Str}(\umm ,\umm )\times \left\langle\mathsf{3}\right\rangle } $$
in which:
\begin{itemize}
\renewcommand\labelitemi{--}
\item The functor $\overline{S} : \Delta _ \mathrm{Str}(\umm ,\umm )\times \mathsf{2}\to \Delta _ \mathrm{Str}(\umm ,\umm )\times \mathsf{1} $
 is defined by $$\overline{S} (\overline{v}: v\cong v' , \mathsf{0}\to\mathsf{1} ) = \left(\left( \mathfrak{n}_0^{-1}\cdot \mathfrak{n}_1 \right)\ast \overline{v}, \id _ {\mathsf{0}}\right)  = \left( s^0 d ^1  v\cong s^0 d^0  v' , \id _ {\mathsf{0}}\right). $$
\item $\left\langle\mathsf{3}\right\rangle $  is the category corresponding to the preordered set 
$$ \left\{(\mathsf{i}, \mathsf{k} )\in \left\{ \mathsf{0}, \mathsf{1} , \mathsf{2}\right\} ^2 : \mathsf{i}\neq 
\mathsf{k}\right\} $$
in which the preorder induced by the first coordinate, that is to say, $(\mathsf{i}, \mathsf{k})\leq (\mathsf{i}', \mathsf{k}') $
if $\mathsf{i}\leq \mathsf{i}' $. In other words, the category $\left\langle\mathsf{3}\right\rangle  $ is defined 
by the preordered set below.
$$\xymatrix{
&
(\mathsf{2}, \mathsf{0} )
\ar@{<->}[r]^-{\cong }
&
(\mathsf{2}, \mathsf{1} ) 
\\
(\mathsf{1}, \mathsf{2} )
\ar@{<->}[r]^-{\cong }
&
(\mathsf{1}, \mathsf{0} )
\ar[u]
&
\\
(\mathsf{0}, \mathsf{2})
\ar[u]
\ar@{<->}[rr]_-{\cong }
&
&
(\mathsf{0}, \mathsf{1}).
\ar[uu]
}$$
\item 
The functors 
$\overline{D}^0, \overline{D}^1 , \overline{D}^2 : \mathsf{2}\to \left\langle\mathsf{3}\right\rangle $ are defined by 
\begin{center}
$\overline{D}^0 (\mathsf{0}\to\mathsf{1} ) = \left((\mathsf{1} , \mathsf{0} )\to (\mathsf{2}, \mathsf{0})\right)$, $
\overline{D}^2 (\mathsf{0}\to\mathsf{1} ) = \left((\mathsf{0} , \mathsf{2} )\to (\mathsf{1}, \mathsf{2})\right)
 $ 

and

$\overline{D}^1 (\mathsf{0}\to\mathsf{1} ) = \left((\mathsf{0} , \mathsf{1} )\to (\mathsf{2}, \mathsf{1})\right). $
\end{center}
\item The natural transformations $\mathfrak{D} (\sigma _ {01} )$, $\mathfrak{D} (\sigma _ {02} )$ and $\mathfrak{D} (\sigma _ {12} )$
are defined by
$$\mathfrak{D} (\sigma _ {ij} ):= \id _ {\id _ {\Delta _ \mathrm{Str}(\umm ,\umm )}}\times \overline{\mathfrak{D} (\sigma _ {ij} )} , $$
in which
\begin{center}
$\overline{\mathfrak{D} (\sigma _ {01} )}_{\mathsf{0}}:=  \left((\mathsf{2}, \mathsf{1})\cong (\mathsf{2}, \mathsf{0})\right) $, 
$\overline{\mathfrak{D} (\sigma _ {02} )}_ {\mathsf{0}}  : =    \left((\mathsf{1}, \mathsf{2})\cong (\mathsf{1}, \mathsf{0})\right)     $ 

and 

$\overline{\mathfrak{D} (\sigma _ {12} )}_ {\mathsf{0}}:= \left((\mathsf{0}, \mathsf{2})\cong (\mathsf{0}, \mathsf{1})\right) $.
\end{center}

\item The natural transformation
$$\mathfrak{D} (\mathfrak{n} _ {i} ) 
: \overline{S}\circ \left(\id _ {\Delta _ \mathrm{Str}(\umm , \umm )}\times d^i\right)\Rightarrow \id  _ {\Delta _ \mathrm{Str}(\umm , \umm )\times \mathsf{1} } $$  
is defined by
$
\mathfrak{D} (\mathfrak{n} _ {i} )_{  (v, \mathsf{0} )   } : = \left( \mathfrak{n} _ {i}\ast \id _ v, \id _{\mathsf{0}} \right)  
$.
\end{itemize}
\begin{defi}[Lax descent object]
Given a $2$-functor $\BBBB :  \Delta _ {\mathrm{Str}}\to \AAA $, if it exists, the weighted limit
$\lim (\mathfrak{D} , \BBBB )$ is called the \textit{lax descent object} of $\BBBB $. 
\end{defi}

\begin{rem}
\textit{Since $\Cat $ has all weighted limits, it has lax descent objects.} More precisely,
if $\AAAA :  \Delta _ {\mathrm{Str}}\to \Cat $ is a $2$-functor, 
$$\lim (\mathfrak{D} , \AAAA )\cong \left[\Delta _ {\mathrm{Str}} , \Cat \right] (\mathfrak{D}, \AAAA)  $$
is the category in which:
\begin{enumerate}
\item Objects are $2$-natural transformations $\overline{\psi }: \mathfrak{D} \longrightarrow \AAAA $. We have a bijective correspondence between such $2$-natural transformations and pairs 
$(w, \psi )$ in which $w$ is an object of $ \AAAA (\umm ) $ and $\psi : \AAAA (\dd ^1)(w)\to\AAAA (\dd ^0)(w) $ is a morphism in $ \AAAA (\doiss ) $ satisfying the following equations:	
\begin{itemize}
\item[] Associativity:
$$\AAAA (\DD ^0 )(\psi )\,\cdot \, \AAAA (\sigma _ {02}) _ {w}\, \cdot \, \AAAA (\DD ^2)(\psi ) = \AAAA (\sigma _ {01} )_ {w}\,\cdot\,\AAAA(\DD ^1)(\psi )\,\cdot\,\AAAA (\sigma _ {12}) _ {w};   $$
\normalsize
\item[] Identity:
$$\AAAA(\mathfrak{n} _0) _ {w}\,\cdot\, \AAAA(\ss^0) (\psi ) = \AAAA(\mathfrak{n} _1) _ {w}. $$
\normalsize
\end{itemize}
If $\overline{\psi }: \mathfrak{D}\longrightarrow \AAAA $ is a $2$-natural transformation, we get such pair by the correspondence 
$$\overline{\psi } \mapsto (\overline{\psi }_{\umm } (\id _ {{}_{\umm }} , \mathsf{0} ), \overline{\psi } _ {\doiss } (\id _ {{}_{\umm }} , \mathsf{0}\to\mathsf{1} ) ) .$$

\item The morphisms are modifications. In other words, a morphism $\mathfrak{m} : (w, \psi )\to (w', \psi ') $ is determined by a morphism $\mathfrak{m}: w\to w' $ in $\AAAA (\umm ) $ such that $$\AAAA (\dd ^0)(\mathfrak{m} )\cdot \psi  = \psi ' \cdot \AAAA (\dd ^1)(\mathfrak{m} ) .$$
\end{enumerate}
\end{rem}
By definition,  if $\BBBB : \Delta _ {\mathrm{Str}}\to \AAA $ is a $2$-functor, 
an object $\mathrm{lax}\textrm{-}\mathcal{D}\mathrm{esc}\left(\BBBB \right) $
is the lax descent object $\lim (\mathfrak{D} , \BBBB )$ of $\BBBB $ if
and only if there is a $2$-natural isomorphism (in $y$) $$\AAA (y, \mathrm{lax}\textrm{-}\mathcal{D}\mathrm{esc}\left(\BBBB \right) )\cong 
\lim (\mathfrak{D} , \AAA (y , \BBBB -) ) .$$ 
Equivalently,  a lax descent object of $\BBBB$ is, if it exists, an object $\lim(\mathfrak{D},\BBBB) $ of $\AAA $
together with a pair 
%
\pu
%
\pu
\begin{equation}
\left(\diag{universalonecelllaxdescentda}, \quad\diag{universaltwocelllaxdescentpsi}\right)
\end{equation} 
of a morphism $\dd^{\left(\mathfrak{D},\BBBB\right)}  $ and a $2$-cell $\Psi^{\left(\mathfrak{D},\BBBB\right)} $ in $\AAA $ satisfying the following three properties.
\begin{enumerate} 
\item For each pair
 $$( h: y\to \BBBB (\umm ),\,  \beta : \BBBB ( \dd ^1 )\cdot h\Rightarrow \BBBB ( \dd ^0 )\cdot h ) $$
in which $h$ is a morphism and $\beta $ is a $2$-cell of $\AAA $ such that  
the equations
%
\pu
%
\pu
%
\pu
%
\pu
\begin{align} 
\diag{laxdescentassociativityleftside}&\qquad =\qquad  \diag{laxdescentassociativityrightside}\label{Associativityequationdescent}\\
\diag{laxdescentidentityleftside}&\qquad=\qquad \diag{laxdescentidentityrightside}\label{Identityequationdescent}
\end{align} 
hold, there is a unique morphism $h^{( \BBBB , \beta) }: y\to \lim (\mathfrak{D} , \BBBB ) $ 
making the diagram
%
\pu
\begin{equation}\label{laxdescentgeneralfactorization_equation}
\diag{laxdescentgeneralfactorization}
\end{equation}
%
\pu
%
\pu
commutative and such that \eqref{eq:lax-descent-univ-equation-now-with-number} holds.
\begin{equation}\label{eq:lax-descent-univ-equation-now-with-number}
\diag{equation_on_the_descent_datum_leftside}\quad  =\quad \diag{equation_on_the_descent_datum_rightside}
\end{equation}
\item The pair $(\dd ^{(\mathfrak{D} , \BBBB )}, \Psi ^{(\mathfrak{D} , \BBBB )} ) $ satisfies the \textit{descent associativity} \eqref{Associativityequationdescent} and the \textit{descent identity} \eqref{Identityequationdescent}. In this case, the unique morphism
induced is clearly the identity on $\lim (\mathfrak{D} , \BBBB )$, that is to say, $$\left(\dd ^{(\mathfrak{D} , \BBBB )}\right) ^{( \BBBB , \Psi ^{(\mathfrak{D} , \BBBB )})} = \id_ {\lim (\mathfrak{D} , \BBBB )} .$$
\item Assume that $(h _ 1, \beta _1 )$ and $(h_0 , \beta _ 0 ) $ are pairs satisfying \eqref{Associativityequationdescent} and \eqref{Identityequationdescent}, inducing 
factorizations $h_1 = \dd^{\left(\mathfrak{D},\BBBB\right)}\circ h_1 ^{(\BBBB , \beta  _1)}$
and $h_0 = \dd^{\left(\mathfrak{D},\BBBB\right)}\circ h_0 ^{(\BBBB , \beta  _0)}$.

For each $2$-cell $$\xi : h_1\Rightarrow h_0 : y\to \BBBB (\umm ) $$ satisfying the equation 
%
\pu
%
\pu
\begin{equation}
\diag{equation_two_cell_for_descent_left_side}\quad =\quad \diag{equation_two_cell_for_descent_right_side}
\end{equation}
there is a unique $2$-cell $$\xi ^{(\BBBB , \beta _ 1, \beta_0)} : h_1 ^{(\BBBB , \beta  _1)}\Rightarrow h _0^{(\BBBB , \beta _ 0 )} : y\to\lim (\mathfrak{D} , \BBBB )  $$ 
such that 
%
\pu
\begin{equation}
\diag{xilaxdescenttwocellproperty}=\quad \xi .
\end{equation}
\end{enumerate}

\subsection{The two-dimensional cokernel diagram}\label{subsectiondefinicaodehighercokernel}
Let $p: e\to b $ be a morphism of a $2$-category $\AAA $. $\AAA $ \textit{has the two-dimensional cokernel diagram of $p $} if
$\AAA $ has the opcomma object
%
\pu
\begin{equation}
\diag{opcomma_p_along_p}
\end{equation}
of $p$ along itself and the pushout \eqref{pushoutsquaredefiningbububp} of $\d ^0 $ along $\d ^1$.
%
\pu
\begin{equation}\label{pushoutsquaredefiningbububp}
\diag{pushoutdiagramfordefiningtwodimensionalcokerneldiagram}
\end{equation}
\textit{Henceforth, in this section, we assume that $\AAA $ has the two-dimensional cokernel diagram of $p $}. We denote by
$$\D ^1 : b\uparrow _p b\to b\uparrow _p b \uparrow _p b $$
the unique morphism such that  the equations
\begin{equation}\label{definitionofD1}
\D ^1 \, \d ^1 = \D^2\, \d ^1 ,\quad \D ^1 \, \d ^0 = \D^0\, \d ^0
\end{equation}
%
\pu
%
\pu
\begin{equation}\label{associativityequationhighercokernel}
\diag{twodimensionalcokerneldiagramassociativityleftsidedefinition}\qquad =\qquad \diag{twodimensionalcokerneldiagramassociativityrightsidedefinition}
\end{equation}
hold, while we denote by $\s ^0 :b\uparrow _ p  b\to b $ the unique morphism such that $\s ^0\cdot \d ^1 = \s ^0 \cdot \d ^0 =  \id _ b $ and 
 \eqref{highercokernelidentity} holds.
%
\pu
%
\pu
\begin{equation}\label{highercokernelidentity}
\diag{twodimensionalcokerneldiagramidentityleftsidedefinition}\qquad =\qquad \diag{twodimensionalcokerneldiagramidentityrightsidedefinition}
\end{equation}
\begin{defi}[Two-dimensional cokernel diagram]\label{definicaodehighercokernel}
Consider the $2$-functor $$\mathcal{H}_p ': \Delta _\mathrm{3}\to \AAA $$  defined by $\H_p ' (d^i : \mathsf{1}\to \mathsf{2} ) = \d ^i $, 
$\H_p ' (d^i : \mathsf{2}\to \mathsf{3} ) = \D ^i $ and $\mathcal{H}_p ' (s ^0) = \s ^0 $. The $2$-functor
$$\mathcal{H}_p : = \mathcal{H}_p '\circ \mathtt{e}_{\Delta _ \mathrm{Str} }: \Delta _\textrm{Str}\to \AAA   $$
\begin{equation}\tag{$\mathcal{H}_p$}
\xymatrix{  b\ar@<2 ex>[rrr]|-{\d ^0 }\ar@<-2ex>[rrr]|-{\d ^1} &&&  b\uparrow _ p  b\ar[lll]|-{\s^0 }
\ar@<2 ex>[rrr]|-{\D ^0}\ar[rrr]|- {\D ^1 }\ar@<-2ex>[rrr]|-{\D  ^2} &&& b\uparrow _p b \uparrow _p b  }
\end{equation}
is called the \textit{two-dimensional cokernel diagram  of $p$}.
\end{defi}

\begin{rem}
	This construction was, for instance, already considered in \cite[pag.~135]{MR1314469} under the name \textit{resolution}, specially in the $2$-category of pretoposes and in the $2$-category of exact categories. 
\end{rem}

\begin{rem}\label{123ample}
The $2$-category of categories $\Cat $ has the two-dimensional cokernel diagram  of any functor. In particular, the two-dimensional cokernel diagram of $\id _ \mathsf{1} $ is
\begin{equation}\tag{$\H _ {\id _ \mathsf{1} }$}
\xymatrix{  \mathsf{1}\ar@<2ex>[rr]|-{d^0}\ar@<-2ex>[rr]|-{d^1} && \mathsf{2}\ar[ll]|-{s^0}
\ar@<2 ex>[rr]|-{d ^0}\ar[rr]|-{d ^1}\ar@<-2ex>[rr]|-{d ^2} && \mathsf{3} }
\end{equation}
which is just the usual inclusion 
of the locally discrete $2$-category $\Delta _ \mathrm{3} $ in $\Cat $. 
\end{rem}

By the definitions of $\D ^1 $ and $\s ^0 $ (\eqref{associativityequationhighercokernel} and \eqref{highercokernelidentity}), the 
pair $$(p : e\to b,\quad \upalpha : \d ^1 \cdot p\Rightarrow \d ^0 \cdot p : e \to b\uparrow _ p b ) $$ satisfies the
descent associativity and identity w.r.t. $\mathcal{H}_p : \Delta _\textrm{Str}\to \AAA  $ (that is to say, \eqref{Associativityequationdescent}
and \eqref{Identityequationdescent} w.r.t. $\mathcal{H}_p$). Hence, if $\AAA   $ has the lax descent object
$$(\lim (\mathfrak{D}, \mathcal{H}_p), \dd ^{p},  \Psi ^{p}) : = (\lim (\mathfrak{D}, \mathcal{H}_p), \dd ^{(\mathfrak{D} , \mathcal{H}_p )},  \Psi ^{(\mathfrak{D} , \H _ p )})$$
of $\mathcal{H}_p $, by the universal property of the lax descent object, there is a unique morphism $$p^\mathcal{H}:= p ^{(\mathcal{H}_p  , \upalpha )}  : e\to \lim (\mathfrak{D}, \mathcal{H}_p) $$ such that 
%
\pu
\begin{equation}\label{laxdescentfactorizationofthehighercokernel}
\diag{semanticdescent_factorization_of_p}
\end{equation}
commutes and the equation
%
\pu
%
\pu
\begin{equation}
	\diag{equation_on_the_descent_datum_facorization_twodimensionalcokerneldiagram_leftside} \quad = \quad\diag{equation_on_the_descent_datum_facorization_twodimensionalcokerneldiagram_rightside}
\end{equation}
holds.

Inspired by our main result (Theorem \ref{principal}), we establish the following terminology.
\begin{defi}
Assume that $\AAA$ has the lax descent object of $\mathcal{H}_p $.
The lax	descent factorization induced by $\mathcal{H}_p $ and the universal $2$-cell of the opcomma object $b\uparrow_pb$ given in 
\eqref{laxdescentfactorizationofthehighercokernel} is called the \textit{semantic lax descent factorization of $p$}.
\end{defi}

\begin{rem}
Given an object $x $ and a morphism $p: e\to b $ of $\AAA $ as above, 
the factorization induced 
by the universal property of the 
lax descent category
$\lim (\mathfrak{D}, \AAA (x,\mathcal{H}_p -) )$ and
by the pair 
%
\pu
\begin{equation}
\left( \AAA (x, p),\quad \diag{image_by_of_opcomma_p_along_p_Ax}\right)
\end{equation}
is given by
%
\pu
\begin{equation}\label{imaginglaxdescentfactorizationofthehighercokernel}
\diag{image_semanticdescent_factorization_of_p}
\end{equation}
in which
\begin{eqnarray*}
\AAA (x, p )^{(\AAA (x,\mathcal{H}_p -), \AAA (x, \upalpha )   )}  : & \AAA (x, e ) & \to \lim (\mathfrak{D}, \AAA (x,\mathcal{H}_p -) )\cong \AAA (x, \lim (\mathfrak{D}, \mathcal{H}_p))\\
& g & \mapsto  (p\cdot g, \upalpha \ast \id _ g )\\
& \chi : g\Rightarrow g' & \mapsto \id _ p \ast \chi\\
&&\\
\dd ^{\left(\mathfrak{D} ,  \AAA (x,\mathcal{H}_p -)   \right)}  :   & 
\lim (\mathfrak{D}, \AAA (x,\mathcal{H}_p -) )
 & \to 
\AAA (x, b )\\
& (f, \psi ) & \mapsto  f\\
& \xi  & \mapsto \xi .
\end{eqnarray*}
Clearly, if $\AAA $ has the lax descent object of the two-dimensional cokernel diagram $ \H _p $, the factorization given in
\eqref{imaginglaxdescentfactorizationofthehighercokernel}
is isomorphic
to the factorization
\begin{equation}
\AAA(x,p) = \AAA(x,\dd^{p})\circ \AAA(x,p^\H)
\end{equation}
given by the image of \eqref{laxdescentfactorizationofthehighercokernel} by $\AAA (x,-): \AAA \to \Cat$,
since the Yoneda embedding 
creates any existing lax descent objects in $\AAA $.
\end{rem}

\begin{rem}
It should be noted that, assuming that we can construct $\H _ p $ in $\AAA $, the factorization \eqref{imaginglaxdescentfactorizationofthehighercokernel} 
always exists, since $\Cat $ has lax descent objects (lax descent categories). 

Moreover, since opcomma objects (weighted colimits in general) might not be preserved by the Yoneda embedding, the definition of the factorization given in \eqref{imaginglaxdescentfactorizationofthehighercokernel}
 does not coincide with the definition of semantic lax descent factorization of $\AAA (x, p ) $ in $\Cat $. 

For instance, consider the example of Remark \ref{123ample}. For any object $x $ of $\Cat $,
clearly the opcomma object of $\Cat[x, \id_ {\mathsf{1}}]$ along itself
is isomorphic to the opcomma object of $ \id_ {\mathsf{1}} $ along itself, that is to say, $\mathsf{2} $. Hence,
since there is a category $x $ such that  $\Cat[x, \mathsf{2} ] $ is not isomorphic to $\mathsf{2} $, this
shows that the Yoneda embedding does not preserve the opcomma object $\id_ {\mathsf{1}}\uparrow \id_ {\mathsf{1}} $.
\end{rem}

\begin{rem}[Duality: two-dimensional kernel diagram and the induced factorization]\label{smaenotionoffactorization}
The codual notion of that of the two-dimensional cokernel diagram gives the same notion of
factorization (assuming the existence of the suitable lax descent object): namely, the factorization given in  \eqref{laxdescentfactorizationofthehighercokernel} of the morphism $p$.

The dual concept of the two-dimensional cokernel diagram, the \textit{$2$-dimensional kernel diagram} of $l:b\to e$, if it exists, is a $2$-functor 
$$\H ^l : \Delta _{\textrm{Str}} ^{\op }\to \AAA   $$
$$\xymatrix{  
b\downarrow ^ l  b\downarrow ^ l  b
\ar@<3 ex>[rrr]|-{\D _0 ^{l\downarrow l} }\ar[rrr]|- {\D _1^{l\downarrow l} }\ar@<-3ex>[rrr]|-{\D  _2 ^{l\downarrow l}}
&&&
b\downarrow ^ l  b
\ar@<3 ex>[rrr]|-{\d _0  ^{l\downarrow l } }\ar@<-3ex>[rrr]|-{\d _1 ^{l\downarrow l } }
&&&
b\ar[lll]|-{\s^{l\downarrow l}_0 }  }$$
constructed from 
suitable
comma objects and pullbacks. 

In this case,
we get the lax codescent factorization induced by the $2$-dimensional kernel diagram  of $l$ \eqref{laxdcoescentfactorizationofthehigherkernel}, provided that
the 
lax
\textit{codescent object}  of $\H ^l $ exists. Herein, we call the factorization \eqref{laxdcoescentfactorizationofthehigherkernel} the \textit{semantic lax codescent factorization of $l$}. 
%
\pu
\begin{equation}\label{laxdcoescentfactorizationofthehigherkernel}
\diag{dual_semanticcodescent_factorization_of_l}
\end{equation}
In the special case of the $2$-category $\Cat $, the $2$-dimensional kernel diagram was, for instance, also considered in \cite[pag.~544]{MR2107402}
under the name \textit{higher kernel} and at \cite[Proposition~4.7]{MR1935980} under the name \textit{congruence}.
\end{rem}

\section{Semantic factorization}\label{semanticandfurther}
Assuming that $\AAA $ has suitable Eilenberg-Moore objects, the  
\textit{semantics-structure} adjunction (\textit{e.g.} \cite[Section~2]{MR0299653}) gives rise to what is called herein the \textit{semantic factorization of a tractable morphism $p$}.
In this section, we recall the semantic factorization of morphisms that   have codensity monads. Before doing so, we recall the definition  
 of the Eilenberg-Moore object of a given monad.

\subsection{Eilenberg-Moore object}
Recall that a monad
in a $2$-category $\AAA $ is a quadruple 
%
\pu
%
\pu

$$\t = \left(b,\quad t: b\to b ,\quad \diag{multiplication_of_the_monad_definition},\, \diag{unit_of_the_monad_definition} \right) $$ in which $b$ is an object,
$t$ is a morphism and $m, \eta $ are $2$-cells in $\AAA $ such that the equations
%
\pu

%
\pu

%
\pu
%
\pu
\begin{eqnarray}
\diag{leftsideoftheequationassociativityofmonad}  &\qquad = \qquad& \diag{rightsideoftheequationassociativityofmonad}\\
\diag{firstsideoftheequationidenityofamonad} &\qquad =\qquad &  \diag{secondsideoftheequationidenityofamonad}\qquad = \qquad \id_t
\end{eqnarray}
hold.
A monad can be seen as a $2$-functor
$\t : \mnd\to \AAA $
from the free monad $2$-category $\mnd = \Sigma \Delta  $ to $\AAA $ (\textit{e.g.} \cite[pag.~178]{MR0401868} or \cite[Sect.~6]{2016arXiv160604999L}).  
If it exists, the \textit{Eilenberg-Moore object}, also called the \textit{object of algebras}, is a special weighted limit of $\t $. More precisely, 
given a monad $\t $ in $\AAA $, 
the object $b^\t $ is the Eilenberg-Moore object of $\t $ if and only if
there is a $2$-natural isomorphism (in $y$)
$$\AAA ( y , b ^\t )\cong \AAA (y, b) ^{\AAA (y, \t )} $$
in which $\AAA (y, b) ^{\AAA (y, \t )}$ is the Eilenberg-Moore category of the monad $$(\AAA (y, b ), \AAA (y, t ), \AAA (y, m ), \AAA (y, \eta )) $$ in $\Cat $.
This means that, if the Eilenberg-Moore object $b ^\t$   of $\t = (b, t, m, \eta ) $ exists, it is characterized by the following universal property.
There is a pair 
%
\pu
\begin{equation}
\left(\uu ^\t : b ^\t\to b ,\quad \diag{algebra_multiplication_of_the_monad_definition} \right)
\end{equation}
in which $\uu ^\t $ is a morphism, and $\mu ^\t $ is a $2$-cell in $\AAA $ satisfying the following three properties.
\begin{enumerate} 
\item For each pair
 $( h: y\to b, \beta : t\cdot h\Rightarrow h ) $ 
in which $h$ is a morphism and $\beta $ is a $2$-cell in $\AAA $ making the equations  
%
\pu

%
\pu

%
\pu
%
\pu
\begin{eqnarray}
\diag{algebra_rightsideoftheequationassociativityofmonad}&\qquad =\qquad  & \diag{algebra_leftsideoftheequationassociativityofmonad}\label{Algebraassociativityextensions}\\
\diag{algebra_firstsideoftheequationidenityofamonad} &\qquad  =\qquad  & \diag{algebra_secondsideoftheequationidenityofamonad}\label{i_Algebraassociativityextensions}
\end{eqnarray}	
hold,  there is a unique morphism $h^{( \t , \beta) }: y\to b ^\t $ such that the equation
$\mu^\t\ast \id_{h^{( \t , \beta) }} = \beta $.
It should be noted that, in this case, we get in particular that
%
\pu
\begin{equation}\label{generalalgebrafactorization_universalproperty_equation}
\diag{generalalgebrafactorization_universalproperty}
\end{equation}
commutes.

\item The pair $(\uu ^\t, \mu ^\t ) $ satisfies the \textit{algebra associativity \eqref{Algebraassociativityextensions} and identity \eqref{i_Algebraassociativityextensions} equations}. In this case, the unique morphism
induced is clearly the identity on $b ^\t$.

\item Assume that 
$h_1 ^{(\t , \beta_1  )}, h _0^{(\t , \beta _0 )} : y\to b ^\t $
are morphisms in $\AAA $. For each $2$-cell 
$$\xi :\uu^\t\cdot h_1 ^{(\t , \beta_1  )} \Rightarrow \uu^\t\cdot h _0^{(\t , \beta _0 )} : y\to b $$
such that the equation
%
\pu
%
\pu
\begin{equation}
\diag{righttwocell_algebrastructure_monad_resulting}\quad =\quad \diag{lefttwocell_algebrastructure_monad_resulting}
\end{equation}
is satisfied, there is a unique $2$-cell $$\xi _{(\t, \beta_1 , \beta _0)} : h_1^{(\t , \beta_1  )}\Rightarrow h _0^{(\t , \beta _0 )} : y\to b ^\t  $$ 
such that 
%
\pu
\begin{equation}
\diag{xilaxdescenttwocellproperty_algebras}=\quad \xi .
\end{equation}

\end{enumerate}

\begin{rem}[Duality: Kleisli objects and co-Eilenberg-Moore objects]
The dual of the notion of Eilenberg-Moore object of a monad is called the \textit{Kleisli object} of a monad,
while the codual is called the \textit{co-Eilenberg-Moore object}, or \textit{object of coalgebras}, of a comonad.
In the special case of $\Cat $, these notions coincide with the usual ones (\textit{e.g.} \cite[Section~5]{MR0299653}).
\end{rem}

\subsection{Kan extensions}\label{importantnotationcodensitymonads}

Let $f: z\to y  $ and $g: z\to x $ be morphisms of a $2$-category $\AAA $. The right Kan extension of $f $ along $g $ is, if it exists, 
the right reflection $\ran _ g f $ of $f$ along the functor
$$\AAA (g, y ) : \AAA (x, y)\to \AAA (z, y). $$
This means that the right Kan extension is actually a pair 
$$\left( \ran _ g f : x\to y , \gamma ^{ \ran _ g f } : \left(\ran _ g f\right)\cdot g\Rightarrow f 
\right) $$ 
of a morphism $\ran _ g f $ and a $2$-cell $\gamma ^{ \ran _ g f }$, called the universal $2$-cell,  such that, for each morphism $h: x\to y $ of $\AAA $,
%
\pu
%
\pu
\begin{equation}
\diag{firstsideoftheequation_definition_of_Kan_Extensions}\qquad \mapsto \qquad\diag{secondsideoftheequation_definition_of_Kan_Extensions}
\end{equation}
defines a bijection $\AAA (x,y)(h, \ran _ g f)\cong \AAA (z,y)(h \cdot g,  f) $.

\begin{rem}[Duality: right lifting and left Kan extension]
The dual notion of that of a right Kan extension is called \textit{right lifting} (see \cite[Section~1]{MR0463261}),
while the codual notion is called the \textit{left Kan extension}. 
Finally, of course, we also have the codual notion of the right lifting: the \textit{left lifting}.
\end{rem}

Let $p_0 : e\to b_0, p_1 : e\to b_1 $  be morphisms of a $2$-category $\AAA $.
Assume that $\AAA $ has the opcomma object $p_0\uparrow p_1$ and $$\alpha ^{p_0\uparrow p_1} : \delta ^1 _{p_0\uparrow p_1 }\cdot p_0\Rightarrow    
\delta ^0 _{p_0\uparrow p_1 }\cdot p_1  $$ is the universal $2$-cell that 
gives $p_0\uparrow p_1 $ as the opcomma object of $p_1 $ along $p_0 $, as in \ref{OPCOMMAOBJECTSSECTIONWEIGHTEDLIMITS} (Eq.~\ref{universal_two_cell_opcomma_object}). In this case, we have:

\begin{lem}\label{teoremadeultimomomento}
Given a morphism
$h: p_0\uparrow p_1\to y $,
the following statements are equivalent.
\begin{enumerate}[label=\roman*)]
\item The pair $(h,\, \id _ { h\cdot \delta ^0_{p_0\uparrow p_1}} ) $ is the right Kan extension of 
$h\cdot \delta ^0_{p_0\uparrow p_1}$ along $\delta ^0_{p_0\uparrow p_1}$.\label{1opcommasupport} 
\item The pair $(h\cdot \delta ^1_{p_0\uparrow p_1} ,\,  \id _ h \ast \alpha ^{p_0\uparrow p_1} ) $ is the right Kan 
extension of $h\cdot \delta ^0_{p_0\uparrow p_1}\cdot p_1 $ along $p_0$.\label{2opcommasupport}
\end{enumerate}  
\end{lem}
\begin{proof}
Assuming \ref{1opcommasupport}, given a $2$-cell 
$$\check{\beta} : h'_1 \cdot p_0\Rightarrow h\cdot \delta ^0_{p_0\uparrow p_1}\cdot p_1 : e\to y, $$
we conclude, by the universal property of the opcomma object, that there is a unique 
morphism $h': p_0\uparrow p_1\to y $ such that 
%
\pu
%
\pu
\begin{equation}
\diag{firstdiagram_for_the_proof_of_opcomma_result_about_Kan_Extensions}\quad = \quad  \diag{seconddiagram_for_the_proof_of_opcomma_result_about_Kan_Extensions}
\end{equation}
and $h' \cdot \delta ^0_{p_0\uparrow p_1} = h\cdot \delta ^0_{p_0\uparrow p_1} $. 

By the universal property of the Kan extension, there is a unique $2$-cell 
$\underline{\beta }: h'\Rightarrow h $ such that the equation
%
\pu
%
\pu
\begin{equation}
\diag{Leftside_KanExtension_Identity_Support_Lemma_Opcomma_Objects}\quad = \quad
\diag{KanExtension_Identity_Support_Lemma_Opcomma_Objects}
\end{equation}
is satisfied.

By the universal 
property of the opcomma object (see \eqref{definition_for_twocells_for_opcommaobjects}), this means that 
$\underline{\beta }\ast \id _ {\delta ^1_{p_0\uparrow p_1}} $
is the unique $2$-cell such that
%
\pu
%
\pu
\begin{equation}
\diag{uniquetwocellbetaunderlineoftheproofsupportopcomma}\quad =\quad  \diag{uniquetwocellbetaunderlineoftheproofsupportopcommadois}\quad = \quad \check{\beta }.
\end{equation}
This proves \ref{2opcommasupport}.

Reciprocally, assuming \ref{2opcommasupport}, by the universal property of the right Kan extension 
$$\ran_{p_0} \left(h\cdot \delta ^0 _{p_0\uparrow p_1}\cdot p_1\right) =  (h\cdot \delta ^1_{p_0\uparrow p_1} ,\, \id _ h \ast \alpha ^{p_0\uparrow p_1} ) $$ of the hypothesis, we have that, given any $2$-cell 
$$\beta _0 : h'\cdot \delta ^0 _{p_0\uparrow p_1}\Rightarrow   h\cdot \delta ^0 _{p_0\uparrow p_1}, $$
 there is a unique $2$-cell  $\beta _1 : h'\cdot\delta ^1 _{p_0\uparrow p_1}\Rightarrow h\cdot \delta ^1 _{p_0\uparrow p_1}  $
such that
$$\xymatrix{
&
e
\ar[ld]|-{p_0}
\ar[rd]|-{p_1}
&
&
&
e
\ar[ld]|-{p_0}
\ar[rd]|-{p_1}
&
\\
b
\ar[rd]|-{h'_1 }
&
\xRightarrow{\id _ {h'}\ast\, \alpha ^{p_0\uparrow p_1} }
&
b
\ar@/^1.2pc/[ld]|-{h_0}
\ar@/_1.2pc/[ld]|-{h'_0}
\ar@{}[ld]|-{\xRightarrow{\mbox{ } \beta _ 0 \mbox{ } } }
\ar@{}[r]|-{=}
&
b
\ar@/^1.2pc/[rd]|-{h_1}
\ar@/_1.2pc/[rd]|-{h_1'}
\ar@{}[rd]|-{\xRightarrow{\mbox{ } \beta _ 1 \mbox{ } } }
&
\xRightarrow{\id _ {h} \ast\, \alpha ^{p_0\uparrow p_1} }
&
b
\ar[ld]|-{h_0 }
\\
&
y
&
&
&
y
&
} $$
holds, in which,  for each $i\in\left\{ 1,2 \right\} $,  $h_i ': = h ' \cdot \delta ^i _{p_0\uparrow p_1}$ and $h_i : = h  \cdot \delta ^i _{p_0\uparrow p_1}$.

By the universal property of the opcomma $p_0\uparrow p_1 $, this implies that there is a unique $\beta: h'\Rightarrow h $ such that $\beta \ast \id _ {\delta ^0 _{p_0\uparrow p_1} } = \beta _0 $. Hence we get \ref{1opcommasupport}.
\end{proof}

\begin{defi}[Codensity monad]\label{definitioncodensitymonadsection}
A morphism $p: e\to b $ of a $2$-category $\AAA $ \textit{has the codensity monad} if the right Kan extension $( \ran _p p, \gamma )$  of $p$ along itself
exists. 
Assuming that $\AAA $ has the codensity monad of $p$ and
denoting $\ran _p p $  by $t$, we consider: 
\begin{itemize}
\renewcommand\labelitemi{--}
\item  the $2$-cell $m : t^2\Rightarrow t $ such that 
%
\pu
%
\pu

\begin{equation}\label{definitionofcodensitymultiplication}
\diag{left_side_codensity_multiplication_definition}\qquad =\qquad\diag{right_side_codensity_multiplication_definition}
\end{equation}
holds;
\item the $2$-cell $\eta : \id _ b\Rightarrow t $ such that \eqref{definitionofcodensityidentity} holds.
%
\pu
%
\pu
\begin{equation}\label{definitionofcodensityidentity}
\diag{left_side_codensity_identity_definition}\qquad = \qquad \diag{right_side_codensity_identity_definition}
\end{equation}
\end{itemize}
In this case, by the universal property of the right Kan extension of $p$ along itself, 
the quadruple  $\t  = (b , t, m, \eta ) $ is a monad called the \textit{codensity monad} of $p$. 
\end{defi}

Assuming that $\t  = (b , t, m, \eta ) $ is the \textit{codensity monad} of $p: e\to b $ as above,
by \eqref{definitionofcodensitymultiplication}  and \eqref{definitionofcodensityidentity}, 
it is clear that the pair $(p: e\to b ,\, \gamma : t\cdot p\Rightarrow p ) $ satisfies the algebra associativity and
identity equations w.r.t. the monad $\t $ (that is to say, \eqref{i_Algebraassociativityextensions} and \eqref{Algebraassociativityextensions} w.r.t. the monad $\t $). Hence, assuming that $\AAA $ has the Eilenberg-Moore object $\left( b^\t ,\, \uu ^\t : b^\t \to b,\, \mu ^\t \right)  $  of the monad $\t $, by the universal property, there is a unique $p^\t := p^{(\t , \gamma)} $
such that
%
\pu
\begin{equation}\label{semanticfactorizationofamorphismthathascodensitymonad}
\diag{semantic_factorization_of_p_definition_dednat}
\end{equation}
commutes and $\mu ^\t\ast \id _ {p ^\t } = \gamma  $. 
\begin{defi}
The factorization given 
by \eqref{semanticfactorizationofamorphismthathascodensitymonad} is called herein the 
\textit{semantic factorization} of $p$.
\end{defi}
For each object $x $, assuming the existence of the semantic factorization of $p$, we can take its 
image  by the representable $2$-functor
$\AAA (x, -) : \AAA \to \Cat $, getting the factorization 
\begin{equation}\label{simpleimageofthesemanticfactorizationbyrepresentable}
\AAA (x, p ) = \AAA (x,\uu ^\t )\circ \AAA (x,p^\t  ).
\end{equation}
Since the Yoneda embedding creates any existing Eilenberg-Moore object of $\AAA $, the factorization
\eqref{simpleimageofthesemanticfactorizationbyrepresentable} 
coincides up to isomorphism with the
\textit{factorization of $\AAA (x, p ) $ induced by $\left( \AAA (x, p ),\AAA (x, \gamma )\right) $ and $\AAA (x, b) ^{\AAA (x, \t ) }$}; that is to say, the commutative triangle 
%
\pu
\begin{equation}\label{imagingsemanticfactorizationofamorphismthathascodensitymonad}
\diag{imageofAx_semantic_factorization_of_p_definition_dednat}
\end{equation}
which is given by
\begin{eqnarray*}
\AAA (x, p )^{\left( \AAA (x, \t ) ,  \AAA (x, \gamma )\right) }    : 
& \AAA (x, e ) & \to \AAA (x, b) ^{\AAA (x, \t ) }\\
& g & \mapsto  (p\cdot g, \gamma \ast \id _ g )\\
& \chi : g\Rightarrow g' & \mapsto \id _ p \ast \chi \\
&&\\
\uu ^{\AAA (x, \t)} :   & 
\AAA (x, b) ^{\AAA (x, \t ) }
 & \to 
\AAA (x, b )\\
& (f, \beta ) & \mapsto  f\\
& \xi  & \mapsto \xi{.}
\end{eqnarray*}

\begin{rem}
Let $p$ be a morphism of $\AAA $ which has the codensity monad $\t $. Since
 $\Cat $ has Eilenberg-Moore objects,
 the
factorization of $\AAA (x, p ) $ induced by $\left( \AAA (x, p ),\AAA (x, \gamma )\right) $ and $\AAA (x, b) ^{\AAA (x, \t ) }$ as above always
 exists, even if $\AAA $ does not have the Eilenberg-Moore object of $\t $ .
\end{rem}

\begin{rem}[Duality: op-codensity monad]
The codual notion of the notion of codensity monad is that of \textit{density comonad}, which is induced by the left Kan extension
of the morphism along itself, assuming its existence. 

The dual notion is herein called \textit{op-codensity monad}. Notice that, if it exists, the op-codensity monad 
of a morphism is induced by the right lifting of the morphism through itself. Finally, of course, we have also the codual notion of the 
op-codensity monad, called herein the \textit{op-density comonad}.

Therefore, we also have factorizations: assuming the existence of the Kleisli object of the 
op-codensity monad of a morphism, we get the \textit{op-semantic factorization}. Codually, we have the
\textit{co-semantic factorization} of a morphism that has the density comonad, provided that the $2$-category
has its co-Eilenberg-Moore object.
\end{rem}

\subsection{Right adjoint morphism}

Recall that an adjunction inside a $2$-category $\AAA $ is a quadruple 
$$\left(l : b\to e,\, p:e\to b ,\, \varepsilon : lp\Rightarrow \id _ e ,\, \eta : \id _ b\Rightarrow pl \right) $$ in which $l, p $ are $1$-cells and $\varepsilon, \eta $ are $2$-cells of $\AAA $ satisfying the \textit{triangle identities}. This means that 
%
\pu
%
\pu
\begin{equation}
\diag{first_diagram_triangle_identity}\qquad\qquad\qquad \diag{second_diagram_triangle_identity}
\end{equation}
are, respectively, the identities $\id _ l : l\Rightarrow l $ and $\id _ p : p\Rightarrow p $. In this case, $p$ is right adjoint to $l$ and we denote the adjunction
by $( l\dashv p, \varepsilon , \eta ) : b\to e $.

If $(l\dashv p,\varepsilon , \eta ): b\to e $ is an adjunction in a $2$-category $\AAA $, $p$ has the codensity monad and the op-density comonad. More precisely, in this case, the
pair $(pl, \id _p \ast \varepsilon )$ is the right Kan extension of $p$ along itself and 
$(lp , \eta \ast \id _ p ) $ is the left lifting of $p $ through itself. Hence, the codensity monad of $p$ coincides with the monad $\t = (b, pl, \id _ p \ast \varepsilon \ast \id _ l , \eta ) $ induced by the 
adjunction, while the op-density comonad coincides with the comonad $(e, lp, \id _ l \ast \eta \ast \id _ p , \varepsilon ) $  induced by the adjunction. 
Codually, if $(l\dashv p,\varepsilon , \eta ): b\to e $ is an adjunction,  the density comonad and the op-codensity  monad induced by $l: b\to e $ are 
the same of
those
induced
 by the adjunction.

Assuming the existence of the Eilenberg-Moore object of the
monad (codensity monad $\t $) induced by the adjunction  $(l\dashv p,\varepsilon , \eta )$, the semantic factorization is the usual factorization 
of the right adjoint morphism through the object of algebras. 
Dually and codually, assuming the existence of the suitable weighted limits and colimits, we get all the four usual factorizations of $l$ and $p$.

More precisely, the op-semantic factorization of $l: b\to e$ is the usual Kleisli factorization
%
\pu
\begin{equation}
\diag{usualfactorizationKleisli}
\end{equation}
w.r.t.  the induced monad ${(b, pl, \id _ p \ast \varepsilon \ast \id _ l , \eta )}$. Codually and dually, the co-semantic and the coop-semantic factorizations are, respectively,
the usual factorization of $l$  through the co-Eilenberg-Moore object, and the usual factorization of $p$ through the co-Kleisli object of the comonad $(e, lp, \id _ l \ast \eta \ast \id _ p , \varepsilon ) $.

\begin{defi}[Preservation of a Kan extension]\label{PreservationKanExtension}
Let $\left( \ran _ g f , \gamma ^{ \ran _ g f }  \right) $ be the right Kan extension of $f: z\to y$ along $g$ in a $2$-category $\AAA $.
A morphism $\delta : y\to y ' $ \textit{preserves the right Kan extension $ \ran _ g f  $ } if 
%
\pu
\begin{equation}
\left(\delta \circ \ran _ g f,\qquad \diag{firstsideoftheequation_definition_preservation_of_Kan_Extensions}\right)
\end{equation}
gives the right Kan extension of the morphism $\delta \cdot f $ along $g $.
Furthermore, the right Kan extension $\left( \ran _ g f , \gamma ^{ \ran _ g f }  \right) $ 
 is \textit{absolute} if it is preserved by any morphism with domain in $y $.
\end{defi}

\begin{rem}[Duality: respecting liftings]
The dual notion of that of preservation of a Kan extension is that of  \textit{respecting} a
lifting.  If a pair $\left( \rlift _ g f , \upgamma ^{ \rlift _ g f }  \right) $ is the right lifting of $f$ through 
$g$, a morphism $\delta : y'\to y  $ \textit{respects the right lifting of $f $ through $g$ } if 
$\left( \left(\rlift _ g f\right)\cdot \delta , \upgamma ^{ \rlift _ g f }\ast \id _ \delta  \right) $ is the right lifting 
of $f\cdot \delta $ through $g$.
\end{rem}

\begin{rem}
In some contexts, such as in the case of  $2$-categories endowed with Yoneda structures~\cite{MR0463261}, we have a stronger notion of Kan extensions: the \textit{pointwise Kan extensions} (for instance, see \cite[Theorem~I.4.3]{MR0280560} for the case of the $2$-category of $V$-enriched categories).
Although this concept plays a fundamental role in the theory of Kan extensions, we do not use
this notion in our main theorem. However, we mention them in our examples and, herein, \textit{a pointwise Kan extension of a functor in $\Cat $ is just a Kan extension that is preserved by any representable functor}. See \cite[Section~X.5]{MR1712872} 
for basic aspects of pointwise Kan extensions and their constructions via conical (co)limits.
\end{rem}

If $(l\dashv p,\varepsilon , \eta ): b\to e $ is an adjunction in a $2$-category $\AAA $, $p$ preserves any right Kan extension with codomain in $b$.
Furthermore:

\begin{theo}[Dubuc-Street~\cite{MR0280560, MR0463261}]\label{DUBUCSTREET}
If $p:e\to b $ is a morphism in a $2$-category $\AAA $, the following statements are equivalent.
\begin{enumerate}[label=\roman*)]
\item The pair $(l , \varepsilon ) $ is the right Kan extension of $\id _ e  $ along $p$ and it is preserved by $p$.\label{1DUBUCSTREET}
\item The pair $(l , \varepsilon ) $ is the right Kan extension of $\id _ e$ along $p$ and it is absolute.\label{2DUBUCSTREET}
\item The morphism $p$ has a left adjoint $l$, with the counit $\varepsilon : lp\Rightarrow \id _ e $.\label{3DUBUCSTREET}
\end{enumerate}
In particular, if $p: e\to b $  has a left adjoint, then it has the codensity monad and the right Kan extension of $p$ along itself is absolute.
\end{theo} 
\begin{proof}
See \cite[Theorem~I.4.1]{MR0280560} or \cite[Propositions~2]{MR0463261}.
\end{proof}

\begin{rem}[Mate correspondence]\label{Mate Correspondence Theorem PhD}
Assume that
\begin{center}
$(l_1\dashv p_1) := (l_1\dashv p_1 , \varepsilon_1  , \eta_1  ): b_1\to e_1   $
and $(l_0\dashv p_0) := (l_0\dashv p_0, \varepsilon _0 , \eta _ 0 ): b_0\to e_0   $ 	
\end{center}
in a $2$-category $\AAA $. Recall that we have the \textit{mate correspondence}~\cite[Proposition~2.1]{MR0357542}.  
More precisely, given $1$-cells $h_b: b_0\to b_1 $ and $ h_e: e_0\to e_1 $ of $\AAA $,
there is a bijection $$\AAA (e_0,b_1)\left(h_b\cdot p_ 0 ,\, p_1\cdot h_e\right)\cong \AAA (b_0,e_1)\left( l_1\cdot h_b,\, h_e\cdot l_ 0 \right) $$ 
defined by
%
\pu
%
\pu
\begin{equation}
\diag{Twocell_Mate_Definition}\qquad\mapsto\qquad \diag{First_Mate_Definition}
\end{equation}
whose inverse is given by 
%
\pu
%
\pu
\begin{equation}
\diag{SecondTwocell_Mate_Definition}\qquad\mapsto\qquad \diag{Second_Mate_Definition}
\end{equation}
The image of a $2$-cell $\beta : h_b\cdot p_ 0 \Rightarrow  p_1\cdot h_e$ by the isomorphism  $\AAA (e_0,b_1)(h_b\cdot p_ 0 , p_1\cdot h_e)\cong \AAA (b_0,e_1)( l_1\cdot h_b, h_e\cdot l_ 0 ) $ above is called the mate of $\beta $ under the adjunction $l_0\dashv p_0 $
and $l_1\dashv p_1 $.
\end{rem}

\section{Main theorems}\label{MainTheoremsSection}
Let $\AAA $ be a $2$-category and $p: e\to b $ a morphism of $\AAA $. Throughout this section, 
we assume that $p$ has the 
codensity monad $\t  = (b , t, m, \eta )$, in which $\left( t, \gamma\right) $ is the right Kan extension of
$p$ along itself. Furthermore, we assume that $\AAA $
has the two-dimensional cokernel diagram $\H _ p :\Delta _\textrm{Str} \to \AAA $ of $p$.

 \textit{We follow the notation respectively established in  
\ref{definitioncodensitymonadsection}, \ref{subsectiondefinicaodehighercokernel} and \ref{definicaodehighercokernel} for the
codensity monad of $p$,	the two-dimensional cokernel diagram of $p$ and the morphisms involved in the diagram. }

 \textit{Moreover, by the universal property of the opcomma object $b\uparrow _ p b $, there is a unique morphism $\ell $ such that
\eqref{eq:defining-l-eqs} and \eqref{eq:defining-l-pasting}  hold. Henceforth, we denoted this morphism by $\ell : b\uparrow _ p b\to b $.}
\begin{equation}\label{eq:defining-l-eqs}
	\ell 	\cdot \d ^0 = \id _ b, \quad \ell \cdot \d ^1 = t, 
\end{equation}
%
\pu
%
\pu
\begin{equation}\label{eq:defining-l-pasting} 
	\diag{firstequationwrtmaintheoremscodensity_monada_Kanextensionopcommaobjects}\qquad =\qquad  \diag{firstequationwrtmaintheoremscodensity_monada_Kanextension} 
\end{equation} 

\begin{prop}\label{rightkanextensionoftheidentityalongd0}
The pair $(\ell , \id _ {\id _ {b}  } ) $ is the right Kan extension of $\id _{b}$  along $\d ^0 : b\to b\uparrow _ p b $.
\end{prop}
\begin{proof}
By Lemma \ref{teoremadeultimomomento}, since $(t , \gamma ) $ is the right Kan extension of $\ell\cdot \d^0\cdot p=p$ along $p$, we get that $(\ell , \gamma ) $
is the right Kan extension of $\ell\cdot \d^0 = \id _ b $ along $\d ^0 $.
\end{proof}

\begin{prop}[Condition]\label{conditionofpreservation}
The right Kan extension $(t , \gamma ) $ of $p$ along itself  is preserved by $\d ^0 : b\to b\uparrow _ p b $ if and only if $\ell $ is left adjoint to $\d ^0 $. In
this case, we have an adjunction $$(\ell\dashv \d ^0, \id _ {\id _ b}, \ueta ): b\uparrow _ p b\to b .$$
\end{prop}
\begin{proof}
It follows from Lemma \ref{teoremadeultimomomento} and Proposition \ref{rightkanextensionoftheidentityalongd0} that $(\d ^0 \cdot \ell , \id _ {\d ^0} ) $ is the right Kan extension of $\d ^0 $ along itself
if and only if 
$(\d^0 \cdot t , \id_{\d ^0} \ast \gamma ) $ is the right Kan extension
of $\d ^0 \cdot p $ along $p$.

By Proposition \ref{rightkanextensionoftheidentityalongd0} and Theorem \ref{DUBUCSTREET}, we know that $\d ^0 $ preserves
the right Kan extension $(\ell ,\id _ {\id _ b} ) $ of $\id _ {b} $ along $\d ^0 $  if and only if $\ell \dashv \d ^0 $ with the counit $\id _ b$.
\end{proof}

We postpone the discussion on examples and counterexamples of morphisms satisfying the condition of Proposition \ref{conditionofpreservation} (see \ref{toyexamples}).
For now, we only observe that:
\begin{prop}\label{okforrightadjoints}
	If the morphism $p: e\to b $ has a left adjoint, then it satisfies the condition of Proposition \ref{conditionofpreservation}. 
	
	In particular,
	since $\Cat $ has the two-dimensional cokernel diagram  of any functor, any right adjoint functor satisfies the condition of 
	Proposition \ref{conditionofpreservation}.
\end{prop}
\begin{proof}
	By the Dubuc-Street Theorem (Theorem \ref{DUBUCSTREET}), if $p: e\to b $ is a right adjoint morphism of the $2$-category $\AAA $, we get, in particular, that $\ran_p p \cong p\circ \ran_p \id_ e $ exists and is absolute. 
\end{proof}

We prove below that the condition of Proposition \ref{conditionofpreservation} on the morphism $p$ also implies that \textit{$\D^0\d^0 $ has a left adjoint $\ella $}. This result is going to be particularly useful to the proof of our main theorem (Theorem~\ref{principal}).

\begin{prop}\label{definitionella}
Assume that $p: e\to b $ satisfies the condition of Proposition \ref{conditionofpreservation}.
There is an adjunction $\left(\ella\dashv  \D ^0 \cdot \d ^0, \id _ {\id _ b }, \etta  \right) : b\uparrow _ p b\uparrow _ p b\to b $. Moreover, the equations
$$\ella \cdot \D ^2 = t\cdot \ell ,\qquad   \ella \cdot \D ^0  =  \ell , \qquad \etta \ast \id _{\D ^0} =\id_{\D ^0 } \ast\ueta $$   
are satisfied. Furthermore, the $2$-cell $\etta _2 $ given by the pasting
$$\xymatrix{
b\uparrow _p b 
\ar@{{}{ }{}}[dd]|-{=}
\ar@/_1.5pc/[dd]|-{\D ^2 \d ^0  \ell}
\ar@/^1.5pc/[dd]|-{\D ^0 \d^1 \ell}
\ar@/_5pc/@{{}{ }{}}[dd]|-{\xRightarrow{\id _ {\D ^2 }\ast \ueta } }
\ar@/^5pc/@{{}{ }{}}[dd]|-{\xRightarrow{ \id_ {\D ^0 }\ast \ueta\ast \id _ {\d ^1 \ell }   }  }
\ar@/_10.5pc/[dd]|-{\D ^2}
\ar@/^10.5pc/[dd]|-{\D ^0 \d ^0 \ell \d ^1 \ell =   \D ^0 \,\d ^0 \, t\, \ell }
\\
\\
b\uparrow _ p b\uparrow _ p b
}$$
is equal to $\etta\ast \id _ {\D ^2}$.
\end{prop}
\begin{proof}
By the universal property of the pushout $b\uparrow _ p b\uparrow _ p b$  of $\d ^0 $  along $\d ^1 $ (as defined in \eqref{pushoutsquaredefiningbububp}), we have that:
\begin{itemize}
\renewcommand\labelitemi{--}
\item there is a unique morphism 
$\ella : b\uparrow _ p b\uparrow _ p b\to b  $ 
such that the diagram
%
\pu
\begin{equation}
\diag{basicdiagramthatcommutesforthedefinitionofella} 
\end{equation}
commutes; 
\item there is a unique $2$-cell  
$ \etta : \id _ {b\uparrow _ p b\uparrow _ p b}\Rightarrow \D ^0\d ^0 \cdot \ella $ such that 
$\etta \ast \id _{\D ^0} =\id_{\D ^0 } \ast\ueta $ and $\etta\ast \id _ {\D ^2} = \etta _2  $, because
$ \etta _ 2 \ast \id _{\d ^0 } $ is equal to the composition of $2$-cells
$$\xymatrix@=4em{
\D ^2\cdot \d ^0 
=
\D ^0 \cdot \d ^1
\ar@{=>}[r]^-{\id _ {\D ^0 } \ast \ueta \ast \id _ {\d ^1 }   }
&
\D ^0\,\d ^0 \cdot \ell\, \d ^1
=
\D ^0 \, \d ^0 \cdot t
}$$
since $\id _ {\D ^2 }\ast   \ueta \ast \id _ {\d ^ 0 } = \id _ { \D ^2\,\d ^0  } $, 
and 
$\id_ {\D ^0 }\ast \ueta\ast \id _ {\d ^1 \ell }\ast \id _ {\d ^ 0 } =  \id _ {\D ^0 } \ast \ueta \ast \id _ {\d ^1 } $.  
\end{itemize}

Furthermore, 
$\etta \ast \id _ {\D ^0 \d ^0} =  \etta \ast \id _ {\D ^0 }\ast \id _{ \d ^0} =  \id _ {\D ^0 }\ast\ueta \ast \id _{ \d ^0}$
is a horizontal composition of identities, since $\ueta\ast \id _ {\d ^0} = \id _{\d ^0 }  $. This proves one of the triangle identities for the adjunction
$\ella\dashv \D ^0 \d ^0 $. 

Finally, by the universal property of the pushout $b\uparrow _ p b\uparrow _ p b $ of $\d ^0 $ along $\d ^1 $, 
the $2$-cell $\id _ {\ella }\ast \etta  $ is the identity on $\ella $, since:
\begin{itemize}
\renewcommand\labelitemi{--}
\item $\left(\id _ {\ella }\ast \etta\right) \ast \id _ {\D ^0 } $ is equal to
$$\id _ {\ella }\ast \etta \ast \id _ {\D ^0 }  = \id _ {\ella }\ast  \id_{\D ^0 } \ast\ueta = \id _ {\ell } \ast \ueta  = \id _ {\ell } = \id _ {\ella\D ^0}; $$
\item $\left(\id _ {\ella }\ast \etta\right) \ast \id _ {\D ^2 } = \id _ {\ella }\ast \etta _2  $ is, by the definition of $\ella $ and $\etta _2 $, equal to
$$\xymatrix{
b\uparrow _p b 
\ar[dd]|-{t\cdot \ell}
\ar@/_2.5pc/@{{}{ }{}}[dd]|-{\xRightarrow{\id _ {t }\ast \id_ {\ell } \ast \ueta } }
\ar@/^2.5pc/@{{}{ }{}}[dd]|-{\xRightarrow{ \id_ {\ell }\ast \ueta\ast \id _ {\d ^1 \ell }   }  }
\ar@/_5.5pc/[dd]|-{t\cdot \ell}
\ar@/^5.5pc/[dd]|-{t\cdot \ell}
\\
\\
b
}$$
which is a vertical composition of identities, since $\id _ {\ell }\ast \ueta $ is equal to the identity by the triangle identity of the adjunction
$(\ell\dashv \d ^0 , \id _ {\id _ b} , \ueta ) $. 
\end{itemize}
This completes the proof that $(\ella \dashv \D ^0 \d ^0 ,\, \id _ {\id _ b},\, \etta ) $ is an adjunction.
\end{proof}
In order to prove Theorem \ref{principal}, we consider the $2$-cells defined in Lemma \ref{important2cells}. Before defining them,
it should be noted that:

\begin{lem}[$\ella\cdot \D ^1 $]\label{ellaD1}
Assume that $p$ satisfies the condition of Proposition \ref{conditionofpreservation}. The morphism $\ella\cdot \D ^1 $ is the unique morphism such that the equations
$$\left(\ella \D ^1\right)\cdot \d ^1 = t^2, \quad\left(\ella \D ^1\right) \cdot \d ^0 = \id _ b\quad\mbox{and}\quad \id _ {\ella \D ^1} \ast \upalpha = \left( \id _ {t} \ast \gamma \right) \cdot \gamma $$
hold.
\end{lem}
\begin{proof}
In fact, by the definitions of $\ell $ (see Proposition \ref{rightkanextensionoftheidentityalongd0}) and 
$\ella $ (see Proposition \ref{definitionella}),
 the equations
\begin{eqnarray*}
	\left(\ella \D ^1\right)\cdot \d ^1 = \ella\cdot \D ^2\d ^1 = t\ell\cdot \d ^1 = t^2, &&
\left(\ella \D ^1\right) \cdot \d ^0 = \ella\cdot \D ^0\d ^0 = \ell\cdot \d^0 = \id _ b,  
\\
\id _ {\ella }\ast \id _ {\D ^0 }\ast \upalpha = \id _ {\ell }\ast \upalpha = \gamma &\mbox{ and } &
\id _ {\ella }\ast \id _ {\D ^2 }\ast \upalpha = \id _ {t\ell }\ast \upalpha = \id _ t\ast \gamma
\end{eqnarray*} 
hold. Therefore, by the definition of $\D ^1 $ (see \ref{subsectiondefinicaodehighercokernel} and, more particularly, \eqref{definitionofD1} and \eqref{associativityequationhighercokernel}),  we get the results.
\end{proof}

\begin{lem}[$\uptheta$ and $\uplambda $]\label{important2cells}
Assume that $p: e\to b $ satisfies the condition of Proposition \ref{conditionofpreservation}.
There are $2$-cells
$$
\uptheta : \s ^0 \Rightarrow \ell \, : b\uparrow _ p b\to b , \qquad
\uplambda :  \ella\cdot \D  ^1\Rightarrow \ell \, : b\uparrow _ p b\to b
$$ 
such that the equations
$$\uptheta \ast \id _ {\d ^1} = \eta , \quad \uptheta \ast \id _ {\d ^0} = \id _ {\id _ b}, \quad \uplambda\ast \id _ {\d ^1} = m , \quad \uplambda\ast \id _ {\d ^0}
= \id _ {\id _ b} 
$$
are satisfied.
\end{lem}
\begin{proof}
In fact, by the universal property of the opcomma object $b\uparrow _p b $ of $p $ along itself:
\begin{itemize}
\renewcommand\labelitemi{--}
\item there is a unique $2$-cell $\uptheta : \s ^0 \Rightarrow \ell $ such that 
$\uptheta \ast \id _ {\d ^1} = \eta $ and $\uptheta \ast \id _ {\d ^0} = \id _ {\id _ b} $, since 
$$\left( \id _ \ell \ast \upalpha \right)\cdot \left( \eta \ast \id _ p \right) = \gamma \cdot \left( \eta \ast \id _ p \right)  = \id _ p = \id _ {\s ^0 }\ast \upalpha
$$
by the definitions of $\ell $, $\eta $ and $\s ^0$;
\item there is a unique $2$-cell $\uplambda :  \ella\cdot \D  ^1\Rightarrow \ell $ such that 
$\uplambda\ast \id _ {\d ^1} = m $ and $ \uplambda\ast \id _ {\d ^0}
= \id _ {\id _ b} $, since it follows from the definitions of $\ell $ and  $m : t ^2\Rightarrow t $ that
$$ \left( \id _ \ell \ast \upalpha \right)\cdot  (m\ast \id _ p ) = \gamma \cdot (m\ast \id _ p ) = \gamma \cdot (\id _ t \ast \gamma ) = \id _ {\ella\cdot \D ^1}\ast \upalpha   $$
by Lemma \ref{ellaD1}.
\end{itemize}
\end{proof}

\begin{theo}\label{principal}
Assume that $p$ satisfies the condition of Proposition \ref{conditionofpreservation}. We get that
the factorization defined in \eqref{imaginglaxdescentfactorizationofthehighercokernel}
is $2$-naturally (in $x$) isomorphic  to the factorization defined in \eqref{imagingsemanticfactorizationofamorphismthathascodensitymonad}.
\end{theo}
\begin{proof}
Recall that \eqref{imaginglaxdescentfactorizationofthehighercokernel} is the factorization of $\AAA (x, p ) $  
 induced by the pair $$\left( \AAA (x, p), \AAA (x, \upalpha )  \right) $$ and the universal property of the 
lax descent category of
$\AAA (x,\mathcal{H}_p -) : \Delta _ \mathrm{Str}\to \Cat $;
and \eqref{imagingsemanticfactorizationofamorphismthathascodensitymonad} is the factorization of $\AAA (x, p ) $ induced by the pair $(\AAA (x, p ), \AAA (x, \gamma )) $ and the universal property of the 
Eilenberg-Moore category $\AAA (x, b) ^{\AAA (x, \t ) }$ of the monad $\AAA (x, \t ) $.

Observe that, 
since $(\ell\dashv \d ^0 , \id _ {\id _ b}, \ueta ): b\uparrow _ p b\to b  $ is an adjunction by Proposition \ref{conditionofpreservation},
 for each morphism $h : x\to b $ of $\AAA $ (\textit{i.e.} for each object of $\AAA (x, b ) $), there is a bijection 
$$\AAA (x, b\uparrow _ p b )(\d ^1\cdot h, \d^0\cdot h )\cong \AAA (x, b)(\ell\cdot\d ^1\cdot h, \ell\cdot\d^0\cdot h ) =
 \AAA (x, b)(t\cdot h,  h ) $$ defined by $\beta\mapsto \id _ {\ell }\ast \beta $,
that is to say, the mate correspondence under the identity adjunction $\id _ x\dashv \id _ x $  and the adjunction $(\ell\dashv \d ^0 , \id _ {\id _ b}, \ueta )$, see Remark \ref{Mate Correspondence Theorem PhD}.

Given an object  $h $ of $\AAA (x, b) $, we prove below that a $2$-cell $\beta : \d ^1\cdot h\Rightarrow \d^0\cdot h  $ 
satisfies the descent associativity and identity (\eqref{Associativityequationdescent} and  \eqref{Identityequationdescent}) w.r.t. $\H_p $ if and only if its corresponding
$2$-cell $\id _ {\ell }\ast \beta $ satisfies 
the algebra associativity and identity equations w.r.t. $\t $ (\eqref{Algebraassociativityextensions} and
\eqref{i_Algebraassociativityextensions}).

\begin{enumerate}
\item Observe that, given a $2$-cell $\beta : \d ^1\cdot h\Rightarrow \d^0\cdot h  $,
by the definition of $\uptheta $ in Lemma \ref{important2cells}, we get that
\begin{eqnarray*} 
\id _ {\s ^0}\ast \beta &=& \id _ {\s^0\cdot \d ^0 \cdot h}\cdot \left( \id _ {\s ^0}\ast \beta\right)\\
&=& \left( \id _ { \id _ b   }\ast \id _ h \right) \cdot \left( \id _ {\s ^0}\ast \beta\right)\\
&=& \left( \left(\uptheta\ast \id _ {\d ^0} \right)\ast \id _ h \right)\cdot \left( \id _ {\s ^0}\ast \beta\right)\\
&=& \left( \uptheta\ast \id _ {\d^0 \cdot h}\right)\cdot \left( \id _ {\s ^0}\ast \beta\right)\\
&=& \uptheta \ast \beta
\end{eqnarray*}  
which, by the interchange law,  is equal to the left side of the equation
$$\left(\id _ \ell  \ast \beta \right)\cdot \left( \uptheta \ast \id _ {\d ^1}\ast \id _ h \right ) = \left(\id _ \ell  \ast \beta \right)\cdot 
\left( \eta\ast \id _ h \right ) $$
which holds by Lemma \ref{important2cells}. Thus, of course, $\left(\id _ \ell  \ast \beta \right)\cdot 
\left( \eta\ast \id _ h \right ) $ is the identity on $h$ if and only if $\id _ {\s ^0}\ast \beta = \left(\id _ \ell  \ast \beta \right)\cdot 
\left( \eta\ast \id _ h \right ) $ is the identity on $h$ as well. This proves that  $\left(h, \beta\right) $ satisfies the descent identity \eqref{Identityequationdescent} w.r.t. $\H_p $ if and only if $\left( h,  \id _ \ell\ast \beta\right) $
satisfies the algebra identity equation w.r.t. $\t $ \eqref{i_Algebraassociativityextensions}.
\item Recall the adjunction $\left(\ella\dashv \D ^0 \d ^0 ,\, \id_{\id _b},\, \etta\right) $ of Proposition 
\ref{definitionella}.  Given a $2$-cell $\beta : \d ^1\cdot h\Rightarrow \d^0\cdot h  $, consider the $2$-cells defined by the pastings below.
%
\pu
%
\pu
\begin{eqnarray}
\upbeta_1 &\qquad \colon = \qquad & \diag{betaumofthemainproof}\\
\upbeta _c &\qquad\colon = \qquad & \diag{betacofthemainproof}
\end{eqnarray}
We have that $\upbeta _c = \upbeta _ 1 $ if, and only if, 
 $\left(h, \beta\right) $  satisfies the descent associativity \eqref{Associativityequationdescent} w.r.t. $\H _p $. Therefore,  by the mate correspondence
under the identity adjunction $\id_ x \dashv \id _ x $ and the adjunction 
$(\ella\dashv \D ^0 \d ^0 , \id_{\id _b}, \etta) $,
we conclude that $\left(h, \beta\right) $  satisfies the \textit{descent associativity} \eqref{Associativityequationdescent} w.r.t. $\H _p $ if, and only if,
\begin{equation}\label{mateequationofdescentassociativityfortheproof}
\id _ {\ella }\ast \upbeta _ c = \id _ {\ella }\ast \upbeta _ 1.
\end{equation} 
Now, we observe that:
\begin{enumerate}
\item Since
$\ella \cdot \D ^2 = t\cdot \ell $  and $  \ella \cdot \D ^0  =  \ell $, 
we have that
%
\pu
\begin{equation}\label{eq:assoc1} 
\id _ {\ella }\ast \upbeta _ c\quad = \quad \diag{algebra_rightsideoftheequationassociativityofmonadforthemainproof}
\end{equation}
\item By Lemma \ref{important2cells},
\begin{eqnarray*}
\id _{\ella}\ast \upbeta _ 1 &=&
\left( \id_ {\ella\D ^1}\ast \beta \right)\\
 &=& \left( \id _ {\id _ b}\ast \id _ h \right) \cdot\left( \id_ {\ella\D ^1}\ast \beta \right)\\
& = &  \left( \uplambda\ast \id _ {\d ^0 }\ast \id _ h \right) \cdot \left( \id_ {\ella\D ^1}\ast \beta \right) 
\end{eqnarray*}
which, by the interchange law and Lemma \ref{important2cells}, is equal to 
%
\pu
\begin{equation}\label{eq:assoc2} 
\uplambda \ast \beta = (\id _ \ell \ast \beta )\cdot ( \uplambda\ast \id _ {\d ^1}\ast \id _ {h} ) \quad =\quad \diag{algebra_leftsideoftheequationassociativityofmonadforthemainproof}. 
\end{equation}

\end{enumerate}
Therefore \eqref{mateequationofdescentassociativityfortheproof} holds if, and only if, the pasting \eqref{eq:assoc1} is equal to \eqref{eq:assoc2}.   

This completes the proof that  $\left( h , \beta\right) $ satisfies descent associativity w.r.t. $\H _p $  if, and only if, $\id _ \ell \ast \beta $ satisfies the algebra associativity \eqref{Algebraassociativityextensions} w.r.t. $\t $.
\end{enumerate}
The above implies that the association $(h, \beta )\mapsto (h, \id _ \ell\ast \beta ) $ 
gives a bijection between the objects of  $ \lim \left( \mathfrak{D} , \AAA (x, \H _ p - )\right) $ and
$\AAA (x, b)^{\AAA (x, \t )} $. 

Given objects $(h _ 1, \beta _1 )$ and $(h_0 , \beta _ 0 ) $ 
of $\lim \left( \mathfrak{D} , \AAA (x, \H _ p - )\right)$, by the mate correspondence under the identity adjunction and $\ell\dashv \d ^0 $, a $2$-cell $$\xi : h_1\Rightarrow h_0 : x\to b $$ 
satisfies the equation
$$\xymatrix@=1.5em{
&
x
\ar@{{}{ }{}}@/^0.6pc/[dd]|-{\xRightarrow{\mbox{ }\beta _ 0\mbox{ } } }
\ar[rd]|-{h_0}
\ar@/^1.2pc/[ld]|-{h_0}
\ar@/_1.2pc/[ld]|-{h_1}
\ar@{}[ld]|-{\xRightarrow{\mbox{ } \xi  \mbox{ } } }
&
&
&
x
\ar@{{}{ }{}}@/_0.6pc/[dd]|-{\xRightarrow{\mbox{ }\beta _ 1\mbox{ } } }
\ar[ld]|-{h_1 }
\ar@/^1.2pc/[rd]|-{h_0}
\ar@/_1.2pc/[rd]|-{h_1}
\ar@{}[rd]|-{\xRightarrow{\mbox{ } \xi  \mbox{ } }} 
&
\\
b
\ar[rd]|-{\d ^1  }
&
&
b
\ar[ld]|-{\d ^0  }
&
b
\ar@{}[l]|-{=}
\ar[rd]|-{\d^1 }
&
&
b
\ar[ld]|-{\d^0 }
\\
&
b\uparrow _ p b
&
&
&
b\uparrow _p b 
&
}$$ 
if and only if the mate of the left side is equal to the mate of the right side, which means
$$\xymatrix@=1.5em{
&
x
\ar@{{}{ }{}}@/^0.6pc/[dd]|-{\xRightarrow{\mbox{ }\beta _ 0\mbox{ } } }
\ar[rd]|-{h_0}
\ar@/^1.2pc/[ld]|-{h_0}
\ar@/_1.2pc/[ld]|-{h_1}
\ar@{}[ld]|-{\xRightarrow{\mbox{ } \xi  \mbox{ } } }
&
&
&
x
\ar@{{}{ }{}}@/_0.6pc/[dd]|-{\xRightarrow{\mbox{ }\beta _ 1\mbox{ } } }
\ar[ld]|-{h_1 }
\ar@/^1.2pc/[rd]|-{h_0}
\ar@/_1.2pc/[rd]|-{h_1}
\ar@{}[rd]|-{\xRightarrow{\mbox{ } \xi  \mbox{ } }} 
&
\\
b
\ar[rd]|-{\d ^1  }
&
&
b
\ar[ld]|-{\d ^0  }
&
b
\ar@{}[l]|-{=}
\ar[rd]|-{\d^1 }
&
&
b
\ar[ld]|-{\d^0 }
\\
&
b\uparrow _ p b\ar[d]^-{\ell }
&
&
&
b\uparrow _p b \ar[d]^-{\ell }
&
\\
&
b
&
&
&
b
&
}$$ 
which is precisely the condition of being a morphism of algebras in $\AAA (x, b)^{\AAA (x, \t )} $.
In other words, this proves that $\xi $ gives a morphism between $(h _ 1, \beta _1 )$ and $(h_0 , \beta _ 0 ) $ 
in $\lim \left( \mathfrak{D} , \AAA (x, \H _ p - )\right) $ if and only if it gives a morphism between $(h _ 1, \id _ \ell\ast \beta _1 )$ and $(h_0 , \id _ \ell\ast \beta _0 ) $ in $\AAA (x, b)^{\AAA (x, \t )} $.

Finally, given the facts above, we can conclude that we actually can define
\begin{eqnarray*}
\lim \left( \mathfrak{D} , \AAA (x, \H _ p - )\right) &\to & \AAA (x, b)^{\AAA (x, \t )}\\
(h, \beta ) &\mapsto & (h, \id _\ell\ast \beta ) \\
\xi   &\mapsto & \xi 
\end{eqnarray*}
which is clearly functorial and, hence, it defines an invertible functor (since it is bijective on objects and fully faithful as proved above). 

This invertible
functor is $2$-natural in $x$, giving a $2$-natural isomorphism between  \eqref{imaginglaxdescentfactorizationofthehighercokernel} and 
\eqref{imagingsemanticfactorizationofamorphismthathascodensitymonad}.
\end{proof}

\begin{theo}[Main Theorem]\label{maintheoremofthepaper}
Assume that  $\ran _p p $ exists and is preserved
by the morphism $\delta  ^0 : b\to b\uparrow_p b  $. 	 We have that the semantic factorization \eqref{semanticfactorizationofamorphismthathascodensitymonad} of $p$ is 
isomorphic to the 
semantic lax descent factorization \eqref{laxdescentfactorizationofthehighercokernel} of $p$, either one existing if the other does.
\end{theo}
\begin{proof}
It is clearly a direct consequence of Theorem \ref{principal}.
\end{proof}

Recall that, since the result above works for any $2$-category, we have the dual results. For instance, we have Theorem \ref{codualmaintheoremofthepaper} and Theorem \ref{dualmaintheoremofthepaper}.

\begin{theo}[Codual]\label{codualmaintheoremofthepaper}
Let $l: b\to e $ be a morphism of $\AAA $ satisfying the following conditions:
\begin{enumerate}
\item  $\AAA $ has the
two-dimensional cokernel diagram  of $l$;
\item the left Kan extension $\lan _ l l $ of $l$ along itself exists (that is to say, $l$ has the density comonad);
\item the left Kan extension $\lan _ l l $ is preserved by $\delta _ {l\uparrow l} ^1 : e\to l\uparrow l  $.
\end{enumerate}
The diagram of the co-semantic factorization of $l$ is isomorphic to the  semantic lax 
descent factorization of $l$ 
either one existing if the other does.
\end{theo}
\begin{proof}
By the observations on the self-coduality of the factorization in Remark \ref{smaenotionoffactorization}, we
get the result from Theorem \ref{maintheoremofthepaper}.
\end{proof}
\begin{theo}[Dual]\label{dualmaintheoremofthepaper}
Let $l: b\to e $ be a morphism of $\AAA $ satisfying the following conditions:
\begin{enumerate}
\item  $\AAA $ has the
higher kernel of $l$;
\item the right lifting of $l$ through itself exists (that is to say, $l$ has the op-codensity monad);
\item the right lifting of $l$ through itself is respected by the arrow $\delta ^ {l\downarrow l} _0 : l\downarrow l\to b  $.
\end{enumerate}
The diagram of the op-semantic factorization of $l$  
is $2$-naturally isomorphic to the semantic lax codescent factorization of $l$ (see Remark \ref{smaenotionoffactorization} and \eqref{laxdcoescentfactorizationofthehigherkernel}).
\end{theo}

As a consequence of Theorem \ref{maintheoremofthepaper} and its duals, by Proposition \ref{okforrightadjoints}, we get:

\begin{theo}[Adjunction]\label{ADJUNCTIONVERSION}
Let $(\mathfrak{l} \dashv \mathfrak{p}, \varepsilon , \eta  ): \mathfrak{b}\to \mathfrak{e} $ be an adjunction in $\AAA $. We have the following:
\begin{enumerate}
\item if $\AAA $ has the two-dimensional cokernel diagram  of $\mathfrak{p}$, then the semantic lax 
descent factorization~\eqref{laxdescentfactorizationofthehighercokernel} of $\mathfrak{p}$
coincides up to isomorphism with the usual factorization of $\mathfrak{p}$ through the Eilenberg-Moore object, either one existing if the other does;
\item  if $\AAA $ has the two-dimensional kernel diagram of $\mathfrak{l} $, then the semantic lax 
codescent factorization of $\mathfrak{l} $  
coincides up to isomorphism with the usual factorization of $\mathfrak{l}$ through the Kleisli object, either one existing if the other does;
\item \label{coalgebrasmain} if $\AAA $ has the two-dimensional cokernel diagram  of $\mathfrak{l} $, then the  
semantic lax 
descent factorization of $\mathfrak{l}$
coincides up to isomorphism with the usual factorization of $\mathfrak{l}$ through the co-Eilenberg-Moore object, either one existing if the other does;
\item if $\AAA $ has the two-dimensional kernel diagram  of $\mathfrak{p} $, then the semantic lax 
codescent factorization  of $\mathfrak{p}$
coincides up to isomorphism with the usual factorization of $\mathfrak{p} $ through the co-Kleisli object, either one existing if the other does.
\end{enumerate}
\end{theo}

\subsection{(Counter)examples of morphisms satisfying Proposition \ref{conditionofpreservation}}
\label{toyexamples}
	Even $\Cat $ has morphisms that do not satisfy the condition of Proposition \ref{conditionofpreservation}. 
	
	For instance, the inclusion of the domain
	$ d^1 : \mathsf{1}\to\mathsf{2} $ has the codensity monad. 	More precisely $\ran _{d^1} d^1 $ is given by $\id _ {\mathsf{2}} : \mathsf{2}\to \mathsf{2} $ with the unique $2$-cell (natural
	transformation)
	$d^1\Rightarrow d^1 $. However,
	in this case, $\delta _ {d^1\uparrow d^1} ^0 $ is the inclusion
	$$\xymatrix{
		\mathsf{0}
		\ar[d]
		\ar@{}[rrrd]|-{\mapsto}
		&
		&
		&
		\mathsf{0}
		\ar[d]
		&
		\mathsf{0}'
		\ar[l]
		\ar[d]
		\\
		\mathsf{1}
		&
		&
		&
		\mathsf{1}
		&
		\mathsf{1}'
	}$$
	which does not preserve the terminal object, since $\mathsf{2} $ has terminal object and $d^1\uparrow d^1 $ does not. Hence $\delta _ {d^1\uparrow d^1} ^0 $ does not have a left adjoint. Actually, 
	it even does not have a codensity monad. Therefore the condition of Proposition \ref{conditionofpreservation} does not hold for $ d^1 : \mathsf{1}\to\mathsf{2} $.
	
	It should be noted that 
	$d ^1 $ is left adjoint to $s ^0 $ and, hence, it does satisfy the codual of the condition of Proposition \ref{conditionofpreservation}. 	More precisely, since $d ^1 : \mathsf{1}\to \mathsf{2} $ is a left adjoint functor, 
	it satisfies the hypothesis of \ref{coalgebrasmain} of Theorem \ref{ADJUNCTIONVERSION}. Hence
	the co-semantic factorization (usual factorization
	through the category of coalgebras) coincides with the semantic lax descent factorization  
	of $d ^1 $. These factorizations are given by  $$d ^1 = d^1 \circ \id _ {\mathsf{1} } .$$

	By Proposition \ref{okforrightadjoints}, any right adjoint morphism satisfies
	Proposition \ref{conditionofpreservation}. The converse is false, that is to say, the condition
	of Proposition \ref{conditionofpreservation} does not imply the existence of a left adjoint.	
	There are simple counterexamples in $\Cat $. In order to construct such an example, we observe that:
\begin{lem}\label{lemarepetido}
Let  $\terminall  _ e: e\to \mathsf{1} $ be a functor between a small category $e $ and the terminal category.
We have that $ \ran _ {\terminall  _ e}\terminall  _ e$ and $ \lan _ {\terminall  _ e}\terminall  _ e$  are given by the identity on $\mathsf{1} $. Therefore the semantic factorization 
and the co-semantic factorization are both given by 
$$\terminall  _ e = \id _ {\mathsf{1} }\circ \terminall  _ e .$$
Moreover, $\terminall  _ e  $ has a left adjoint (right adjoint) if and only if $e$ has initial object (terminal object).
\end{lem}	
Since the \textit{thin} category $\mathbb{R} $ corresponding to the usual preordered set of real numbers does not have initial or terminal objects, the only
	functor $\terminall  _ \mathbb{R} : \mathbb{R}\to \mathsf{1} $ does not have any adjoint. However, it is clear
	that every functor $\mathsf{1}\to b $ preserves the (conical) limit of  $\mathbb{R}\to \mathsf{1} $
	and, hence, any such functor  does preserve  $\ran _ {\terminall  _ \mathbb{R}} \terminall  _ \mathbb{R} $. In particular, $\terminall  _ \mathbb{R}  $
	does satisfy Proposition \ref{conditionofpreservation}.
	
	This proves that, although the morphism $\terminall  _ \mathbb{R} : \mathbb{R}\to \mathsf{1} $  does not satisfy any of the versions of Theorem \ref {ADJUNCTIONVERSION},
	it does satisfy the conditions of Theorem \ref{maintheoremofthepaper}. Hence the
	semantic lax descent factorization 
	of $\terminall  _ \mathbb{R} $ (Eq.~\ref{laxdescentfactorizationofthehighercokernel}) coincides with the semantic factorization of $\terminall  _ \mathbb{R} $. In this case, by Lemma \ref{lemarepetido},  both
	factorizations are given by $$\terminall  _ \mathbb{R} = \id _ {\mathsf{1} }\circ \terminall  _ \mathbb{R} .$$

\begin{rem}
Although  (by Lemma \ref{lemarepetido}) the codensity monad and the density comonad 
of $\terminall  _{\mathsf{1}\sqcup \mathsf{1}} $ are the identity on $\mathsf{1} $, the functor $\terminall  _{\mathsf{1}\sqcup \mathsf{1}} $ does not satisfy the condition of Proposition \ref{conditionofpreservation} nor the codual.

In fact,
	a functor $\mathsf{1} \to b $ preserves 
	$\ran _ {\terminall  _{\mathsf{1}\sqcup \mathsf{1}} } \terminall  _{\mathsf{1}\sqcup \mathsf{1}}  $ ($\lan _{\terminall  _{\mathsf{1}\sqcup \mathsf{1}} } \terminall  _{\mathsf{1}\sqcup \mathsf{1}}  $)
	if and only if  $\mathsf{1} \to b $ preserves binary products (binary coproducts) which does happen if and only if the image of $\mathsf{1}\to b $ is a preterminal
	(preinitial) object (see \cite[Remark~4.5]{arXiv:1711.02051} for instance). 
	
	The opcomma category of $\terminall  _{\mathsf{1}\sqcup \mathsf{1}}  $ along itself is the category with two distinct objects and two parallel arrows between them: hence it does
	not have any preterminal or preinitial objects. This shows that neither  $\ran _ {\terminall  _{\mathsf{1}\sqcup \mathsf{1}}}  \terminall  _{\mathsf{1}\sqcup \mathsf{1}}  $ nor $\lan _{\terminall  _{\mathsf{1}\sqcup \mathsf{1}}}  \terminall  _{\mathsf{1}\sqcup \mathsf{1}}  $ is preserved by any functor $\mathsf{1}\to \terminall  _{\mathsf{1}\sqcup \mathsf{1}} \uparrow \terminall  _{\mathsf{1}\sqcup \mathsf{1}}  $.

This proves that the functor $\terminall  _{\mathsf{1}\sqcup \mathsf{1}}  : \mathsf{1}\sqcup \mathsf{1}\to \mathsf{1} $
does not satisfy the hypotheses of Theorem \ref{maintheoremofthepaper}. Since $\Cat $ has Eilenberg-Moore objects, two-dimensional cokernel diagrams  and lax descent objects,
we have the semantic lax descent factorization of $\terminall  _{\mathsf{1}\sqcup \mathsf{1}} $
and the semantic factorization. However, in this case, they do not coincide. More precisely, they are respectively given by the
commutative triangles below.
$$\xymatrix{
\mathsf{1}\sqcup \mathsf{1}
\ar[rr]|-{\terminall  _{\mathsf{1}\sqcup \mathsf{1}}}
\ar[rd]|-{\id _ {\mathsf{1}\sqcup \mathsf{1} }  }
&
&
\mathsf{1}
&
&
\mathsf{1}\sqcup \mathsf{1}
\ar[rr]|-{\terminall  _{\mathsf{1}\sqcup \mathsf{1}}}
\ar[rd]|-{\terminall  _{\mathsf{1}\sqcup \mathsf{1}}}
&
&
\mathsf{1}
\\
&
\mathsf{1}\sqcup \mathsf{1}
\ar[ru]|-{\terminall  _{\mathsf{1}\sqcup \mathsf{1}} }
&
&
&
&
\mathsf{1}
\ar[ru]|-{\id _ {\mathsf{1} }  }
&
}$$

\end{rem}

\section{Monadicity and effective faithful morphisms}

In this section, we show direct consequences of Theorem \ref{maintheoremofthepaper} on monadicity.
Henceforth, whenever a $2$-category $\AAA $ has the two-dimensional cokernel diagram  $\H _ p $ of a morphism $p: e\to b$, we use the notation of 
\ref{subsectiondefinicaodehighercokernel} and \ref{definicaodehighercokernel}. If $\AAA $ has the two-dimensional
kernel diagram of a morphism $l: e\to b $,
we use the notation of Remark \ref{smaenotionoffactorization}.

Recall that a morphism $\mathfrak{p}: \mathfrak{e}\to \mathfrak{b} $ of a $2$-category $\AAA $ is an equivalence if there is are a 
morphism $\mathfrak{l} : \mathfrak{b}\to \mathfrak{e} $ and invertible $2$-cells $\mathfrak{l}\mathfrak{p}\Rightarrow \id _ \mathfrak{e} $,
$\id _ \mathfrak{b} \Rightarrow \mathfrak{p}\mathfrak{l} $. It is a basic coherence result the fact that, whenever we have such a data, we
can actually get an adjunction $\mathfrak{l} \dashv \mathfrak{p} $ and an adjunction $\mathfrak{p}\dashv \mathfrak{l} $ with invertible units and invertible counits. These adjunctions are called \textit{adjoint equivalences}.

\begin{defi}
Let $\AAAA : \Delta _ \mathrm{Str}\to \AAA $ be a $2$-functor.
We say that the pair 
\begin{equation}
\left( p: e\to b, \uppsi : \AAAA (\dd ^1 ) \cdot p\Rightarrow \AAAA (\dd ^0 ) \cdot p  \right),
\end{equation}
in which $p$ is a morphism and $\uppsi $ is a $2$-cell, is \textit{effective} w.r.t.  $\lim ( \mathfrak{D} , \AAAA ) $ if the following statements hold:
\begin{itemize}
\renewcommand\labelitemi{--}
\item $\AAA $ has the lax descent object $\lim ( \mathfrak{D} , \AAAA ) $;
\item the pair $(p, \uppsi )$ satisfies the descent associativity \eqref{Associativityequationdescent} and identity \eqref{Identityequationdescent} w.r.t. $\AAAA $;
\item the induced factorization $p =\dd ^{(\mathfrak{D} , \AAAA )}\circ p^{( \AAAA , \uppsi ) } $
is such that $p^{( \AAAA , \uppsi ) }$ is an equivalence.
\end{itemize}
\end{defi}

\begin{defi}[effective faithful morphism]\label{effectivemonomorphismdefinition}
Let $p:e\to b $ be a morphism of a $2$-category $\AAA $. The morphism $p$ is an 
\textit{effective faithful morphism}
of $\AAA $ if the following statements hold:
\begin{itemize}
\renewcommand\labelitemi{--}
\item  $\AAA $ has the two-dimensional cokernel diagram of $p$;
\item $\AAA $ has the lax descent object of the 
two-dimensional cokernel diagram  $\H _ p $;
\item  the semantic lax descent factorization \eqref{laxdescentfactorizationofthehighercokernel} of $p$, $p =  \dd ^p \circ p ^\H  $, is such that $p ^\H $
is an equivalence, that is to say, $(p, \upalpha ) $ is effective w.r.t. $\H _p $.
\end{itemize}
\end{defi}
\begin{rem}\label{1dimensionalmotivation}
	The terminology of Definition \ref{effectivemonomorphismdefinition} is motivated by the $1$-dimensional case. In a category with
	suitable pushouts and coequalizers, every morphism $p$ has a factorization 
	induced by the equalizer of the ``cokernel pair'' 
	$$
	\xymatrix{
		b
		\ar@<-0.6ex>[rr]
		\ar@<0.6ex>[rr]
		&&
		b\sqcup _ e b
	}
	$$
	of $p$. If the morphism $p$ is itself the equalizer, $p$ is said to be an  \textit{effective monomorphism}.
\end{rem}

\begin{rem}
The concept of \textit{effective faithful morphism} was already considered in \cite[pag.~142]{MR1314469} under the name \textit{effective descent morphism}. The terminology of \cite[pag.~142]{MR1314469} does not agree with the usual terminology of Grothendieck descent theory within the context of \cite{MR1466540, 2016arXiv160604999L}. Since the present work is intended to be applied to the context of \cite{2016arXiv160604999L}, we adopt the alternative terminology.
\end{rem}

\begin{rem}
If a morphism $p: e\to b $ of a $2$-category $\AAA $ is effective w.r.t. any  
$2$-functor $\AAAA : \Delta _ \mathrm{Str}\to \AAA $, it is clear that $p$ is faithful.
More precisely, in this case, for any object $x$, $\AAA (x, p ): \AAA (x, e)\to \AAA(x,b)$ 
is faithful.
In particular, if $p$ is an effective faithful morphism, then $p$ is faithful. 
\end{rem}

By Remark \ref{smaenotionoffactorization},  the codual of \eqref{laxdescentfactorizationofthehighercokernel}
of a morphism $p$
gives the same factorization, provided that it exists. Hence, we have:

\begin{lem}[Self-coduality]\label{selfcodual}
Let $p $ be a morphism of a $2$-category $\AAA $. The morphism $p$ is an effective faithful morphism in $\AAA $
if and only if the morphism corresponding to $p$ is an effective faithful morphism in $\AAA ^\co $.
\end{lem}

\begin{defi}[Duality: effective op-faithful morphism]
Let $p: e\to b $ be a morphism of a $2$-category $\AAA $. 
The morphism $p$ is an \textit{effective op-faithful morphism} of $\AAA $ 
if the morphism corresponding to $p$ is an effective faithful morphism in $\AAA ^\op $.
\end{defi}

\begin{defi}[Monadicity, comonadicity, Kleisli morphism]\label{Definitionmonadicity}
Let $p: e\to b $ be a morphism of a $2$-category $\AAA $. We say that $p$ is \textit{monadic} if the following
statements hold:
\begin{itemize}
\renewcommand\labelitemi{--}
\item $p$ has a codensity monad $\t = (t, m, \eta ) $;
\item $\AAA $ has the Eilenberg-Moore object of $\t $;
\item the semantic factorization of  $p= \uu ^\t\circ p^\t $ is such that $p^\t $ is an equivalence.
\end{itemize}
Dually, $l: b\to e $ is a \textit{Kleisli morphism} if the corresponding morphism in $\AAA ^\op $ is monadic, while $l$
is \textit{comonadic} if its corresponding morphism in $\AAA ^\co $ is monadic. 
\end{defi}

By Theorem \ref{maintheoremofthepaper} and its dual versions, we get the following characterizations of monadicity, comonadicity and Kleisli morphisms:

\begin{coro}[Monadicity theorem]\label{FirstMonadicity}
Assume that $\AAA $ has the two-dimensional cokernel diagram  of a morphism $p: e\to b $. 
\begin{enumerate}
\item provided that $\ran _p p $ exists and is preserved by $\d ^0 $, $p$ is monadic if and only if $p$ is an effective faithful morphism;
\item provided that $\lan _p p $ exists and is preserved by $\d ^1 $, $p$ is comonadic if and only if $p$ is an effective faithful morphism.
\end{enumerate}
\end{coro}
\begin{proof}
The first result follows immediately from the definitions and from Theorem \ref{maintheoremofthepaper}. The second one is just its codualization
(see Lemma \ref{selfcodual}, Theorem \ref{codualmaintheoremofthepaper}  and Remark \ref{smaenotionoffactorization}). 
\end{proof}

\begin{coro}[Characterization of Kleisli morphisms]\label{FirstKleisli}
Assume that $\AAA $ has the higher kernel of a morphism $l: b\to e $. 
\begin{enumerate}
\item assuming that $\rlift _l l $ exists and is respected by $\d _0^{l\downarrow l} $, $l$ is a Kleisli morphism if and only if $l$ is an effective op-faithful morphism;
\item assuming that $\llift _p p $ exists and is respected by $\d _1^{l\downarrow l} $, $l$ is comonadic if and only if $l$ is an effective op-faithful morphism.
\end{enumerate}
\end{coro}

It is a well known fact that, whenever a morphism is monadic in a $2$-category $\AAA $, it has a left adjoint. In our setting, if $p $ is monadic as in Definition \ref{Definitionmonadicity}, the existence of a left adjoint follows from: (1) 
since $p^\t $ is an equivalence, it has a left adjoint;
(2)  $\uu ^\t $ has always a left adjoint induced by the underlying morphism of the monad $ t: b\to b $, the multiplication $m: t^2\Rightarrow t $
and the universal property of $b ^\t $; and (3) composition of right adjoint morphisms is right adjoint. 
From this fact and Theorem \ref{ADJUNCTIONVERSION}, we get cleaner versions of our monadicity results:

\begin{coro}[Monadicity theorem]\label{MONADICITYCLEAN}
Assume that the $2$-category $\AAA $ has the two-dimensional cokernel diagram  of a morphism $p$. 
\begin{enumerate}
\item The morphism $p$ is monadic if and only if $p$ has a left adjoint and $p$ is an effective faithful morphism.
\item The morphism $p$ is comonadic if and only if $p$ has a right adjoint and $p$ is an effective faithful morphism.
\end{enumerate}
\end{coro}

\begin{coro}[Characterization of Kleisli morphisms]\label{KelisliCLEAN}
Assume that the $2$-category $\AAA $ has the two-dimensional kernel diagram  of a morphism $l$. 
\begin{enumerate}
\item The morphism $l$ is a co-Kleisli morphism if and only if $l$ has a left adjoint and $l$ is an effective op-faithful morphism.
\item The morphism $l$ is Kleisli morphism if and only if $l$ has a right adjoint and $l$ is an effective op-faithful morphism.
\end{enumerate}
\end{coro}

\begin{rem}[Monadicity vs comonadicity]
It should be noted that, unlike Beck's monadicity theorem in $\Cat $, the condition to get monadicity
from a right adjoint morphism coincides with the condition to get comonadicity from a left adjoint morphism; namely, to be an effective faithful morphism.
Of course, as a consequence, we get that, under the conditions of Corollary \ref{MONADICITYCLEAN},
if the morphism $p$ has a left and a right adjoint morphism, the following statements are equivalent:
\begin{enumerate}[label=\roman*)]
\item $p$ is an effective faithful morphism;
\item $p$ is monadic;
\item $p$ is comonadic.
\end{enumerate} 
\end{rem}

\begin{rem}[Beck's monadicity theorem vs formal monadicity theorem]\label{whyitisnotthesamething}
Beck's monadicity theorem states that, in $\Cat $, a functor $p$ is monadic if and only if $p$ has a left adjoint and $p$ creates absolute coequalizers.
By our monadicity theorem, we can conclude that,\textit{ provided that a functor $p:e\to b $ has a left adjoint, $p$ creates absolute coequalizers if and only if 
$p$ is an effective faithful morphism in $\Cat $}. 

However, the effective faithful morphisms in $\Cat $ are not characterized by the property of creation of
absolute coequalizers. For instance, this follows from the fact that, by Lemma \ref{selfcodual}, the concept of effective faithful morphism
is self codual, while the property of creation of absolute coequalizers is not self dual. 

More precisely, one of the  key aspects of duality in $1$-dimensional category theory 
is that the usual $2$-functor $\op $ given by
\begin{eqnarray*}
\op : & \Cat ^\co & \to \Cat  \\
& e & \mapsto e^\op\\
& p: e\to b &\mapsto p ^\op : e^\op\to b^\op \\
&\beta & \mapsto \beta ^\op
\end{eqnarray*}
is invertible. Therefore a functor $p ^\op : e ^\op\to b ^\op $ is an effective faithful morphism in $\Cat $ if and only if 
the morphism $p: e\to b $ is an effective faithful morphism in $\Cat ^\co $. Moreover, by Lemma \ref{selfcodual}, the morphism $p$  is
an effective faithful morphism in $\Cat ^\co $ if and only if the corresponding morphism (functor) $p$
 is an effective faithful morphism in $\Cat $.  Hence, by abuse of notation, $p$ is an effective faithful morphism in $\Cat $ if and only if
$p ^\op $ is an effective faithful morphism in $\Cat $. 
 
It is clear that a functor $p: e\to b $ creates absolute coequalizers if and only if the functor corresponding to $p^\op : e^\op\to b ^\op $ 
creates absolute equalizers. Since there are functors that create absolute coequalizers but do not create absolute equalizers, the property
of creation of absolute equalizers is not self dual. It follows, then,
that there are functors that do create absolute coequalizers but are not effective faithful morphisms.

For instance, consider the usual forgetful functor between the category of free groups and the category of sets.
This functor reflects isomorphisms and preserves  equalizers: hence, since
the category of free groups has equalizers, this forgetful functor creates all equalizers. However, since it has a left adjoint, it does not create absolute
coequalizers and it is not an effective faithful morphism in $\Cat $ (otherwise, it would be monadic). Therefore the image of the morphism corresponding to this
functor in $\Cat  ^\co $ by $\op $ is a functor that creates absolute coequalizers
but it is not an effective faithful morphism in $\Cat $.
\end{rem}

\begin{rem}[Characterization of effective faithful morphisms]\label{characterizationsuggestion}
	In \cite{MR1314469}, Zawadowski gave the characterizations of effective faithful morphisms in some special $2$-categories. Moreover, as a consequence of proof of \cite[Proposition~3.1]{MR2107402}, the effective op-faithful morphisms  in $\Cat $ are
	precisely the functors that are essentially surjective on objects.
\end{rem}

\section*{Acknowledgments}
I am grateful to the \textit{Coimbra Category Theory Group}  for their unfailing support. I am also thankful to
the \textit{Software Technology Group} at Utrecht University 
for their welcoming and supportive environment. 

This work was realized during my postdoctoral fellowship at the Centre for Mathematics, University of Coimbra, in 2018.  My ideas and results for this project were positively influenced by the peaceful  and inspiring atmosphere I found there. 

I want to give special thanks to Marino Gran and Tim Van der Linden, who warmly hosted me at \textit{Universit\'{e} catholique de Louvain} in May 2018, 
where my ideas and work for this project started. In my stay there, I found a very encouraging and productive environment.

Maria Manuel Clementino made my stay at \textit{UCLouvain} even more memorable, allowing me to work with her on exciting problems about Janelidze-Galois Theory.
My discussions on examples of codensity monads with her and Eduardo Dubuc, whom I also wish to thank, were instrumental in coming up with counterexamples to some tempting but wrong strengthening of some of the present paper's results.

After publishing this work, I have received several questions  and suggestions for future work, terminology and bibliography. In these directions, I am specially thankful to George Janelidze,  Zurab Janelidze and Steve Lack.

I want to thank Eduardo Ochs for taking his valuable time to teach me about his \texttt{Dednat6}~\LaTeX~package. During the revision of the present paper, \texttt{Dednat6} proved very efficient to write and  edit diagrams quickly.

\bibliographystyle{plain}
\bibliography{references}

\pu

\end{document}
